\documentclass[a4paper, 10pt, reqno]{amsart}

\usepackage[margin=1in]{geometry}  
\usepackage{graphicx}               
\usepackage{amsmath}              
\usepackage{amsfonts}              
\usepackage{bm}                        
\usepackage{amsthm}                
\usepackage{mathtools}             
\usepackage{esint}                     
\usepackage{color}
\usepackage{amssymb}             
\usepackage[colorlinks]{hyperref}
\usepackage[noabbrev]{cleveref}
\usepackage{tikz}                       
\usetikzlibrary{patterns}              
\usepackage[english]{babel}
\usepackage[utf8]{inputenc}
\usepackage[T1]{fontenc}
\usepackage{soul}                     
\usepackage{enumitem}
\usepackage[foot]{amsaddr}

\newtheorem{thm}{Theorem}[section]
\newtheorem*{thm*}{Theorem}
\newtheorem{lem}[thm]{Lemma}
\newtheorem{prop}[thm]{Proposition}
\newtheorem{cor}[thm]{Corollary}

\newtheorem{rmk}[thm]{Remark}
\newtheorem{definition}[thm]{Definition}
\newtheorem{example}[thm]{Example}

\DeclareMathOperator{\dist}{dist}

\DeclareMathOperator{\cof}{cof}
\DeclareMathOperator{\supp}{supp}

\newcommand{\RR}{\mathbb{R}}     
\newcommand{\NN}{\mathbb{N}}     

\newcommand{\E}{\mathcal{E}}
  
\newcommand{\J}{\mathcal{J}}  

\newcommand{\e}{\varepsilon}

\newcommand{\weakstarto}{\overset{\ast}{\rightharpoonup}}

\newcommand{\abs}[1]{\left| #1 \right|}
\newcommand{\norm}[2][]{\left\| #2 \right\|_{#1}}

\newcommand{\Adm}[2]{T_{#1}(\mathcal{E} \ifx&#2& \else \cap #2 \fi)}

\hypersetup{colorlinks = true, linkcolor = blue, citecolor = red}

\makeatletter
\@namedef{subjclassname@2020}{\textup{2020} Mathematics Subject Classification}
\makeatother
         
\numberwithin{equation}{section} 

\begin{document}


\title[Inertial evolution in the face of (self-)collision]{Inertial evolution of non-linear viscoelastic solids in the face of (self-)collision} 
\author[] {Anton\'in \v{C}e\v{s}\'ik$^\dag$}
\address{$^\dag$Department of Mathematical Analysis\\Faculty of Mathematics and Physics\\ 
Charles University\\Prague, Czech Republic}
\email {cesik@karlin.mff.cuni.cz}
\author[] {Giovanni Gravina$^{\ddag, *}$} 
\address{$^\ddag$School of Mathematics and Statistical Sciences\\ Arizona State University \\Tempe, AZ 85281}
\email {ggravina@asu.edu} 
\author[] {Malte Kampschulte$^{\dag}$}
\email {kampschulte@karlin.mff.cuni.cz} 
\keywords{Non-interpenetration of matter, variational inequalities, non-simple materials}
\subjclass[2020]{35Q74, 49J40, 74A30, 74B20, 74M15}
\thanks{\emph{Author ORCIDs.} Anton\'in \v{C}e\v{s}\'ik: \href{https://orcid.org/0000-0003-0207-6044}{https://orcid.org/0000-0003-0207-6044}, Giovanni Gravina: \href{https://orcid.org/0000-0001-8985-7964}{https://orcid.org/0000-0001-8985-7964}, Malte Kampschulte: \href{https://orcid.org/0000-0002-9150-4460}{https://orcid.org/0000-0002-9150-4460}.}
\date{\today}

\begin{abstract}
 We study the time evolution of non-linear viscoelastic solids in the presence of inertia and (self-)contact. For this problem we prove the existence of weak solutions for arbitrary times and initial data, thereby solving an open problem in the field. Our construction directly includes the physically correct, measure-valued contact forces and thus obeys conservation of momentum and an energy balance. In particular, we prove an independently useful compactness result for contact forces. 
\end{abstract}

\maketitle

 \section{Introduction}
 
 The study of contact and dynamic collisions between (elastic) bodies has a long history, starting from classical, 18th-century, physical considerations about conservation of energy and momentum and ranging into modern continuum mechanics. However in this, and in particular in the mathematical treatment of the latter, until now, there has always been a divide between more phenomenological and \emph{ab initio} approaches. 
 
 This divide is a direct consequence of the difficulty and irreducibility of the full problem. If one is not able to fully treat dynamic contact between deformable elastic bodies, then one has to simplify the problem in one of several directions. 
 
 The first is to remove the ability of the bodies to deform, treating them as rigid bodies. However, when this is done, the problem immediately becomes ill-posed and needs to be supplemented by a phenomenological contact law. The second is to soften the contact itself, replacing the hard dichotomy of ``in contact'' vs.\ ``not in contact'' with a soft repulsion potential. Yet this also introduces an indirect contact law. The third possibility is to remove the dynamic aspects and focus on the static or quasistatic situation instead.

 Additionally, even in this last case, there is a difference in difficulty between the collision of an elastic solid with a static obstacle and the collision of two elastic solids with each other or even an elastic solid deforming so far as to collide with itself. In this work for the first time, we prove the existence of weak solutions to this general case involving inertia and large deformations.
 
 Specifically, consider one or more elastic solids, given in Lagrangian coordinates by a reference configuration $Q \subset \RR^n$ (with multiple solids represented by multiple connected components of $Q$), as well as a deformation $\eta \colon Q \to \Omega \subset \RR^n$ to describe the current configuration. To each deformation we attach an elastic energy $E(\eta)$ and to each change in configuration a dissipation potential $R(\eta,\partial_t \eta)$, which for physical reasons has to also depend on the deformation itself \cite{antmanPhysicallyUnacceptableViscous1998}.
 
 If the solid has density $\rho$ in the reference configuration, then by Newton's second law we expect the solid to evolve according to
 \begin{align*}
  \rho \partial_{tt} \eta + DE(\eta) + D_2R(\eta,\partial_t \eta) = f,
 \end{align*}
 where $DE$ and $D_2R$ denote the formal Fr\'echet derivatives with respect to $\eta$ and $\partial_t \eta$ respectively, and $f$ is an external force.

 Of course this only works in the absence of collisions. As alluded to before, to deal with collisions, there needs to be another modelling assumption, which in turn needs to be justified. We choose here to adopt the absolute \emph{minimal assumption}, which is also the only assumption that we can be sure holds universally, namely \emph{non-interpenetration of matter}. Translated into the mathematical framework this means that the only additional information we assume in modelling is that $\eta$ is injective, except possibly for a set of measure zero.\footnote{In the absence of rigid bodies and point-masses this turns out to also be a sufficient assumption. For more details see the discussion in Section \ref{sec:physConsiderations}.}
 
 Such a restriction of the set of admissible deformations by its very nature results in a Lagrange-multiplier, which we can readily identify as the contact force $\sigma$. This force has to be supported on the contact set and because of considerations involving conservation of energy and momentum, it has to be of equal magnitude and opposite direction at each pair of points where solids touch. Additionally, if we do not assume any friction, it has to point in the same direction as the respective interior normal of the physical configuration.
 
 With this, we can summarize the system under study as
 \begin{align} \label{eq:strongSolution}   \begin{cases}
   \rho \partial_{tt} \eta + DE(\eta) + D_2R(\eta,\partial_t \eta) = \sigma + f, \\
   \sigma = |\sigma| n_\eta,\,\, \supp(\sigma) \subset C_\eta, 
  \end{cases}
 \end{align}
 where $C_\eta$ is the set of (self-)contact points and $n_\eta$ denotes the interior normal of the deformed configuration. Additionally, we require an ``equal and opposite'' assumption on the contact force, which is easier to write in weak form (see \Cref{def-contact-force} below). We can then give an abridged version of out main result as
 \begin{thm*}[\Cref{main-thm-VI} (abridged)]
  Under some (physical) assumptions on $E$ and $R$, for any sufficiently reasonable initial data (injective a.e.\@ and of finite energy) and any time $T> 0$, the system \eqref{eq:strongSolution} has a weak solution in the interval $[0,T]$. This solution obeys conservation of momentum and the physical energy inequality.
 \end{thm*}

 We mention here that in the full result we also treat the case in which the evolution happens in the presence of rigid, immovable obstacles.
 
 The proof of the theorem relies mainly on variational methods. On one hand, this is not surprising, as more classical PDE methods (e.g., fixed-points or Galerkin approaches) are often inadequate to handle the difficulties that arise from the fact that the non-interpenetration constraint results in a non-convex, non-linear state space. On the other hand, so far, variational methods have generally been restricted to static and quasistatic systems. Indeed, for the problem at hand, we build quite explicitly on the work \cite{PalmerHealey} by Palmer and Healey, where the authors study self-contact for the static case. Their results have been recently extended to the quasistatic case by Kr\"omer and Roub\'{i}\v{c}ek \cite{kromerQuasistaticViscoelasticitySelfcontact2019}.\footnote{As a necessary step towards the proof, we also improve their result to give a quantization of the contact force as a measure, which in \cite{kromerQuasistaticViscoelasticitySelfcontact2019} was only characterized as part of a distribution. We thus in fact solve both of their open problems (See \Cref{rem:KRopenProblem1}). In particular, we believe that the more detailed treatment of convergence of contact forces used for this might be of independent interest.}

 The crucial ingredient in our proofs is the method of using two time-scales developed in \cite{benesovaVariationalApproachHyperbolic2020}, which shows a way to lift almost any quasistatic weak existence result to a corresponding weak existence result for the associated inertial problem. 
 
 For the details of this approach, we also refer to \cite[Sec.\@ 3]{benesovaVariationalApproachHyperbolic2020}, where the corresponding result without collisions is shown (see also the introductory sections of \cite{bks2021poroelastic}, where an attempt is made to elaborate on some of the more general underpinnings of this method). However, we aim to keep the use of this method in the current paper self-contained.
 
 Finally we note that throughout the paper we will restrict our attention to generalized second order materials (compare in particular with \cite{MR2567249} and the references therein), i.e.\@, we assume that the elastic energy will depend on the second derivative $\nabla^2 \eta$ of the deformation. While at first glance this might seem like a departure from the classical theory of elasticity, we note that for the study of (self-)contact such a restriction is necessary. Indeed, not only is it needed to avoid issues arising from points where the Jacobian $\det \nabla \eta$ vanishes, but also because without the resulting $C^1$-regularity that the theory implies, microscopic oscillations of the exterior normal can lead to boundary microstructure and produce artificial friction (see \Cref{rem:lowerRegularityFriction}), all of which is outside of the scope of the current paper.
 
 \subsection{Structure of the paper} In \Cref{Modelling-sec}, we will give a precise definition of the assumptions we require in terms of energy and the dissipation, as well as a precise statement of the main theorem and some possible extensions and corollaries. 
 Next, in \Cref{Contact-sec}, we will go on and derive some of the properties related to contact forces as well as their convergence behavior, which might be of independent interest.
 The bulk of the paper will then be devoted to the proof of the main theorem, first by showing an auxiliary quasistatic result in \Cref{Aux-sec} and then by using this to generate an approximating sequence of time-delayed solutions to the actual equation in \Cref{inertia-sec}.

 Finally, in \Cref{sec:physConsiderations} we then relate this result to the actual physics we are aiming to describe by giving an example energy-dissipation pair that satisfies the assumptions and by discussing how the result is connected to momentum conservation. We also show by an example that the condition of non-interpenetration of matter that we use is indeed sufficient and there are no phenomenological contact laws needed to arrive at the correct behavior.

 \subsection{Notation}
Throughout this paper we use the following notation, unless stated otherwise.

Deformations will be denoted by $\eta\colon Q \to \Omega$ or $\eta\colon [0,T] \times Q \to \Omega$ in the case of time-dependence, where $Q\subset \RR^n$ is a reference configuration, and $\Omega \subset \RR^n$ a possibly unbounded containing domain. 

Points in $Q$ are notated as $x$. Functions which depend on space and time are always notated time first e.g. $\eta(t,x)$. If we want to consider deformations for a fixed time, we will also use $\eta(t) \coloneqq \eta(t,\cdot)$ to ease notation.
 
 For any $x\in \partial Q$, the vector $n_Q(x)$ denotes the interior unit normal to $\partial Q$ at $x$. Given additionally a sufficiently regular deformation $\eta \colon Q \to \Omega$, we will use 
\[
n_\eta(x) \coloneqq \frac{\cof \nabla \eta(x) n_Q(x)}{\abs{\cof \nabla \eta(x) n_Q(x)}}
\]
to denote the interior unit normal of $\eta(Q)$ at $\eta(x)$. In case $\eta$ is also time-dependent we use $n_\eta(t,x)$ in place of $n_{\eta(t)}(x)$. Finally for $y\in \partial \Omega$, the vector $n_\Omega(y)$ denotes the \emph{exterior}\footnote{To reduce the distinction between cases, it is best to not think of $\Omega$ as the domain, but of $\RR^n \setminus \Omega$ as a fixed, rigid obstacle. Thus $n_\Omega$ is the interior normal of that obstacle, in the same way $n_\eta$ is the interior normal of the movable solids.} unit normal at $y$. 

We use the usual notations $W^{k,p}(Q;\RR^n)$ for Sobolev spaces and $L^q((0,T); W^{k,p}(Q;\RR^n))$ for the respective Bochner spaces. A subscript $\Gamma$ is used to denote spaces of functions whose trace vanishes on that set, e.g.\@ $W^{k,p}_\Gamma(Q) \coloneqq \{u \in W^{k,p}(Q): u|_\Gamma = 0\}$. Additionally, we denote by $M(K;\RR^n)$ the space of $\RR^n$-valued Radon measures and by $M^+(K)$ the set of non-negative Radon measures on a compact set $K$. For a measure $\sigma\in M(K;\RR^n)$ and $\varphi\in C(K;\RR^n)$ we denote 
\[
\langle \sigma , \varphi \rangle \coloneqq \int_{K} \varphi\cdot d\sigma.
\]
 
 Spaces are written out in full (e.g., $L^2((0,T); W^{1,2}(Q;\RR^n))$), but when writing norms, we usually omit both the domain and the image if there is no chance of confusion (e.g., we write $\norm[L^2]{f}$ instead of $\norm[L^2(Q;\RR^n)]{f}$). Additionally, when dealing with linear operators on Sobolev spaces, we write the linear argument in angled brackets, e.g.
 \[A\langle u \rangle \coloneqq \langle A, u\rangle_{(W^{k,p})^* \times W^{k,p}}\]
 for $A \colon W^{k,p}(Q) \to \RR$ and $u \in W^{k,p}(Q)$.
 
 \subsection{Acknowledgements}
The authors acknowledge the support of the the Primus research programme of Charles University under grant No.\@ PRIMUS/19/SCI/01. The research of A.\v{C}.\@ and M.K.\@ was partly funded by the ERC-CZ grant LL2105. A.\v{C}.\@ further acknowledges the support of Charles University, project GA UK No.\@ 393421. The work of G.G.\@ and M.K.\@ was partially supported by the Charles University research program No.\@ UNCE/SCI/023 and by the Czech Science Foundation (GA\v{C}R) under grant No.\@ GJ19-11707Y.

\subsection{Research data policy and data availability statements}
Data sharing not applicable to this article as no datasets were generated or analyzed during the current study.

 \section{Modelling of viscoelastic materials and statement of the main results}
 \label{Modelling-sec}
 
 \subsection{Viscoelastic solids}
The time evolution of a viscoelastic solid body in $\RR^n$ can be described in Lagrangian coordinates by a (time dependent) deformation of a reference configuration $Q \subset \RR^n$, which in the following we typically denote by $\eta \colon [0, T] \times Q \to \RR^n$. The set $Q \subset \RR^n$ will be a $C^{1,\alpha}$-smooth, bounded domain, or alternatively a disjoint union of finitely many of such domains in order to describe multiple bodies. We assume that the movement of the solid is confined to the set $\Omega \subset \RR^n$ which is a $C^{1,\alpha}$-smooth domain, but possibly unbounded (e.g.\ $\Omega = \RR^n$, half-space, etc.). Furthermore, in order to rule out non-physical phenomena such as self-interpenetration, we restrict our attention to deformations that are almost everywhere globally injective and orientation preserving. These assumptions are encoded in the class of admissible deformations, which we define using the Ciarlet--Ne\v{c}as condition \cite{CN87} as
\begin{equation}
\label{domain-def}
\E \coloneqq \left\{\eta \in W^{2, p}(Q; \RR^n) : \eta(Q) \subset \Omega, \eta|_\Gamma = \eta_0, \det \nabla \eta > 0, \text{ and } \mathcal{L}^n(\eta(Q)) = \int_{Q}\det \nabla \eta(x)\,dx \right\}.
\end{equation}
Here we use $\eta_0$ to denote a given admissible (initial) deformation and let $\Gamma$ be a (fixed) measurable subset of $\partial Q$. Note, however, that for the main result of this paper we do not assume that $\mathcal{H}^{n - 1}(\Gamma) > 0$ and refer the reader to \Cref{partial-Dir} for more information. Here and in the following we assume that $p > n$. In particular, this implies that every $\eta \in \E$ admits a representative of class $C^{1, 1 - \frac{n}{p}}$. Throughout the rest of the paper we identify $\eta$ with this regular representative without further notice. Additionally, we assume a constant Lagrangian density $\rho \in (0,\infty)$, but remark that all our arguments also work for variable densities, as long as these are bounded from above and away from zero.

Next, we specify the assumptions on the energy-dissipation pair $(E, R)$. To be precise, we assume that the elastic energy $E \colon W^{2,p}(Q; \RR^n)\to (-\infty, \infty]$ has the following properties: 
\begin{enumerate}[label=(E.\arabic*), ref=E.\arabic*]
\item \label{E1} There exists $E_{\min} > - \infty$ such that $E(\eta) \ge E_{\min}$ for all $\eta \in W^{2, p}(Q; \RR^n)$. Moreover, $E(\eta) < \infty$ for every $\eta \in W^{2, p}(Q; \RR^n)$ with $\inf_Q\det \nabla \eta > 0$.
\item \label{E2} For every $E_0  \ge E_{\min}$ there exists $\e_0 > 0$ such that $\det \nabla \eta \ge \e_0$ for all $\eta$ with $E(\eta) \le E_0$.
\item \label{E3} For every $E_0 \ge E_{\min}$ there exists a constant $C$ such that 
\[
\|\nabla^2 \eta\|_{L^p} \le C 
\]
for all $\eta$ with $E(\eta) < E_0$.
\item \label{E4} $E$ is weakly lower semicontinuous, that is, 
\[
E(\eta) \le \liminf_{k \to \infty} E(\eta_k)
\]
whenever $\eta_{k} \rightharpoonup \eta$ in $W^{2, p}(Q; \RR^n)$. Moreover, $E$ is continuous with respect to strong convergence in $W^{2, p}(Q; \RR^n)$.

\item \label{E5} $E$ is differentiable in its effective domain with derivative $DE(\eta) \in (W^{2, p}(Q; \RR^n))^*$ given by 
\[
 DE(\eta) \langle \varphi \rangle = \left.\frac{d}{d\e} E(\eta + \e \varphi))\right|_{\mathrlap{\e = 0}}. 
\]
Furthermore, $DE$ is bounded on any sub-level set of $E$ and $DE(\eta_k) \langle \varphi \rangle \to DE(\eta)\langle \varphi \rangle$ whenever $\eta_k \to \eta$ in $W^{2, p}(K; \RR^n)$ for all $K$ compactly contained in $\overline{Q}$ with $\dist(K, \Gamma) > 0$ and $\varphi \in W^{2, p}_{\Gamma}(Q; \RR^n)$.

\item \label{E6} $DE$ satisfies 
\[
\liminf_{k \to \infty}  (DE(\eta_k) - DE(\eta)) \langle (\eta_k - \eta)\psi \rangle \ge 0
\]
for all for $\psi \in C^{\infty}_{\Gamma}(Q; [0, 1])$ and all sequences $\eta_k \rightharpoonup \eta$ in $W^{2, p}(Q; \RR^n)$. In addition, $DE$ satisfies the following Minty-type property: If 
\[
\limsup_{k \to \infty} (DE(\eta_k) - DE(\eta))\langle (\eta_k - \eta)\psi \rangle \le 0
\]
for all $\psi \in C^{\infty}_{\Gamma}(Q; [0, 1])$, then $\eta_k \to \eta$ in $W^{2, p}(K; \RR^n)$ for all $K$ compactly contained in $\overline{Q}$ with $\dist(K, \Gamma) > 0$.
\end{enumerate}

Additionally, we assume that the dissipation potential $R \colon W^{2, p}(Q; \RR^n) \times W^{1, 2}(Q; \RR^n) \to [0, \infty)$ satisfies the following properties:

\begin{enumerate}[label=(R.\arabic*), ref=R.\arabic*]
\item \label{R1} $R$ is weakly lower semicontinuous in its second argument, that is, for all $\eta \in W^{2, p}(Q; \RR^n)$ and every $b_{k} \rightharpoonup b$ in $W^{1, 2}(Q; \RR^n)$ we have that
\[
R(\eta, b) \le \liminf_{k \to \infty} R(\eta, b_k)
\]
\item \label{R2} $R$ is homogeneous of degree $2$ with respect to its second argument, that is, 
\[
R(\eta, \lambda b) = \lambda^2 R(\eta, b)
\]
for all $\lambda \in \RR$.
\item \label{R3} $R$ admits the following Korn-type inequality: For any $\e_0 > 0$, there exists $K_R$ such that
\[
K_R \|b\|_{W^{1,2}}^2 \le \|b\|_{L^2}^2 +  R(\eta, b)
\]
for all $\eta \in \mathcal{E}$ with $\det \nabla \eta > \e_0$ and all $b \in W^{1, 2}(Q; \RR^n)$. 
\item \label{R4} $R$ is differentiable in its second argument, with derivative $D_2R(\eta, b) \in (W^{1, 2}(Q; \RR^n))^*$ given by
\[
 D_2R(\eta, b)\langle \varphi \rangle \coloneqq \left.\frac{d}{d\e} R(\eta, b + \e \varphi)\right|_{\mathrlap{\e = 0}}. 
\]
Furthermore, the map $(\eta, b) \mapsto D_2R(\eta, b)$ is bounded and weakly continuous with respect to both arguments, that is, 
\[
\lim_{k \to \infty}  D_2R(\eta_k, b_k) \langle \varphi \rangle =  D_2R(\eta, b) \langle \varphi \rangle
\]
holds for all $\varphi \in W^{1, 2}(Q; \RR^n)$ and all sequences $\eta_k \rightharpoonup \eta$ in $W^{2, p}(Q; \RR^n)$ and $b_k \rightharpoonup b$ in $W^{1, 2}(Q; \RR^n)$.
\end{enumerate}
We also introduce a variant of \eqref{R3} that will be used for the quasistatic case in the form of
\begin{enumerate}[label=(R.3\textsubscript{q}), ref=R.3\textsubscript{q}]
 \item \label{R3q} $R$ admits the following Korn-type inequality: For any $\e_0 > 0$, there exists $K_R$ such that
\[
K_R \|b\|_{W^{1,2}}^2 \le  R(\eta, b)
\]
for all $\eta \in \mathcal{E}$ with $\det \nabla \eta > \e_0$ and all $b \in W^{1, 2}_{\Gamma}(Q; \RR^n)$.
\end{enumerate}

We mention here that the assumptions on the energy-dissipation pair are standard within the framework of second-order viscoelastic materials (see in particular \cite{MR2567249}, \cite{kromerQuasistaticViscoelasticitySelfcontact2019}, and the references therein). For the convenience of the reader, explicit examples of $E$ and $R$ that satisfy the assumptions above are given in \Cref{sec:physConsiderations}.

\begin{rmk}\label{DRgrowth}
 Note that in particular \eqref{R2} and \eqref{R4} allow us to derive some additional growth conditions on the dissipation and its derivative. First of all, by taking the derivative of the identity in \eqref{R2} with respect to $b$ and dividing by $\lambda$, we get
 \begin{equation} \label{eq:R1homog}
  D_2R(\eta,\lambda b) = \lambda D_2R(\eta,b)
 \end{equation}
 i.e., $D_2R$ is homogeneous of degree $1$ with respect to its second argument. 
 
 Furthermore we can prove that for any $E_0 > E_{\min}$ there exists a constant $C$ such that
 \begin{align} \label{eq:D2Rgrowthbounds}
  \norm[(W^{1, 2})^*]{D_2R(\eta,b)} &\leq C \norm[W^{1,2}]{b},   \\ \label{eq:Rgrowthbounds}
  2R(\eta,b) &\leq C \norm[W^{1,2}]{b}^2 \end{align}
 for all $b \in W^{1,2}(Q;\RR^n)$ and all $\eta \in W^{2,p}(Q;\RR^n)$ with $E(\eta) \leq E_0$. To see this, assume that \eqref{eq:D2Rgrowthbounds} is not true. Then there exist sequences $\{\eta_k\}_{k\in\NN}$ and $\{b_k\}_{k\in\NN}$ with $\norm[(W^{1, 2})^*]{D_2R(\eta_k,b_k)} > k \norm[W^{1,2}]{b_k}$ and $E(\eta_k) \leq E_0$. Additionally, due to the $1$-homogeneity of $D_2R$ (see \eqref{eq:R1homog} above), we can assume without loss that $\norm[W^{1,2}]{b_k}= 1$. This allows us to use\eqref{E3} to pick weakly converging subsequences (which we do not relabel) and respective limits $\eta$ and $b$. But then on one hand, by \eqref{R4}, $D_2R(\eta_k,b_k)$ converges and thus $\norm[(W^{1, 2})^*]{D_2R(\eta_k,b_k)}$ needs to stay bounded, and on the other hand, by our assumption it is larger than $k$, which is a contradiction. This proves \eqref{eq:D2Rgrowthbounds}. Finally \eqref{eq:Rgrowthbounds} follows from \eqref{eq:D2Rgrowthbounds} by noting that due to \eqref{R2} we have
 \begin{align*}
  2R(\eta,b) = D_2R(\eta,b)\langle b\rangle \leq \norm[(W^{1, 2})^*]{D_2R(\eta,b)} \norm[W^{1,2}]{b} \leq C \norm[W^{1,2}]{b}^2.
 \end{align*}
\end{rmk}

\begin{rmk}
\label{partial-Dir}
The difference between \eqref{R3} and \eqref{R3q} is subtle but central to the difference between quasistatic and inertial evolutions. The reasons for this do not only become evident in the proof, but also have a physical explanation. Indeed, in contrast to the full inertial problem, in the quasistatic regime there is no automatic conservation of linear or rotational momentum. As a result, when considering physical dissipations such as $R(\eta,b) = \int_Q | \nabla \eta^T \nabla b + \nabla b^T \nabla \eta|^2\, dx$, which are invariant under Galilean transformations, we need to include additional restrictions to the admissible deformations in $\mathcal{E}$, such as (partial) Dirichlet boundary data or a fixed center of mass and rotation around it.
\end{rmk}

\begin{rmk}[On Dirichlet boundary data and contact] \label{rmk:DirichletBoundary}
 As it is a common occurrence in various applications, we incorporated the potential for Dirichlet boundary data into our formulation. The handling of these boundary conditions during the evolution is mostly standard; however, some subtleties arise when it comes to contact. When a freely moving part of the body comes into contact with a section of the solid where the Dirichlet boundary condition is specified, the latter behaves like a fixed obstacle. Since we are able to deal with fixed obstacles, this situation does not present any issues. 
 
 What can potentially be more problematic, however, is the transition between the fixed and the free part of the boundary. As we require to cut off test functions in proximity of the fixed part of the boundary, we inevitably lose control over the resulting contact force. Notice that as long as this only happens to one of the sides that comes into contact, this is not an issue. Indeed, there is a corresponding opposite and equal force on the other side, which carries the same information that was lost due to the cut-off. In particular, we only run into issues if contact happens on both sides at points or regions where the fixed portion of the boundary transitions into the freely moving boundary.
 
 To avoid this situation and keep the mathematical details manageable, we require $\eta_0|_\Gamma$ to be injective and that $\eta_0(\Gamma) \cap \partial \Omega = \emptyset$. We note however that more general situations can be handled with some additional care as well.
\end{rmk}

\subsection{Statement of the main results} \label{main-sec}
The precise definition of (weak) solution to the initial value problem considered in this paper can be formulated as follows. 
\begin{definition}
\label{var-in-def}
Let $T>0$, $\eta_0 \in \E$, $\eta^* \in L^2(Q; \RR^n)$, and $f \in L^2((0, T); L^2(Q; \RR^n))$ be given. We say that 
\[
\eta \in L^{\infty}((0,T); \E) \cap W^{1,2}((0, T); W^{1, 2}(Q; \RR^n))
\]
with $E(\eta)\in L^\infty((0,T))$ is a weak solution to \eqref{eq:strongSolution} in $(0, T)$ if $\eta(0) = \eta_0$ and the variational inequality
\begin{multline}
\label{var-ineq-def}
\int_0^T DE(\eta(t))\langle  \varphi(t) \rangle + D_2R(\eta(t), \partial_t \eta(t))\langle  \varphi(t)\rangle\,dt \\ - \rho \langle \eta^*, \varphi(0)\rangle_{L^2} - \int_0^T \rho \langle \partial_t \eta(t), \partial_t\varphi(t) \rangle_{L^2} \,dt \ge \int_0^T \langle f(t), \varphi(t) \rangle_{L^2}\,dt
\end{multline}
holds for all $\varphi \in C([0, T]; \Adm{\eta}{}) \cap C^1_c([0, T); L^2(Q; \RR^n))$. Here the set $\Adm{\eta}{}$ denotes the class of admissible perturbations for the deformation $\eta$; its precise definition is given below in \Cref{def-admissibleTestFunctions-timedep}.

Furthermore, we say that this $\eta$ is a weak solution with a contact force if additionally it satisfies
\begin{multline*}
\int_0^T DE(\eta(t))\langle  \varphi(t) \rangle + D_2R(\eta(t), \partial_t \eta(t))\langle  \varphi(t)\rangle\,dt \\ - \rho \langle \eta^*, \varphi(0)\rangle_{L^2} - \int_0^T \rho \langle \partial_t \eta(t), \partial_t\varphi(t) \rangle_{L^2} \,dt = \int_{[0,T]\times \partial Q} \varphi(t,x)\cdot d \sigma(t,x) +  \int_0^T \langle f(t), \varphi(t) \rangle_{L^2}\,dt
\end{multline*}
 for all $\varphi \in C([0, T]; W^{2,p}_\Gamma(Q;\RR^n)) \cap C^1_c([0, T); L^2(Q; \RR^n))$, where $\sigma \in M([0,T] \times\partial Q; \RR^n)$ is a contact force satisfying the action-reaction principle in the sense of \Cref{def-contact-force}.
\end{definition}

Observe that in view of its regularity, $\eta$ belongs to the space $C_w([0,T];W^{2,p}(Q;\RR^n))$. Therefore, we have $\eta(t) \in W^{2,p}(Q;\RR^n)$ for all $t \in [0,T]$, and in particular the initial condition $\eta(0) = \eta_0$ holds in the classical sense. 

With this in hand, we can state the main result of this paper.

\begin{thm}
\label{main-thm-VI}
Let $E$ and $R$ be as in \eqref{E1}--\eqref{E6} and \eqref{R1}--\eqref{R4}, respectively, and let $T> 0$, $\eta_0 \in \E$, $\eta^* \in L^2(Q; \RR^n)$, and $f \in L^2((0,T); L^2(Q; \RR^n))$ be given. Then \eqref{eq:strongSolution} admits a weak solution with a contact force in $(0,T)$ in the sense of \Cref{var-in-def}, where the resulting contact force $\sigma$ has no concentrations in time. Additionally this solution satisfies the energy inequality
\begin{align*}
 E(\eta(t)) + \frac{\rho}{2} \norm[L^2]{\partial_t \eta(t)}^2 + \int_0^t 2R(\eta(s),\partial_t\eta(s)) \,ds \leq E(\eta_0) + \frac{\rho}{2} \norm[L^2]{\eta^*}^2 + \int_0^t \langle f(s),\partial_t \eta(s)\rangle_{L^2}\, ds
\end{align*}
for almost all $t\in [0,T]$.
\end{thm} 

Let us now give a brief description of the method used in the proof. Following the approach developed in \cite{benesovaVariationalApproachHyperbolic2020}, our goal will be to approximate solutions to the original problem with solutions to suitably defined initial value problems for equations of first order, gradient flow type. Thus, we begin by considering a version of \eqref{eq:strongSolution} where the inertial term is replaced by a time discretization. To be precise, for a fixed $h > 0$, we consider the problem
\begin{equation}
\label{time-delayed}
\rho\frac{\partial_t \eta^{(h)}(t) - \partial_t \eta^{(h)}(t - h)}{h} + DE(\eta^{(h)}(t)) + D_2R(\eta^{(h)}(t), \partial_t\eta^{(h)}(t)) = f(t)+\sigma^{(h)}(t),
\end{equation}
complemented by an initial condition in the form of $\eta^{(h)}(0) = \eta_0$. We begin by finding a solution to \eqref{time-delayed} in the interval $[0, h]$, under the assumption that $\partial_t \eta^{(h)}$ is known for $t < 0$ (to be precise, one can assume that $\partial_t \eta^{(h)} = \eta^*$ on $[- h, 0]$). This way, the term $\partial_t \eta^{(h)}(t - h)$ on the left-hand side of \eqref{time-delayed} can be regarded as a forcing term, and the standard machinery of minimizing movements (see \cite{MR1260440}) yields a solution $\eta^{(h)}$ defined on $[0, h]$. We then consider \eqref{time-delayed} on the interval $[h, 2h]$, again with the understanding that the time-shifted time derivative $\partial_t \eta^{(h)}(t - h)$ should not be regarded as part of the solution, but as a known term, in this case given by the solution found in the previous step. Iterating this procedure leads to a piecewise-defined function, still denoted by $\eta^{(h)}$, defined on the whole time interval $[0, T]$. A solution to the original problem can then be obtained by passing to the limit with $h \to 0^+$. This delicate limiting process is explained in detail in \Cref{inertia-sec}.

It is worth noting that this method allows us to derive a corresponding existence result (including contact forces) for the quasistatic case (see \Cref{KR-recovery}). In particular, in this paper we solve the two open problems formulated in Remarks 3.2 and 3.3 in \cite{kromerQuasistaticViscoelasticitySelfcontact2019}.

\subsection{Outlook and future research directions}
The flexibility of the methods suggests that the results obtained can be generalized to different situations. The list below includes some of the directions that we plan to investigate in the future.
 
 \begin{itemize}
  \item \textbf{Irregular domains:} To simplify proofs and discussion, we have focused on domains that are somewhat regular, i.e.\ per construction the normals vary continuously and there are no edges or corners. In order to deal with more complex geometries, more care needs to be taken to study admissible directions for the contact force. In the static case (see \cite{MR3990255} and \cite{schurichtVariationalApproachContact2002}), this difficulty has been dealt with via the use of the Clarke-subdifferential \cite{clarkeBook}; it would be interesting to extend these results to the fully dynamic case as well.
  \item \textbf{Friction:} Similarly, we have simplified the situation by ignoring the effects of friction. In particular, dynamic friction could be of interest here, as it is purely dissipative as well as dependent on the contact force itself. It would thus perfectly fit into the framework presented.
  \item \textbf{Homogenization:} A common explanation of friction is via microscopic irregularities in the surface. As our result allows us to consider the effect of contact between macroscopic irregularities, it should be worth studying the resulting limit regimes when the scale of these irregularities is sent to zero.
  \item \textbf{Fracture mechanics:} A longstanding, active topic in solid mechanics is the study of fractures. While these represent the pulling apart of material and thus the opposite of collisions, they naturally result in situations where disconnected parts of the solid are close to each other. In particular in shear fractures immediate contact is to be expected. Extending our methods to this case thus seems like a natural target for future research. We mention here an important contribution of Dal Maso and Larsen \cite{MR2847479}, where the authors introduce a minimizing movements scheme for the study of the wave equation on domains with (evolving) cracks.
  \item \textbf{Fluid structure interactions:} Collisions between elastic bodies rarely occur in a vacuum. Instead, in numerous physical applications the volume that separates the solids is typically filled with a fluid (for example, air). Depending on the fluid, the boundary conditions, and the regularity of all surfaces involved, the presence of a fluid can result in large changes in behavior, in some situations even entirely preventing collisions (we refer to \cite{hillairet2007} for more information and to \cite{gravinaRebound2022} for a study of rebound dynamics). When contact is theoretically possible, existence results for this kind of systems generally break at the time of first collision. Thus, it is natural to ask whether the methods presented in this paper can be used to extend solutions past that point.
\end{itemize}
 
\section{Contact forces, admissible test functions, and their convergence}\label{Contact-sec} 

Before we begin with the proof of the main theorem, we first need to complete its statement with a more precise discussion of two related concepts: contact forces and the set of admissible test functions. These and their properties will not only be crucial in what follows, but some of the considerations here should be of independent interest for proving related results. Throughout the section we let $I \subset \RR$ be a closed time interval. For consistency with the strategy outlined at the end of Subsection \ref{main-sec}, $I$ will play the role of $[0, h]$ in our study of the quasistatic problem (see \Cref{Aux-sec}) and $[0, T]$ when considering the full problem (see \Cref{inertia-sec}).

\subsection{Contact set and forces}
Let us begin by giving the definitions of contact set and contact force for a given deformation $\eta$. For the convenience of the reader, we  recall some well-known properties that will be used throughout the rest of the section.
 
\begin{definition}[Contact set]\label{def-contact-set}
\begin{enumerate}[label=(\roman*)]
\item[$(i)$] Let $\eta\in \mathcal{E}$. The \emph{(Lagrangian) contact set} of $\eta$ is defined via
  \begin{align*}
   C_\eta & \coloneqq \{x \in \overline{Q} : \eta(x) \in \partial\Omega \text{ or } \eta^{-1}(\eta(x)) \neq \{x\}  \}.
  \end{align*}
Note that $C_{\eta}$ consists of points of self-contact as well as points of contact with the boundary of the fixed domain $\Omega$.
\item[$(ii)$] Let $\eta \colon I \times \overline Q \to \RR^n$ be such that $\eta(t) \in \mathcal{E}$ for all $t \in I$, then we define its \emph{(Lagrangian) contact set} as 
\begin{equation*}
C_\eta \coloneqq \{(t,x) \in I \times \overline{Q} : x \in C_{\eta(t)} \},
\end{equation*}
where $C_{\eta(t)}$ denotes the contact set for the deformation $\eta(t, \cdot)$, defined as in $(i)$. 
\end{enumerate}
\end{definition}

The following result contains well-known structural properties of the contact set. In particular, by the regularity of $\eta$ and the Ciarlet-Ne\v{c}as condition (see \eqref{domain-def}), one can show that there are no contact points in the interior. 
\begin{lem}
\label{Ceta-basic-properties}
For $\eta \in \mathcal E$, let $C_{\eta}$ be given as in \Cref{def-contact-set}. Then $C_{\eta} \subset \partial Q$. Furthermore, for $x\in C_\eta$ we have that 
\begin{itemize}
\item[$(i)$] if $\eta(x) \in \partial\Omega$, then $\eta^{-1}(\eta(x)) = \{x\}$ and $n_{\eta}(x)$ coincides with the interior unit normal vector to $\partial \Omega$ at $\eta(x)$;
\item[$(ii)$] if $\eta^{-1}(\eta(x))\neq \{x\}$, then $\eta^{-1}(\eta(x)) = \{x, y\}$ for some $y\in \partial Q$ and $n_{\eta}(x) + n_{\eta}(y) = 0$.
\end{itemize}
\end{lem}

For a proof of \Cref{Ceta-basic-properties} we refer to Theorem 2 in \cite{CN87} (see also Lemma 2 in \cite{PalmerHealey}). The time-dependent version stated below follows from the same argument with straightforward changes. Below we use $\eta^{-1}$ to denote the inverse with respect to only the space variable, that is, $\eta^{-1}(t,\eta(t,x)) \coloneqq \{z\in \overline Q: \eta(t,z)=\eta(t,x) \}$.

\begin{lem}
\label{Ceta-t-lem}
For $\eta \colon  I  \times \overline Q \to \RR^n$ with $\eta(t) \in \mathcal{E}$ for all $t \in  I $, let $C_{\eta}$ be given as in \Cref{def-contact-set}. Then $C_{\eta} \subset I \times \partial Q$. Furthermore, for $(t, x) \in C_{\eta}$ we have that
\begin{itemize}
\item[$(i)$] if $\eta(t, x) \in \partial \Omega$, then $\eta^{-1}(t, \eta(t, x)) = \{x\}$ and $n_{\eta}(t, x)$ coincides with the interior unit normal vector to $\partial \Omega$ at $\eta(t, x)$;
\item[$(ii)$] if $\eta^{-1}(t, \eta(t, x)) \neq \{x\}$, then $\eta^{-1}(t, \eta(t, x)) = \{x, y\}$ and $n_{\eta}(t, x) + n_{\eta}(t, y) = 0$.
\end{itemize}
\end{lem}

\begin{rmk}\label{normal-continuity}
Note that \eqref{E2} guarantees that the normal field $n_\eta$ inherits the continuity of $\nabla \eta$. Indeed, if $\eta \in \E$ is such that $E(\eta)\leq E_0$, then by \eqref{E2} and \eqref{E3} we have that 
\[
|(\nabla\eta)^{-1}| \le \frac{|\nabla \eta|^{n-1}}{\det \nabla \eta} \le C
\]
for some constant $C$ that depends only on $E_0$. In particular, this implies that $|(\nabla \eta)^{-T} n_Q|\geq \e_0$ for some $\e_0$ depending on $E_0$. In turn, we have that 
\[
n_\eta = \frac{\cof \nabla \eta \,n_Q}{| \cof \nabla \eta\, n_Q|} = \frac{(\nabla \eta)^{-T} n_Q}{|(\nabla \eta)^{-T} n_Q|}
\] 
belongs to $C^{0,\alpha}(\partial Q; \RR^n)$, with H\"older seminorm bounded by a constant that only depends on $E_0$. 

Additionally, for time dependent deformations, we note that if $\eta \colon I \times Q \to \Omega$ is such that $E(\eta(t)) < E_0$, then by an application of the chain rule, we see that $n_{\eta(t)}$ inherits some of the regularity of $\partial_t \nabla \eta$, e.g.\@ 
\[
\norm[L^2]{\partial_t n_{\eta(t)}} \leq C \norm[L^2]{\partial_t \nabla \eta(t)}
\]
for a constant $C$ depending only on $Q$ and $E_0$.
\end{rmk}

From this we can derive a well-known result about up-to the boundary local injectivity.
\begin{lem}
\label{IFT}
Let $\eta \in \E$ with $E(\eta)<\infty$. Then there exists a positive radius $r$ depending on $E(\eta)$ such that the restriction of $\eta$ to $B_r(x) \cap \overline{Q}$ is injective for all $x \in \overline{Q}$.
\end{lem} 
\begin{proof}
Suppose that we have $x,y\in \overline{Q}$ with $\eta(x)=\eta(y)$ and $x\neq y$. Then by \Cref{Ceta-basic-properties} we have that $x,y\in\partial Q$ and $n_\eta(x)=-n_\eta(y)$, so in particular $|n_\eta(x)-n_\eta(y)|=2$. As noted in \Cref{normal-continuity}, the seminorm $|n_\eta|_{C^{0,\alpha}}$ can be bounded in terms of only $E(\eta)$. Therefore $|x-y|>2r$ with $r$ depending only on $E(\eta)$. This implies that $\eta$ must be injective on every ball of radius $r$.
\end{proof}

With this in hand, we can define contact forces as follows.

\begin{definition}[Contact force]\label{def-contact-force}
\begin{enumerate}[label=(\roman*)]
\item[$(i)$] Let $\eta \in \mathcal{E}$. Then a \emph{contact force} for $\eta$ is a vector valued measure $\sigma\in M(\partial Q;\RR^n)$ with $\supp \sigma \subset C_\eta$ and with the property that on its support $\sigma$ points in the direction of $n_{\eta}$ in the sense that there exists a non-negative measure $|\sigma| \in M^+(\partial Q)$ such that $d\sigma = n_\eta d|\sigma|$, that is,  
\begin{equation*}
\int_{\partial Q} \varphi \cdot d\sigma = \int_{\partial Q} \varphi \cdot n_\eta \, d|\sigma|
\end{equation*}
for all $\varphi \in C(\partial Q; \RR^n)$.
\item[$(ii)$] Let $\eta \colon  I  \times \partial Q \to \RR^n$ be such that $\eta(t)\in \mathcal E$ for all $t \in  I $ and assume that $\eta$ is Borel measurable when considered as a mapping from $I$ into $W^{2,p}(Q;\RR^n)$. Then a \emph{contact force} for $\eta$ is a vector valued measure $\sigma \in M( I  \times \partial Q;\RR^n)$ with $\supp \sigma \subset C_\eta$ and with the property that on its support $\sigma$ points in the direction of $n_{\eta}$ in the sense that there exists a non-negative measure $|\sigma|\in M^+( I  \times \partial Q)$ such that $d\sigma = n_\eta d|\sigma|$, that is,
\begin{equation*}
\int_{ I  \times \partial Q} \varphi \cdot d\sigma = \int_{ I \times \partial Q} \varphi \cdot n_\eta \, d|\sigma|
\end{equation*}
for all $\varphi \in C( I \times \partial Q; \RR^n)$.
\item[$(iii)$] We say that a contact force $\sigma$ \emph{satisfies the action-reaction principle at self-contact} if 
\begin{equation*}
\int_{\partial Q} (\varphi \circ \eta) \cdot d\sigma = 0
\end{equation*}
for all $\varphi \in C_c(\Omega;\RR^n)$. Similarly, for time-dependent deformations, we say that $\sigma$ \emph{satisfies the action-reaction principle at self-contact} if
\begin{equation*}
\int_{ I \times \partial Q} (\varphi \circ \eta) \cdot d\sigma = 0
\end{equation*}
for all $\varphi \in C_c( I  \times \Omega;\RR^n)$.
\end{enumerate}
\end{definition}

\begin{rmk}
By the polar decomposition of measures (see, for example, Corollary 1.29 in \cite{AFP}), every measure $\mu \in M( I \times\partial Q;\RR^n)$ can be decomposed as $d\mu = g\, d|\mu|$, where $|\mu|\in M^+( I \times\partial Q)$ is the total variation of $\mu$ and $g\colon  I \times \partial Q\to \RR^n$ is Borel measurable with $|g|\equiv 1$ everywhere on $ I \times \partial Q$. Thus, if $\supp \sigma \subset C_\eta$, the definition above says that $\mu$ is a contact force whenever this decomposition holds with $g = n_\eta$ on $ I  \times \partial Q$.
\end{rmk}

\subsection{Compactness-Closure theorems for contact forces}
In this section we investigate compactness and closure properties of contact forces. These will enable us to conclude that contact forces associated to approximate solutions will converge to the contact force of the limiting solution. We present time-independent and time-dependent versions of these results, as both will be needed throughout the rest of the paper. However, we omit the proofs for the former case since these follow from analogous (but simpler) arguments.

\begin{thm}[Compactness-Closure for contact forces]
\label{cc-cf}
Let $\{\eta_k\}_{k \in \NN} \subset \E$ be given and assume that there exist $\eta \in \E$ and a constant $E_0$, independent of $k$, such that $\eta_k \to \eta$ in $C^1(Q;\RR^n)$, $E(\eta_k) \leq E_0$ for all $k$, and $E(\eta) \le E_0$. 
For every $k$, let $\sigma_k \in M(\partial Q;\RR^n)$ be a contact force for $\eta_k$ and assume that 
\[
\sup_k \|\sigma_k\|_{M(\partial Q;\RR^n)} \le C.
\]
Then there exist a subsequence (not relabelled) and a limit measure $\sigma$ such that $\sigma_k \weakstarto \sigma$ in $M(\partial Q;\RR^n)$. Moreover, $\sigma$ is a contact force for $\eta$, and if $\sigma_k$ satisfies the action-reaction principle at self-contact for all $k$, then so does $\sigma$.
\end{thm}

\begin{thm}[Compactness-Closure for contact forces, time-dependent]
\label{compactness-cforce}
Let $\eta_k, \eta \colon  I \times Q \to \RR^n$ be such that $\eta_k(t), \eta(t) \in \E$ for all $t$ and all $k \in \NN$ and assume that $E(\eta_k(t)) \le E_0$ for all $k$, $E(\eta(t)) \le E_0$ for some $E_0$ independent of $k$. Furthermore, assume that $\eta_k(t) \to \eta(t)$ in $C^1(Q; \RR^n)$ uniformly in $t$ for $t \in I$, that $\eta_k$ is Borel measurable in time for all $k$, and that $n_\eta\in C(I\times\partial Q;\RR^n)$.
For every $k$, let $\sigma_k \in M( I \times \partial Q;\RR^n)$ be a contact force for $\eta_k$ and assume that 
\[
\sup_k \|\sigma_k\|_{M( I \times \partial Q;\RR^n)} \leq C.
\]
Then there exist a subsequence (not relabelled) and a limit measure $\sigma$ such that $\sigma_k \weakstarto \sigma$ in $M(I \times \partial Q; \RR^n)$. Moreover, $\sigma$ is a contact force for $\eta$ and if $\sigma_k$ satisfies the action-reaction principle at self-contact for all $k$, then so does $\sigma$.
\end{thm}

\begin{proof}
By \Cref{def-contact-force}, we have that $d \sigma_k = n_{\eta_k}d|\sigma_{k}|$ with $|\sigma_{k}|\in M^+( I \times \partial Q)$ and $\supp \sigma_k\subset C_{\eta_k}$. Since the sequence $\{\sigma_k\}_{k \in \NN}$ is bounded in $M(I \times \partial Q;\RR^n)$, we must have that $\{|\sigma_{k}|\}_{k \in \NN}$ is bounded in $M^+(I \times \partial Q)$. Thus, eventually extracting a subsequence (which we do not relabel), we have that $|\sigma_{k}| \weakstarto |\sigma|$ for some $|\sigma| \in M^+(I \times \partial Q)$. Let $\sigma \in M( I \times\partial Q; \RR^n)$ be defined by setting $d\sigma = n_\eta d|\sigma|$. We claim that $\sigma_k \weakstarto \sigma$ in $M( I \times \partial Q; \RR^n)$. To prove the claim, let $g \in C(I \times \partial Q; \RR^n)$ and observe that
\begin{equation}
\label{identification-proof}
\int_{ I \times \partial Q} g\cdot n_{\eta_k} \, d|\sigma_k| -\int_{I \times \partial Q} g\cdot n_\eta \, d|\sigma| = \int_{I \times \partial Q} g\cdot (n_{\eta_k}-n_\eta) \, d|\sigma_k| + \int_{I \times \partial Q} g\cdot n_\eta \, d(|\sigma_k| - |\sigma|).
\end{equation}
It is worth noting that the first term on the right-hand side of \eqref{identification-proof} is well defined since $n_{\eta_k}$ is Borel measurable by assumption and $|\sigma_k|$ is a non-negative Radon measure. Moreover, observe that 
\begin{equation*}
\left|\int_{ I \times \partial Q} g\cdot (n_{\eta_k} - n_\eta) d|\sigma_k| \right| \le \|n_{\eta_k}-n_\eta\|_{L^{\infty}} \| |\sigma_k| \|_{M^+} \le C \| n_{\eta_k} - n_\eta \|_{L^\infty} \to 0
\end{equation*}
as $k \to \infty$ since $n_{\eta_k}\to n_\eta$ uniformly on $ I \times\partial Q$. Finally, the continuity of $g \cdot n_\eta$ and the fact that $|\sigma_k| \weakstarto |\sigma|$ imply that, as $k \to \infty$, the last term on the right-hand side of \eqref{identification-proof} vanishes as well. This proves the claim.


Next, we prove that $\supp \sigma \subset C_\eta$. To this end, fix $(t,x) \in \supp\sigma$. By the weak convergence of measures, we can find $(t_k,x_k) \in \supp\sigma_k \subset C_{\eta_k}$ such that $(t_k,x_k) \to (t,x)$. There are now two possibilities: either there is a further subsequence (not relabelled) such that $\eta_k(t_k,x_k) \in \partial\Omega$, or there exist points $y_k \in \partial Q$ with $y_k \neq x_k$ such that $\eta_k(t_k,x_k) = \eta_k(t_k,y_k)$. In the first case, using the estimate
\begin{equation*}
|\eta(t,x) - \eta_k(t_k,x_k)| \leq |\eta(t,x) - \eta(t_k,x_k)| + |\eta(t_k,x_k) - \eta_k(t_k,x_k)|,
\end{equation*}
the uniform convergence of $\eta_k$ to $\eta$ and the uniform continuity of $\eta$, we get that $\eta_k(t_k,x_k) \to \eta(t,x)$. Since $\eta_k(t_k,x_k) \in \partial \Omega$ for all $k$, we must also have that $\eta(t,x) \in \partial \Omega$, and therefore $(t,x) \in C_{\eta}$. 
In the second case, in view of the fact that $E(\eta_k(t_k)) \le E_0$, \Cref{IFT} gives the existence of a minimal distance $r$ (which only depends on $E_0$) with the property that $|x_k - y_k| \ge r$.  By the compactness of $\partial Q$, eventually extracting a further subsequence, we can assume that $y_k \to y \in \partial Q$. Reasoning as above, by the uniform convergence we see that $\eta(t, x) = \eta(t, y)$. Since necessarily $x \neq y$, we conclude that $(t, x) \in C_\eta$ also in this case.  

Finally, we note that for any $\varphi \in C_c(\Omega;\RR^n)$ we have $\varphi \circ \eta_k \to \varphi \circ \eta$ uniformly. This, together with the convergence of $\sigma_k \weakstarto \sigma$, implies that the action-reaction principle continues to hold in the limit.
\end{proof}

\begin{rmk}
An added difficulty in the proof is that we do not have continuity in time, but only Borel measurability. This makes justifying that $\sigma_k \weakstarto \sigma$ not an immediate consequence of the convergence $|\sigma_k| \weakstarto |\sigma|$. However, as shown above, continuity of the limit and uniform convergence can be used to justify this convergence. While these assumptions may not seem natural at first glance, we mention here that they arise from the construction of solutions to the quasistatic problem. To be precise, we will consider approximations that are piecewise constant and uniformly bounded in time and prove that these converge uniformly to a limiting deformation that is continuous in time. 
\end{rmk}

\begin{rmk} \label{rem:lowerRegularityFriction}
 \Cref{cc-cf} and \Cref{compactness-cforce} are designed for higher order materials and the associated sense of convergence. In fact, the statements are in general not true if the convergence is not strong enough. Specifically, if there is no pointwise convergence of $\nabla \eta_k$, then the limit measure might not lie in the contact set. Additionally, it is necessary to have a strong sense of convergence for the normal vector to guarantee that the limit measure is still pointing in the normal direction. Indeed, take $Q = [0,1]^2$ and consider a sequence of deformations such that 
\[
\eta_k(0,x_2) = \eta_k(1,1-x_2) = \left(\frac{1}{k}\sin(kx_2),x_2\right)
\]
for all $x_2 \in [0,1]$. This can be done in such a way that $\eta_k$ converges to some $\eta$ with 
\[
\eta(0,x_2) = \eta(1,1-x_2) = (0,x_2)
\]
in a sense that does not imply pointwise convergence of $\nabla \eta_k$, e.g., weakly in $W^{2,2}(Q;\RR^2)$.
  
  Now, observe that every contact set $C_{\eta_k}$ contains points $x$ for which $n_{\eta_k}(x)$ points in direction $(1,1)$. In particular, we can construct contact forces $\sigma_k$ which only point in this direction and have unit mass. But then there exist a subsequence and a limit measure $\sigma$ with the same property. However, all normals associated to $C_\eta$ are of the form $(0,\pm 1)$.
  
   Yet, we also note that if we are not studying self contact, but contact between a deformable solid and an immovable obstacle, the approach presented above can be adapted to work also for lower order materials, as the obstacle's normal is fixed.
 \end{rmk}

Finally we note that if the measures $\sigma_k$ are additionally uniformly $L^2$ in time, as will be the case for the corresponding quasistatic result, then the same holds for the limit. Before stating this result, we discuss the notion of a measure being ``$L^2$ in time'', namely the space $L^2_{w^*}(I;M(\partial Q;\mathbb R^n))$.

\begin{rmk}[Definition of $L^2_{w^*}(I;M(\partial Q;\mathbb R^n))$]\label{l2weakstar-def}
We use $L^2_{w^*}(I;M(\partial Q;\mathbb R^n))$ to denote the space of all $\sigma\colon I \to M(\partial Q;\mathbb R^n)$ which are weak$^*$ measurable, satisfy 
\[
\|\sigma\|_{L^2_{w^*}(I;M(\partial Q;\mathbb R^n))}^2\coloneqq\int_I\|\sigma(t)\|^2_{M(\partial Q;\mathbb R^n)}dt<\infty,
\]
and any such two $\sigma,\tilde \sigma$ are considered to be equivalent if $\sigma(t)=\tilde\sigma(t)$ for $\mathcal{L}^1$-a.e.\@ $t \in I$. Here we recall that $\sigma$ is said to be weakly$^*$ measurable if 
\[
t \mapsto \int_{\partial Q} \varphi \cdot d\sigma(t)
\]
is measurable for all $\varphi\in C(\partial Q;\mathbb R^n)$.
\newline
Observe that any $\sigma \in L^2_{w^*}(I;M(\partial Q;\mathbb R^n))$ can be regarded as an element of $M( I \times \partial Q;\mathbb R^n)$ by setting $$\int_{ I\times \partial Q}\varphi (t,x) \cdot d\sigma(t,x) \coloneqq \int_I \int_{\partial Q} \varphi (t,x)\cdot d\sigma(t)\, dt,\quad \varphi\in C( I\times\partial Q;\mathbb R^n).$$
Conversely, if $\sigma \in M( I\times \partial Q;\mathbb R^n)$ can be represented as $$\int_{ I\times \partial Q}\varphi (t,x) \cdot d\sigma(t,x) = \int_I \int_{\partial Q} \varphi (t,x)\cdot d\nu_t \, \phi(t)\, dt,\quad \varphi\in C( I \times\partial Q;\mathbb R^n)$$
with $\phi\in L^2(I;\mathbb R)$ and $\{\nu_t\}_{t\in I}\subset M(\partial Q;\mathbb R^n)$ a bounded weakly$^*$ measurable family of measures, then one can regard $\sigma$ as belonging to $\sigma \in L^2_{w^*}(I;M(\partial Q;\mathbb R^n))$ by setting $$\sigma(t)\coloneqq \phi(t)\nu_t,$$
for $\mathcal{L}^1$-a.e.\@ $t\in I$.
\end{rmk}

\begin{lem}\label{cforce-l2intime}
Let $\sigma_k\in L_{w^*}^2( I ;M(\partial Q;\RR^n))$ with $\|\sigma_k\|_{L^2_{w^*}( I ;M(\partial Q;\RR^n))}\leq C$ be such that $\sigma_k \weakstarto \sigma$ in $M(  I \times \partial Q;\RR^n)$. Then $\sigma\in L_{w^*}^2( I ;M(\partial Q;\RR^n))$ and $\|\sigma\|_{L_{w^*}^2( I ;M(\partial Q;\RR^n))}\leq C$.
\end{lem}
\begin{proof} 

By the disintegration theorem (see e.g.\@ \cite[Theorem 2.28]{AFP}) if we define $\mu\in M^+( I )$ by $\mu(A)=|\sigma|(A\times \partial Q)$ for any Borel set $A\subset  I $, we obtain a weakly$^\ast$ measurable family of measures $\{\nu_t\}_{t\in I }\subset M(\partial Q;\RR^n)$  with $\|\nu_t\|_{M(\partial Q;\RR^n)}=1$ such that $\sigma = \mu(dt)\otimes\nu_t$. More precisely, we have that
\begin{equation*}
\int_{{I} \times \partial Q} \varphi \,d\sigma = \int_{ I} \int_{\partial Q} \varphi(t,x)\,d\nu_t(x)\,d\mu(t), \quad \varphi \in C(  I \times \partial Q).
\end{equation*}
We now show that $\mu$ has $L^2$ density with respect to the Lebesgue measure on $ I $. To see this, let $g \in C( I )$ with $\|g\|_{L^2}\leq 1$. Then for all $k$ we have that
\begin{equation}
\label{dis-duality}
C\geq \| \sigma_k \|_{L_{w^*}^2(I ;M(\partial Q;\RR^n))}\geq \int_I \|  \sigma_k(t)  \|_{M(\partial Q;\RR^n)} g(t) \,dt =\int_{I \times \partial Q} \varphi \,d|\sigma_k| 
\end{equation}
where $\varphi(t,x) \coloneqq g(t)$. Since $|\sigma_k| \weakstarto |\sigma|$ in $M( I \times \partial Q)$, \eqref{dis-duality} implies that
\begin{equation}
\label{dis-duality2}
C \geq \int_{I \times \partial Q} \varphi\,d|\sigma|=\int_I\int_{\partial Q} \varphi(t,x) \,d\nu_t(x)\,d\mu(t)=\int_I g(t)\,d\mu(t).
\end{equation}
Since this is true for all $g$, \eqref{dis-duality2} shows that $\mu$ defines a linear functional on $C(I)$ which is bounded with respect to the $L^2$-norm. By the Hahn--Banach theorem, this functional admits an extension to $L^2(I)$ and can be represented by $\phi \in L^2(I)$, in the sense that \begin{equation*}
\int_I g \,d\mu = \int_I g \phi \,dt, \quad g \in C(I).
\end{equation*}
Thus, as argued at the end of \Cref{l2weakstar-def}, we see that $\sigma\in L_{w^*}^2(I;M(\partial Q;\RR^n))$. Moreover, we have  $\|\sigma\|_{L_{w^*}^2(I;M(\partial Q;\RR^n))}=\|\phi\|_{L^2(I)}\leq C$. This completes the proof.
\end{proof}

\begin{rmk}
Even the stronger assumption of $\sigma_k\in L^2( I ;M(\partial Q))$, i.e.\@ strong measurability, (which is satisfied for our approximation in \Cref{existence-CF-quasi}) does not help, as the strong measurability may be lost in the weak$^{*}$ limit. This is closely related to the fact that $L^2( I ;M(\partial Q))$ is not the dual space to $L^2( I ;C(\partial Q))$, as $M(\partial Q)$ does not satisfy the Radon-Nikod\'ym property (see e.g. \cite[Section 1.3]{banachSpacesBook}).
\end{rmk}

\subsection{Normals and almost normals}
In the core sections of the paper, it would prove convenient to use the normal field $n_{\eta}$ as a way of perturbing an admissible deformation. This will however not be allowed as normal vectors lack the needed regularity. We overcome this difficulty by introducing an ``almost'' normal field that is sufficiently regular for our purposes. 

\begin{definition}[Almost normals]
\label{AN-def}
\begin{itemize}
\item[$(i)$] Let $\eta\in \mathcal E$. Then $\tilde n_\eta\colon \overline Q \to \RR^n$ is an \emph{almost normal} to $\eta$, if \[
\tilde n_\eta(x)\cdot n_\eta(x) > \frac12 \ \forall x\in \partial Q \qquad \text{ and } \qquad |\tilde n_\eta(x)|\leq 1 \ \forall x \in\overline Q. 
\]
\item[$(ii)$] Let $\eta\in L^\infty( I ;W^{2,p}(Q;\RR^n))\cap W^{1,2}( I ;W^{1,2}(Q;\RR^n))$ be such that $E(\eta)\in L^\infty( I )$. Then a vector field $\tilde n_\eta \colon  I \times \overline Q\to \RR^n$ is called \emph{almost normal} to $\eta$, if
\[
\tilde n_\eta(t,x)\cdot n_\eta(t,x) > \frac12 \ \forall (t,x)\in  I \times \partial Q \qquad \text{ and } \qquad  |\tilde n_\eta(t,x)|\leq 1 \ \forall (t,x)\in  I \times \overline Q. 
\]
\end{itemize}
\end{definition}

Notice in particular that almost normals are defined also in the interior of $Q$. While the existence of a sufficiently regular (with respect to the space variables) almost normal $\tilde n_\eta$ is well known to specialists in the field, the precise statement used in the following and its proof are included here for the convenience of the reader. As before, we present both the time-independent and time-dependent versions, but only prove the slightly more complicated time-dependent version.

\begin{prop}\label{smooth-almost-normal-notime}
Let $\eta\in \mathcal{E}$ with $E(\eta)< \infty$. Then there exists an almost normal $\tilde n_\eta$ to $\eta$ with $\tilde n_\eta \in C^{k_0}(\overline Q;\RR^n)$ for all $k_0\in \NN$ satisfying $\| \tilde n_\eta \|_{C^{k_0}(Q;\RR^n)}\leq  C_{k_0}$, where $C_{k_0}$ depends only on $E(\eta)$ and $k_0$.
\end{prop}

\begin{prop}\label{smooth-almost-normal}
Let $\eta\in L^\infty( I ;\mathcal{E})\cap W^{1,2}( I ;W^{1,2}(Q;\RR^n))$ be such that $E(\eta(t)) \leq E_0$ for all $t\in  I $. Then there exists an almost normal $\tilde n_\eta$ to $\eta$ with $\tilde n_\eta \in L^\infty( I ;C^{k_0}(\overline Q;\RR^n))$ for all $k_0\in \NN$, satisfying $\| \tilde n_\eta \|_{L^\infty( I ;C^{k_0}(Q;\RR^n))}\leq  C_{k_0}$, where $ C_{k_0}$ depends only on $E_0$ and $k_0$. Moreover, we can have that $\tilde n_\eta\in W^{1,2}(I;L^2(Q;\mathbb R^n))$ with $\|\partial_t \tilde n_\eta\|_{L^2(I\times Q;\mathbb R^n)}\leq C \|\partial_t \nabla \eta\|_{L^2(I\times Q;\mathbb R^n)}$, where $C$ depends only on $E_0$.
\end{prop}

\begin{proof}
Reasoning as in \Cref{normal-continuity}, due to the uniform bound on $E(\eta(t))$ and the fact that $\eta \in C( I ;C^{1,\alpha}(Q;\RR^n))$, we obtain that $n_\eta\in C( I ;C^{0,\alpha}(\partial Q;\RR^n))$. Notice that we can find $\delta > 0$ such that for all $t \in I$ and all $x,\tilde{x} \in \partial Q$ it holds that 
\begin{equation*}
n_\eta(t,x)\cdot n_\eta(t,\tilde x) > 3/4
\end{equation*}
whenever $|x - \tilde{x}| < \delta$. We then consider an extension of $n_Q$ (still denoted by $n_{Q}$) to the $\delta$-neighborhood of $Q$, namely $Q_{\delta}$, such that $n_Q\in C^{0,\alpha}(Q_\delta;\RR^n)$ with 
\begin{equation*}
\|n_Q\|_{C^{0,\alpha}(Q_\delta;\RR^n)}
\leq C(\delta )\|n_Q\|_{C^{0,\alpha}(\partial Q;\RR^n)}
\end{equation*}
and with the property that $|n_Q|\equiv 1$ on $ P_\delta$, where $P_\delta$ denotes the $\delta$-neighborhood of $\partial Q$. We then consider the extension of the deformed normal $n_\eta$ to $I\times Q_\delta$ (again, still denoted by $n_\eta$) by  \[
n_\eta (t,x) = \frac{\cof \nabla \eta(t,x) \,n_Q(x)}{| \cof \nabla \eta(t,x)\, n_Q(x)|} |n_Q(x)|, \quad (x,t)\in I\times Q_\delta \text{ if } |n_Q(x)| \neq 0
\]
and by zero otherwise, where we use an extension of $\eta$ to $I\times Q_\delta$ by a standard extension operator from the fixed domain $I\times Q$, which in particular can be chosen linear and so that it preserves the norms in $L^\infty( I ;W^{2,p}(Q_\delta;\RR^n))$ and $W^{1,2}( I ;W^{1,2}(Q_\delta;\RR^n))$ up to a constant. Note that this in particular implies sufficient regularity on the extended domain, so that the previous expression is well defined as $\cof \nabla \eta(t,x)\, n_Q(x)$ is uniformly continuous and bounded away from zero on $\partial Q$.

Arguing as in \Cref{normal-continuity} we see that $n_\eta\in C(I;C^{0,\alpha}(Q_\delta;\RR^n))$ with H\"older seminorm dependent only on $E_0$. Moreover, by an application of the chain rule, we see that 
\[
\norm[L^2(Q)]{\partial_t n_{\eta}(t)} \leq C \norm[L^2(Q)]{\partial_t \nabla \eta(t)}
\]
for a constant $C$ depending only on $Q$ and $E_0$.

Choose then $\tilde \delta >0$ (possibly smaller that $\delta$, dependent only on $E_0$) so that for all $t\in I$, all $x\in \partial Q$, and all $\tilde x\in Q_\delta$ we have that
\begin{equation*}
n_\eta(t,x)\cdot n_\eta(t,\tilde x) > 1/2
\end{equation*}
whenever $|x-\tilde{x}|< \tilde \delta$. Finally, we mollify the (extended) normal field in space, that is, we set $\tilde n_\eta \coloneqq n_\eta * \xi_{\tilde \delta}$ in $ I \times \overline{Q}$, where $*$ is convolution with respect to $x$ and $\xi_{\tilde \delta}$ is the standard mollification kernel with parameter $\tilde{\delta}$. Then, for each time $t$, we have that 
\[
\|\tilde n_\eta(t)\|_{C^{k_0}(\overline Q;\RR^n)} \leq C_{k_0}\|n_\eta(t)\|_{C(\partial Q;\RR^n)},
\]
where $C_{k_0}$ depends only on $E_0$ and $k_0$, and 
\[
\norm[L^2(Q)]{\partial_t \tilde n_{\eta}(t)} \leq\norm[L^2(Q)]{\partial_t n_{\eta}(t)} \leq C \norm[L^2(Q)]{\partial_t \nabla \eta(t)}.
\]
As one can readily check, $\tilde n_\eta$ is an almost normal to $\eta$ in the sense of \Cref{AN-def}. This concludes the proof.
%
\end{proof}

\subsection{Test functions for problems involving self-contact}
We can now introduce the cone of admissible test functions for our variational inequality and provide practical characterizations that are used throughout the rest of the paper. For our purposes, we need to develop this theory for both the static and the dynamical case. We believe that these characterizations are of independent interest and thus provide the precise statements for both cases. 

In the following we use $X$ to denote a Banach space that embeds compactly into $C^{1, \alpha}(Q; \RR^n)$. Throughout the paper we apply the results of this section for $X = W^{k_0, 2}(Q; \RR^n)$ and $X = W^{2, p}(Q; \RR^n)$ (with suitable conditions on $k_0$ and $p$).

\subsubsection{Test functions for static problems}

We begin by giving a definition of the admissible test functions.

\begin{definition} \label{def-admissibleTestFunctions}
Let $\eta \in \mathcal{E}$ be a given deformation with $E(\eta) < \infty$. Then we define the corresponding cone of admissible test functions (i.e., the tangent cone to $\mathcal{E}$) as 
\begin{equation}
\label{T-def}
\Adm{\eta}{X} \coloneqq \left\{\varphi \in W^{2,p}(Q;\RR^n) \cap X : 
\forall x \in C_{\eta} \sum_{z\in\eta^{-1}(\eta(x))}\varphi(z)\cdot n_\eta(z)\geq 0, \varphi|_\Gamma = 0\right\}.
\end{equation}
In case $X = W^{2,p}(Q;\RR^n)$, we omit it in the notation.
\end{definition}

We remark here for clarity that in view of \Cref{Ceta-basic-properties}, the sum in \eqref{T-def} consists of either one (in the case of contact with $\partial \Omega$) or two (in the case of self-contact) elements. Intuitively, the inequality in the definition of $\Adm{\eta}{X}$ signifies that admissible test functions, i.e.\@ admissible perturbations of the deformation $\eta$, cannot point outside of $\Omega$ whenever $\eta$ already lies on $\partial \Omega$, and cannot further displace the deformed configuration in a way that would cause interpenetration of matter.

\begin{lem}
Let $U,V\subset \RR^n$ be open, disjoint sets with $\partial U \cap \partial V \neq \emptyset$. Let $p\in \partial U \cap \partial V$ be such that $\partial U$, $\partial V$ are $C^1$ near $p$  and let vectors $u,v\in\RR^n$ be such that \begin{equation*}
u\cdot n_U(p) + v\cdot n_V(p) \leq -\delta<0.
\end{equation*}
Then there exist $\e_0>0$ and $\rho>0$ depending on $\delta$, $|u-v|$, and the $C^{1}$ moduli of continuity of $\partial U$ and $\partial V$ (in a neighborhood of $p$), such that $B_{\rho\e}(p+\e u) \subset V +\e v$ for all $\e \in (0,\e_0)$. 
\end{lem}

\begin{proof} 
As the proof follows from elementary consideration, we only give a sketch of the general idea. By assumption, the vector $w=u-v$ at the point $p$ points inside $V$. Since the boundary of $V$ is of class $C^1$, there exists a non-empty open cone in $V$ with vertex at $p$ in the direction $w$. The height of the cone gives $\e_0$, and the aperture of the cone gives $\rho$. \qedhere

%
%
\end{proof}

\begin{prop}\label{bad-displacement}
Let $\eta \in \mathcal{E}$ and $\varphi\in W^{2,p}_{\Gamma}(Q;\RR^n)$. If there exists $x\in C_\eta$ with 
\begin{equation*}
-\delta\coloneqq\sum_{z\in \eta^{-1}(\eta(x))}\varphi(z)\cdot n_\eta(z) <0,
\end{equation*} then there is $\e_0>0$ depending on $\delta$, $\varphi$, and $E(\eta)$ such that $\eta +\e \varphi \notin \mathcal E$ for all $\e \in (0,\e_0)$.
\end{prop}

\begin{proof}
Let $x$ and $\delta>0$ be as in the statement. If $\eta(x) \in \partial\Omega$ we set $U=\eta(Q)$, $V=\RR^n\setminus\overline\Omega$, $p=\eta(x)$, $u=\varphi(x)$, and $v=0$. Recalling that the $C^1$-modulus of continuity of $\eta(\partial Q)$ depends on $E(\eta)$ (see \Cref{normal-continuity}), the previous lemma gives $\e_0>0$ such that $\eta(x)+\e \varphi(x)\notin \overline\Omega$ for all $\e\in(0,\e_0)$, which implies $(\eta+\e\varphi)(Q)\not\subset \Omega$, and therefore that $\eta+\e \varphi \notin \mathcal E$.

If this is not the case then by \Cref{Ceta-basic-properties} we have that $\eta^{-1}(\eta(x))=\{x,y\}$ for some $y \neq x$. Then by \Cref{IFT} we can find a radius $r>0$ (which depends on $E(\eta)$) such that $U=\eta(B_r(x)\cap Q)$ and $V=\eta(B_r(y)\cap Q)$ are disjoint. As before, denote $p=\eta(x)$, $u=n_\eta(x)$, $v=n_\eta(y)$ and apply the previous lemma. We thus have $\e_0>0$ and $\rho>0$ such that \begin{equation*}
B_{\rho\e}(\eta(x))+ \e \varphi(x) \subset \eta(B_r(y)\cap Q)+ \e \varphi(y)
 \end{equation*}
for all $\e\in(0,\e_0)$. Eventually replacing $r$ with a smaller number so that $|\varphi(w)-\varphi(y)|\leq \rho$ for all $w\in B_r(y)$, we get that 
\begin{equation*}
\operatorname{dist}(\eta(B_r(y)\cap Q)+ \e \varphi(y),(\eta+\e\varphi)(B_r(y)\cap Q))\leq \rho\e.
\end{equation*} 
Consequently, we have that
 \begin{equation*}
\eta(x)+ \e \varphi(x) \in (\eta+\e\varphi)(B_r(y)\cap Q).
\end{equation*}
for all $\e\in(0,\e_0)$. Since by \Cref{Ceta-basic-properties} self-contact cannot happen at an interior point, this implies that $\eta+\e\varphi\notin \E$. This concludes the proof.
\end{proof}


Equipped with this, we can derive some useful characterizations for the cone of admissible test functions.

\begin{prop}[Characterizations of $\Adm{\eta}{X}$]
\label{T-char}
Let $\eta \in \mathcal{E}\cap X$ and $\varphi\in X$. The following are equivalent:
\begin{itemize}
\item[$(i)$] $\varphi \in \Adm{\eta}{X}$.
\item[$(ii)$] There exists a sequence $\{\varphi_k\}_k$ such that $\varphi_k \to \varphi$ in $X$, $\varphi_k = 0$ on $\Gamma$, and 
\begin{align} \label{eq:equivInterior}
\sum_{z\in \eta^{-1}(\eta(x))} \varphi_k(z)\cdot n_\eta(z)>0
\end{align}
for all $k \in \NN$ and all $x\in C_\eta$. 
\item[$(iii)$] There exist a sequence of positive real numbers $\{\e_k\}_k$ and a sequence of test functions $\{\varphi_k\}_k$ with $\varphi_k \to \varphi$ in $X$ such that $\eta + \e \varphi_k \in \E$ for all $\e \in [0, \e_k)$.
\end{itemize}
Additionally, the condition
\begin{itemize}
 \item[$(iv)$] there exists a curve $\Phi \in C([0, \e_0); \E \cap X) \cap C^1([0, \e_0); X)$ such that $\Phi(0) = \eta$ and $\Phi'(0^+) = \varphi$ 
\end{itemize}
implies any of $(i)$-$(iii)$ and conversely if $\varphi$ satisfies the condition imposed on $\varphi_k$ in \eqref{eq:equivInterior}, then this implies $(iv)$.\footnote{Note that this implies that \eqref{eq:equivInterior} characterizes the interior of $\Adm{\eta}{X}$. Additionally, for a sufficiently regular $\eta$, it is not hard to prove that $(iv)$ is in fact equivalent to the other three conditions. However, this proof involves splitting $\varphi$ into its precise normal and tangential components and is thus not easily transferred to the general situation.}
\end{prop}
\begin{proof}
\begin{description}
\item[$(i)\implies (ii)$]
It is enough to consider $\varphi_k \coloneqq \varphi + \frac{1}{k}\xi_{\Gamma} \tilde{n}_\eta$ where $\tilde{n}_\eta$ is the smooth almost normal to $\eta$ given by \Cref{smooth-almost-normal-notime} and $\xi_{\Gamma} \colon \overline{Q} \to \RR$ is a smooth function that vanishes on $\Gamma$ and is otherwise positive. Indeed, then for $x \in C_\eta$ we have that
\begin{equation*}
\sum_{z\in\eta^{-1}(\eta(x))} \varphi_k(z)\cdot n_\eta(z) \geq \sum_{z\in\eta^{-1}(\eta(x))} \frac{1}{k} \xi_{\Gamma}(z) \tilde n_\eta(z) \cdot n_\eta(z) \geq \frac{1}{2k} \sum_{z\in\eta^{-1}(\eta(x))} \xi_{\Gamma}(z)>0.
\end{equation*}
\item[$(ii)\implies (iii)$] This is proved in Proposition 3 in \cite{PalmerHealey} for the case of self-contact and $X=W^{2,p}$, but the argument easily extends to general $X$. The case of contact with the boundary is the same, as it can be treated like a part of the solid on which $\varphi = 0$.
\item[$(iii)\implies (i)$]
Let $\{\varphi_k\}_k$ be as in $(iii)$. We claim that $\varphi_k\in \Adm{\eta}{X}$. In fact, in view of \Cref{bad-displacement} $\varphi_k \notin \Adm{\eta}{X}$ implies $\eta+\e\varphi_k\notin\E$ for $\e$ arbitrarily small, which is a contradiction with $(iii)$. Now it is apparent that $\Adm{\eta}{X}$ is closed with respect to uniform convergence, therefore also with respect to convergence in $X$. Thus $\varphi\in\Adm{\eta}{X}$. This proves that $(i)$--$(iii)$ are equivalent. 
\item[$(iv)\implies (i)$]
To see this, let $x \in C_{\eta}$ and set 
\[
\tilde\varphi(\e, x) \coloneqq \frac{\Phi(\e, x) - \eta(x)}{\e}.
\]
We claim that we must have
\begin{equation}
\label{iv}
\lim_{\e \to 0^+} \sum_{z\in\eta^{-1}(\eta(x))} \tilde\varphi(\e, z) \cdot n_{\eta}(z) \geq 0.
\end{equation}
Notice that the limit in \eqref{iv} exists since $\Phi\in C^1([0,\e_0);X)$ and that furthermore it equals $\sum_{z\in\eta^{-1}(\eta(x))}\varphi(z) \cdot n_{\eta}(z)$. To prove the claim, arguing by contradiction, assume that there exists a number $\delta>0$ with
\[
\sum_{z\in\eta^{-1}(\eta(x))}\tilde\varphi(\e, z) \cdot n_{\eta}(z) \leq  -\delta < 0
\]
for all $\e$ small enough. Then, by \Cref{bad-displacement} we have that $\Phi(\e, \cdot) = \eta(\cdot) + \e \tilde \varphi(\e, \cdot) \notin \mathcal E$ for all $\e>0$ sufficiently small and we have thus arrived at a contradiction.

\item[$(i)+$\eqref{eq:equivInterior} $\implies(iv)$] Set $\Phi(\e,x) \coloneqq \eta(x) + \e \varphi(x)$. Then clearly $\Phi(0)= \eta$ and $\Phi'(0^+)=\varphi$. Additionally $\Phi$ has the required regularity. It is only left to check that $\Phi(\e,\cdot) \in \mathcal{E}$. For this, we note that per definition $\Phi(0) \in \mathcal{E}$ and there is $\e_0 >0$ such that $\Phi(\e,\cdot)$ is globally injective (and maps to the interior of $\Omega$) for any $\e \in (0,\e)$. 

In fact it is enough to check for injectivity at the boundary. For any points $x,y\in \partial Q$ such that $\eta(x)$ and $\eta(y)$ (resp. $\eta(x)$ and $\partial \Omega$) have a fixed minimum distance, this follows from continuity. Similarly for any points $x,y \in Q$ that are close to each other, the same follows from local injectivity (see \Cref{IFT}). This allows us to restrict our attention to pairs $(x,y)$ in an arbitrary neighborhood of the compact set $\{(\tilde{x},\tilde{y}) \in \partial Q \times \partial Q: \eta(\tilde{x}) = \eta(\tilde{y})\}$ (resp.\@ points $x$ in a neighborhood of $\{\tilde{x} \in \partial Q: \eta(\tilde{x}) \in \partial \Omega\}$). 

But then we can choose this neighborhood small enough so that the condition in \eqref{eq:equivInterior} extends to those pairs as well. Together with the definition of the normal vector, for any such pair $(x,y)$, this allows us to pick a unit vector $n$ such that $n \cdot (\eta(x)-\eta(y)) = 0$ and $(\varphi(x)-\varphi(y)) \cdot n >0$. This then implies that $\Phi(\e,x)$ and $\Phi(\e,y)$ cannot coincide. A similar result holds for $z\in \partial \Omega$ in place of $\eta(y)$.  \qedhere

\end{description}
\end{proof}

\begin{rmk} 
 Note that the approximating sequences in (ii) and (iii) are necessary. For a general $\varphi$ satisfying one of the conditions, it is not true that $\eta + \e \varphi \in \mathcal{E}$ for any $\e > 0$. Even if one excludes tangential movement, a condition like $\varphi \cdot n \geq 0$ only guarantees that the points of $C_\eta$ themselves are not displaced in a way that would violate interpenetration of matter or move outside of $\Omega$, but it cannot be used to conclude the same about any neighborhood.

 To see this in a simple example, consider the following situation. Take $\Omega$ to be the upper half-plane, $Q$ a subset of the upper half-plane such that $\partial Q$ includes $(-1,1)\times \{0\}$ and $\eta(x_1,x_2) \coloneqq (x_1,x_2 + x_1^2)$ as well as $ \varphi(x_1,x_2) \coloneqq (0,2x_1)$. Then there is contact only at $(0,0)$, where $\varphi \cdot n = 0$, but for any $\e > 0$ we have $(\eta(s,0) + \e \varphi_k(s,0))_2 < 0$ for some $s$ small enough and thus a non-superficial intersection with the boundary.
 
 However, $\varphi$ is still an admissible test function in the sense of $(i)$ and one can even construct the curve
 \begin{align*}
  \Phi(t,x) \coloneqq (x_1,x_2 + (x_1+t)^2)
 \end{align*}
 which satisfies $\Phi(0,\cdot) = \eta$ and $\tfrac{d}{dt} |_{t=0}\Phi = \varphi$ and as such is admissible in the sense of $(iv)$.
%
\end{rmk}

The next result shows that the set of admissible test functions is well-behaved with respect to sequences of approximating deformations.
\begin{prop} \label{prop:TErecoveryLocal}
Let $\{\eta_k\}_{k \in \NN} \subset \E \cap X$ be given with $E(\eta_k)$ uniformly bounded and assume that there exists $\eta \in \E \cap X$ such that $\eta_k \to \eta$ in $X$. Then, for every $\varphi \in \Adm{\eta}{X}$ there exists a sequence $\{\varphi_k\}_{k \in \NN}$ such that $\varphi_k \to \varphi$ in $X$ and with the property that $\varphi_k \in \Adm{\eta_k}{X}$ for all $k \in \NN$.
\end{prop}
\begin{proof}

By \Cref{T-char} $(ii)$, we can approximate $\varphi$ by a sequence of $\{\varphi_l\}_{l \in \NN}$ for which \eqref{eq:equivInterior} holds. Additionally, as we have seen in the proof of \Cref{compactness-cforce}, any converging sequence of contact points $\{x_k\}_k$ with $x_k \in C_{\eta_k}$ has to converge to a contact point $x \in C_\eta$. But then by the fact that $C_\eta$ is compact, $n_{\eta_k}$ converges uniformly, and since each element of $\{\varphi_l\}_{l \in \NN}$ is continuous, we have that for fixed $l$, \eqref{eq:equivInterior} holds for all $\eta_k$ with $k$ large enough. In particular, we can use this to pick a non-decreasing sequence $l_k$ such that $\varphi_{l_k} \in \Adm{\eta_k}{X}$.
\end{proof}

\subsubsection{Test functions for evolutionary problems}
Throughout this section, let
\begin{equation}\label{eta-time-dependent}
\eta\in L^\infty(I;\E \cap X) \cap W^{1,2}(I^{\circ};W^{1,2}(Q;\RR^n)) \quad \text{with} \quad E(\eta(t))\leq E_0 \text{ for all } t \in I.
\end{equation}
Note that this matches exactly the regularity that we ask for solutions to our time-dependent problem (see \Cref{var-in-def}). Recall that by the Aubin-Lions lemma $\eta$ admits a representative in $C(I; C^{1, \alpha}(Q; \RR^n))$. In the following, we always work with this representative without further notice.

Our aim here is to present the time-dependent versions of the results in the previous section. They follow along the same lines, therefore we will be brief and only indicate the differences to the static versions.

\begin{definition}\label{def-admissibleTestFunctions-timedep}
Let $\eta$ be given as in \eqref{eta-time-dependent}. By a slight abuse of notation, we denote
\begin{equation*}
C(I;\Adm{\eta}{X}) \coloneqq \{ \varphi \in C(I;W^{2,p}(Q;\RR^n) \cap X): \varphi(t) \in \Adm{\eta(t)}{X}, t\in I \},
\end{equation*}
where $\Adm{\eta(t)}{X}$ is understood as in the time-independent \Cref{def-admissibleTestFunctions}.
\end{definition}

\begin{lem}
\label{bad-displacement-timedep}
Let $\eta$ be given as in \eqref{eta-time-dependent} and let $\varphi \in C(I;X)$ be such that $\varphi(t, \cdot) = 0$ on $\Gamma$. If there exists $(t,x)\in C_\eta$ with\footnote{As before, $\eta^{-1}$ denotes the preimage of $\eta(t,\cdot)$.}
\begin{equation*}
-\delta\coloneqq \sum_{z\in\eta^{-1}(t,\eta(t,x))}\varphi(t,z)\cdot n_\eta(t,z)<0,
\end{equation*}
then there is $\e_0>0$ depending on $\delta, \varphi$ and $E_0$ such that $\eta(t)+\e\varphi(t)\notin \E$ for all $\e\in(0,\e_0)$.

\end{lem}
This is a straightforward consequence of \Cref{bad-displacement}. Therefore, we can formulate the time-dependent version of the various  characterizations of our test functions.
\begin{prop}[Characterizations of $C(I;\Adm{\eta}{X})$]
\label{time-TF}
The following are equivalent:
\begin{itemize}
\item[$(i)$] $\varphi \in C(I;\Adm{\eta}{X})$.
\item[$(ii)$] There exists a sequence $\{\varphi_k\}_k$ with $\varphi_k \to \varphi$ in $C(I;X)$ such that $\varphi(t, \cdot) = 0$ on $\Gamma$, and  
\begin{equation} \label{eq:equivInteriorTime}
\sum_{z \in \eta^{-1}(t, \eta(t, x))}\varphi_k(t, z)\cdot n_{\eta}(t, z) > 0
\end{equation}
for all $(t, x) \in C_{\eta}$ and all $k \in \NN$.
\item[$(iii)$] There exist a sequence of positive real numbers $\{\e_k\}_k$ and a sequence of test functions $\{\varphi_k\}_k$ with $\varphi_k \to \varphi$ in $C(I;X)$ such that $\eta(t) + \e \varphi_k(t) \in \E$ for $\mathcal{L}^1$-a.e.\@ $t \in I$ and all $\e \in [0, \e_k)$;
\end{itemize}
Additionally the condition
\begin{itemize}
\item[$(iv)$] there exists $\Phi \in C([0, \e_0); L^{\infty}(I; \E \cap X)) \cap C^1([0, \e_0); L^\infty(I;X))$ such that $\Phi(0) = \eta$ and $\Phi'(0^+) = \varphi$
\end{itemize}
implies any of $(i)$-$(iii)$ and conversely if $\varphi$ satisfies the condition imposed on $\varphi_k$ in \eqref{eq:equivInteriorTime}, then this implies $(iv)$.
\end{prop}
The proof follows analogously as the proof of \Cref{T-char} with the following changes: The smooth almost normal $\tilde n_\eta$ is now the time-dependent version from \Cref{smooth-almost-normal}; instead of \Cref{bad-displacement} one has to invoke the time-dependent version \Cref{bad-displacement-timedep}.

\begin{prop}\label{prop:TErecoveryGlobal}
Let $\{\eta_k\}_k \subset L^{\infty}(I;\E \cap X) \cap W^{1,2}(I^{\circ}; W^{1,2}(Q; \RR^n))$ with $E(\eta_k(t))\leq E_0$ be given and assume that there exists $\eta$ satisfying \eqref{eta-time-dependent} such that $\eta_k(t) \to \eta(t)$ in $X$, uniformly in $t$ for $t \in I$. Then, for every $\varphi \in C(I;\Adm{\eta}{X})$ there exists a sequence $\{\varphi_k\}_k$ such that $\varphi_k \to \varphi$ in $C(I; X)$ and with the property that $\varphi_k \in C(I;\Adm{\eta_k}{X})$ for all $k \in \NN$. Additionally, if $J \subset I$ and $\varphi \in C(I;\Adm{\eta}{X}) \cap C^1_c(J; L^2(Q; \RR^n))$, then we can find a sequence $\{\varphi_k\}_k$ as above but with the property that $\varphi_k \in C(I;\Adm{\eta_k}{X}) \cap C^1_c(J; L^2(Q; \RR^n))$ for all $k \in \NN$.
\end{prop}
Again, the proof is the same as in \Cref{prop:TErecoveryLocal}, using the time-dependent almost normal and uniform convergence in time.

\section{Existence results for the quasistatic regime}
\label{Aux-sec}
This section is concerned with the study of the quasistatic counterpart to \eqref{eq:strongSolution}. To be precise, for a given energy-dissipation pair $(E, R)$, we prove existence of solutions to 
\begin{equation}
\label{aux-Pb}
DE (\eta) + D_2 R(\eta, \partial_t \eta) = f
\end{equation}
in a framework compatible with the standard assumptions of second-order nonlinear elasticity. As a particular case of our analysis, we obtain an existence result for a parabolic variant of \eqref{eq:strongSolution}, that is, 
\begin{equation}
\label{aux-par}
\rho\frac{\partial_t \eta}{h} + DE(\eta) + D_2R(\eta, \partial_t \eta) = f + \rho\frac{\zeta}{h},
\end{equation}
where $\zeta(t)$ will later correspond to $\partial_t\eta(t-h)$ and the two terms thus form the difference quotient 
\[
\rho\frac{\partial_t \eta (t) -\partial_t \eta(t-h)}{h}.
\]

While \eqref{aux-par} has no direct physical relevance, solutions to this problem will play a fundamental role when inertial effects are taken into account (see \Cref{inertia-sec}). 

Similarly to \cite{PalmerHealey}, where the authors study equilibrium configurations for a second-order nonlinear solid, our presentation is divided into two parts. First, we prove the existence of solutions to a variational inequality via minimizing movements. Then, in the second part of the section, we recover the governing equations by reintroducing contact forces. These forces can be interpreted as a Lagrange multiplier associated to the global injectivity constraint.

\subsection{Existence of weak solutions via minimizing movements}
Throughout the section we consider an energy-dissipation pair $(E, R)$, which can be thought of as a regularized version of the energy-dissipation pair introduced in  \Cref{Modelling-sec} where, roughly speaking, the underlying space $W^{2,p}$ is replaced by $W^{k_0, 2}$, with $k_0 - \frac{n}{2}  > 2 - \frac{n}{p}$, so that $W^{k_0, 2}(Q; \RR^n)$ compactly embeds in $W^{2, p}(Q; \RR^n)$. Note that  the leading term in $E$ is now quadratic (see \Cref{inertia-sec}), which results in a corresponding linear term in the equation. The assumptions \eqref{E1}--\eqref{E5} and \eqref{R1}--\eqref{R4} are thus replaced by their regularized versions as follows.

We assume that $E\colon W^{k_0,2}(Q;\RR^n)\to (-\infty,\infty]$ satisfies the following properties: 
\begin{enumerate}[label=(E$'$\!.\arabic*), ref=E$'$\!.\arabic*]
\item \label{A1} There exists $E_{\min} > - \infty$ such that $E(\eta) \ge E_{\min}$ for all $\eta \in W^{k_0, 2}(Q; \RR^n)$. Moreover, $E(\eta) < \infty$ for every $\eta \in W^{k_0, 2}(Q; \RR^n)$ with $\inf_Q \det \nabla \eta > 0$.
\item \label{A2} For every $E_0  \ge E_{\min}$ there exists $\e_0 > 0$ such that $\det \nabla \eta \ge \e_0$ for all $\eta$ with $E(\eta) \le E_0$.
\item \label{A3} For every $E_0 \ge E_{\min}$ there exists a constant $C$ such that 
\[
\|\nabla^{k_0} \eta\|_{L^2} \le C 
\]
for all $\eta$ with $E(\eta) \leq E_0$.
\item \label{A4} $E$ is weakly lower semicontinuous, that is, 
\[
E(\eta) \le \liminf_{k \to \infty} E(\eta_k)
\]
whenever $\eta_{k} \rightharpoonup \eta$ in $W^{k_0, 2}(Q; \RR^n)$.
\item \label{A5} $E$ is differentiable in its effective domain with derivative $DE(\eta) \in (W^{k_0, 2}(Q; \RR^n))^*$ given by 
\[
DE(\eta) \langle \varphi \rangle = \left.\frac{d}{d\e} E(\eta + \e \varphi))\right|_{\mathrlap{\e = 0}}. 
\]
Furthermore, $DE$ is bounded and continuous with respect to weak convergence in $W^{k_0, 2}(Q; \RR^n)$ on any sub-level set of $E$. \end{enumerate}

Furthermore, we assume that the dissipation potential $R \colon W^{k_0, 2}(Q; \RR^n) \times W^{k_0, 2}(Q; \RR^n) \to [0, \infty)$ satisfies the following properties:
\begin{enumerate}[label=(R$'$\!.\arabic*), ref=R$'$\!.\arabic*]
\item \label{B1} $R$ is weakly lower semicontinuous in its second argument, that is, for all $\eta \in W^{k_0, 2}(Q; \RR^n)$ and every $b_{k} \rightharpoonup b$ in $W^{k_0, 2}(Q; \RR^n)$ we have that
\[
R(\eta, b) \le \liminf_{k \to \infty} R(\eta, b_k).
\]
\item \label{B2} $R$ is homogeneous of degree $2$ with respect to its second argument, that is, 
\[
R(\eta, \lambda b) = \lambda^2 R(\eta, b)
\]
for all $\lambda \in \RR$.
\item \label{B3} $R$ admits the following Korn-type inequality: For all $\e_0 > 0$, there exists a constant $K_R$ such that
\[
K_{R} \| b\|_{W^{k_0,2}}^2 \le \norm[L^2]{b}+R(\eta, b)
\]
for all $\eta \in \mathcal{E} \cap W^{k_0, 2}(Q; \RR^n)$ with $\det \nabla \eta > \e_0$ and $b \in W^{k_0,2}(Q;\RR^n)$.
\item \label{B4} $R$ is differentiable in its second argument, with derivative $D_2R(\eta, b) \in (W^{k_0, 2}(Q; \RR^n))^*$ given by
\[
 D_2R(\eta, b)\langle \varphi \rangle \coloneqq \left.\frac{d}{d\e} R(\eta, b + \e \varphi)\right|_{\mathrlap{\e = 0}}. 
\]
Furthermore, the map $(\eta, b) \mapsto D_2R(\eta, b)$ is bounded and weakly continuous with respect to both arguments, that is, 
\[
\lim_{k \to \infty}  D_2R(\eta_k, b_k) \langle \varphi \rangle =  D_2R(\eta, b) \langle \varphi \rangle
\]
holds for all $\varphi \in W^{k_0, 2}(Q; \RR^n)$ and all sequences $\eta_k \rightharpoonup \eta$ and $b_k \rightharpoonup b$ in $W^{k_0, 2}(Q; \RR^n)$.
\end{enumerate}
The quasistatic equivalent of \eqref{B3} is formulated as follows:
\begin{enumerate}[label=(R$'$\!.3\textsubscript{q}), ref=R$'$\!.3\textsubscript{q}]
 \item \label{B3q} $R$ admits the following Korn-type inequality: For any $\e_0 > 0$, there exists $K_R$ such that
\[
K_R \|b\|_{W^{k_0,2}}^2 \le  R(\eta, b)
\]
for all $\eta \in \mathcal{E}$ with $\det \nabla \eta > \e_0$ and all $b \in W^{k_0,2}_{\Gamma}(Q; \RR^n)$.
\end{enumerate}

Next, we give the precise definition of a solution to \eqref{aux-Pb}.
\begin{definition}
\label{sol-def-aux}
Let $h > 0$, $\eta_0 \in \E \cap W^{k_0, 2}(Q; \RR^n)$, and $f \in L^2((0, h); L^2(Q; \RR^n))$ be given. We say that 
\[
\eta \in W^{1,2}((0, h); W^{k_0,2}(Q; \RR^n)) \cap L^{\infty}((0, h); \E \cap W^{k_0, 2}(Q; \RR^n)) \quad \text{with} \quad E(\eta)\in L^\infty((0, h))
\]
is a solution to \eqref{aux-Pb} in $(0, h)$ if $\eta(0) = \eta_0$ and
\begin{equation}
\label{var-ineq}
\int_0^h  DE(\eta(t)) \langle \varphi(t) \rangle + D_2R(\eta(t), \partial_t \eta(t)) \langle \varphi(t)\rangle\,dt \ge \int_0^h \langle f(t), \varphi(t) \rangle_{L^2}\,dt
\end{equation}
holds for all $\varphi \in C([0,h];\Adm{\eta}{W^{k_0, 2})}$.
\end{definition}

\begin{rmk} Let us comment here on \Cref{sol-def-aux}.
\begin{itemize}
\item[$(i)$] It is worth noting that the initial condition $\eta(0) = \eta_0$ in \Cref{sol-def-aux} is satisfied in the classical sense since every element $\eta$ of $W^{1,2}((0, h); W^{k_0,2}(Q; \RR^n))$ admits a representative in the space $C_{w}([0,h]; W^{k_0, 2}(Q; \RR^n))$. 
\item[$(ii)$] We also remark that the variational inequality \eqref{var-ineq} is preserved along sequences of test functions $\{\varphi_k\}_k$ such that $\varphi_k \rightharpoonup \varphi$ in $L^2((0, h); W^{k_0,2}(Q; \RR^n))$. 
\item[$(iii)$] The inequality in \eqref{var-ineq} is a consequence of the injectivity constraint imposed on the class of admissible deformations. In particular, it is evident from the fact that $C^{\infty}_c(Q; \RR^n) \subset \Adm{\eta}{X}$ that if $\varphi(t, \cdot) \in C^{\infty}_c(Q; \RR^n)$ then \eqref{var-ineq} holds as an equality. Similarly, equality holds if collisions can be excluded a priori. 
\end{itemize}
\end{rmk}

The existence of solutions to \eqref{aux-Pb} is established in the following theorem.

\begin{thm}
\label{AuxExistence}
Let $E$ satisfy \eqref{A1}--\eqref{A5}, $R$ satisfy \eqref{B1}, \eqref{B2}, \eqref{B3q} as well as \eqref{B4} and let $h$, $\eta_0$, and $f$ be given as above. Then there exists a solution $\eta$ to \eqref{aux-Pb} in the sense of \Cref{sol-def-aux}. Furthermore, for every $t \in [0,h]$, $\eta$ satisfies the energy inequality
\begin{equation}
\label{EI}
E(\eta(t)) + 2\int_0^t R(\eta(s), \partial_t \eta(s))\,ds \le E(\eta_0) + \int_0^t \langle f(s), \partial_t \eta(s)\rangle_{L^2} \,ds.
\end{equation}
\end{thm}

\begin{proof}
For the convenience of the reader, we divide the proof into several steps.
\newline
\emph{Step 1:} Given $M \in \NN$, we set $\tau \coloneqq h/M$ and decompose $[0, h]$ into subintervals $[k \tau, (k + 1)\tau]$ of length $\tau$. Moreover, for every $1\leq k\leq M$ we let 
\[
f_k \coloneqq \frac{1}{\tau}\int_{(k - 1)\tau}^{k \tau}f(t)\,dt \in L^2(Q; \RR^n).
\]
We then define recursively $\eta_k \in \E \cap W^{k_0, 2}(Q; \RR^n)$ to be a solution of the minimization problem for
\begin{equation}
\label{Jk-def}
\J_k(\eta) \coloneqq E(\eta) + \tau R\left(\eta_{k - 1}, \tfrac{\eta - \eta_{k - 1}}{\tau}\right) - \tau \left \langle f_k, \tfrac{\eta - \eta_{k-1}}{\tau} \right\rangle_{L^2},
\end{equation}
that is, 
\begin{equation}
\label{etak-def}
\eta_k \in \arg\min \left\{\J_k(\eta) : \eta \in \E \cap W^{k_0, 2}(Q; \RR^n) \right\}.
\end{equation}
We remark that the existence of $\eta_k$ is a consequence of the direct method in the calculus of variations. Indeed \eqref{A3} and \eqref{B3q} imply that $\J_k$ is coercive and the rest follows from \eqref{A1}, \eqref{A2}, \eqref{A4} and \eqref{B1}. Therefore, in view of \eqref{etak-def} we have that $\J_k(\eta_k) \le \J_k(\eta_{k - 1})$; moreover, since $R(\eta, 0) = 0$ for all $\eta$ (see \eqref{B2}), the inequality can be rewritten as
\begin{equation}
\label{etak-min}
E(\eta_k) + \tau R\left(\eta_{k - 1}, \tfrac{\eta_k - \eta_{k - 1}}{\tau}\right) \le E(\eta_{k - 1}) + \tau \left \langle f_k, \tfrac{ \eta_{k} - \eta_{k - 1}}{\tau} \right \rangle_{L^2}.
\end{equation}
Additionally, \eqref{B3q} implies that 
\begin{equation}
\label{P+K}
\frac{K_R}{\tau} \|\eta_k - \eta_{k - 1}\|_{L^2}^2 \le \tau R \left(\eta_{k - 1}, \tfrac{\eta_k - \eta_{k - 1}}{\tau}\right).
\end{equation}
Notice that by Young's inequality we have that 
\begin{equation}
\label{PP}
\tau \left \langle f_k, \tfrac{ \eta_{k} - \eta_{k - 1}}{\tau} \right \rangle_{L^2} \le \frac{\tau}{2 c} \|f_k\|_{L^2}^2 + \frac{c}{2 \tau} \|\eta_k - \eta_{k - 1}\|_{L^2}^2.
\end{equation}
Thus, combining \eqref{P+K} and \eqref{PP} with \eqref{etak-min} yields
\begin{equation}
\label{etak-min-CS}
E(\eta_k) + \frac{\tau}{2} R\left(\eta_{k - 1}, \tfrac{\eta_k - \eta_{k - 1}}{\tau}\right) \le E(\eta_{k - 1}) + \frac{\tau}{2 c} \|f_k\|_{L^2}^2.
\end{equation}
Summing over $k = 1, \dots, m$ in \eqref{etak-min-CS}, where $m \le M$, we arrive at
\begin{equation}
\label{UB-tau}
E(\eta_m) + \sum_{k = 1}^m \frac{\tau}{2} R\left(\eta_{k - 1}, \tfrac{\eta_k - \eta_{k - 1}}{\tau}\right) \le E(\eta_0) + \frac{1}{2c}\sum_{k = 1}^m \tau \|f_k\|_{L^2}^2.
\end{equation}
Notice that an application of Jensen's inequality yields
\begin{equation}
\label{fL2L2}
\sum_{k = 1}^m \tau \|f_k\|_{L^2}^2 \le \sum_{k = 1}^m \int_{(k - 1) \tau}^{k \tau}\|f(t)\|_{L^2}^2\,dt \le \|f\|_{L^2(L^2)}^2.
\end{equation}
Combining \eqref{UB-tau} and \eqref{fL2L2} we can find a constant $C_0$ which depends only on $Q$, $K_{R}$, $E(\eta_0)$ and $\|f\|_{L^2(L^2)}$ such that the following estimates hold uniformly in $\tau$:
\begin{align}
\sup_{k \le M} E(\eta_k)  & \le C_0, \label{ape-1} \\
\sum_{k = 1}^M \tau \left\| \frac{\eta_k - \eta_{k - 1}}{\tau} \right\|_{L^2}^2 & \le C_0, \label{ape-2} \\
\sum_{k = 1}^M \tau R\left(\eta_{k - 1}, \tfrac{\eta_k - \eta_{k - 1}}{\tau}\right) & \le C_0. \label{ape-3}
\end{align}
Next, for $t \in [(k - 1) \tau, k \tau)$ we let
\[
\underline{\eta}_{\tau}(t) \coloneqq \eta_{k - 1}, \qquad \overline{\eta}_{\tau}(t) \coloneqq \eta_k, \qquad \text{ and }  \qquad \eta_{\tau}(t) \coloneqq \frac{k \tau - t}{\tau} \eta_{k - 1} + \frac{t - (k - 1) \tau}{\tau}\eta_k.
\]
Notice that due to the coercivity of $E$, the families of functions $\{\underline{\eta}_{\tau}\}_{\tau}$, $\{\overline{\eta}_{\tau}\}_{\tau}$, and $\{\eta_{\tau}\}_{\tau}$ are uniformly bounded in  $L^{\infty}((0, h); W^{k_0,2}(Q; \RR^n))$. Moreover, from \eqref{ape-2} and \eqref{ape-3} we also obtain that 
\begin{align*}
\int_0^h \|\partial_t \eta_{\tau}(t)\|_{L^2}^2\,dt & \le C_0, \\
K_{R}\int_0^h \|\nabla^{k_0} \partial_t \eta_{\tau}(t)\|_{L^2}^2\,dt & \le C_0,
\end{align*}
where $K_{R}$ is the constant in the Korn's inequality for $R$ (see \eqref{B3q}), which is bounded thanks to \eqref{ape-1}. Reasoning as above we have that $\partial_t \eta_{\tau}$ is bounded in $L^2((0, h); W^{k_0,2}(Q; \RR^n))$. Consequently, $\{\eta_{\tau}\}_{\tau}$ is bounded in $W^{1,2}((0, h); W^{k_0,2}(Q; \RR^n))$, and eventually extracting a subsequence (which we do not relabel) we find $\underline{\eta}$, $\overline{\eta} \in L^{\infty}((0, h); W^{k_0, 2}(Q; \RR^n))$ and $\eta \in W^{1,2}((0, h); W^{k_0, 2}(Q; \RR^n))$ such that
\[
\arraycolsep=1.4pt\def\arraystretch{1.6}
\begin{array}{rll}
\underline{\eta}_{\tau} & \overset{\ast}{\rightharpoonup} \underline{\eta} & \text{ in } L^{\infty}((0, h); W^{k_0, 2}(Q; \RR^n)), \\
\overline{\eta}_{\tau} & \overset{\ast}{\rightharpoonup} \overline{\eta} & \text{ in } L^{\infty}((0, h); W^{k_0, 2}(Q; \RR^n)),
\end{array}
\]
and
\begin{equation}
\label{weakH1Hk}
\eta_{\tau} \rightharpoonup \eta \ \ \text{ in } W^{1,2}((0, h); W^{k_0,2}(Q; \RR^n)).
\end{equation}
Next, we claim that $\underline{\eta} = \overline{\eta} = \eta$. Indeed, by \eqref{P+K} and \eqref{ape-3} we see that 
\begin{equation}
\label{holder}
\|\eta_k - \eta_{k - 1}\|_{W^{k_0,2}}^2 \le \frac{C_0 \tau}{c}.
\end{equation}
In turn, we have that for $t \in [(k - 1) \tau, k \tau)$
\begin{equation}
\label{C1/2}
\|\eta_{\tau}(t) - \overline{\eta}_{\tau}(t)\|_{W^{k_0,2}}^2 = \left\|\frac{k \tau - t}{\tau}(\eta_{k - 1} - \eta_k)\right\|_{W^{k_0,2}}^2 \le \frac{C_0 \tau}{c}.
\end{equation}
To prove the claim, it is enough to notice that in view of \eqref{C1/2},
\[
\left|\int_0^h \langle \varphi(t), \eta_\tau(t) - \overline \eta_\tau(t) \rangle \,dt \right| \le \lim_{\tau \to 0} \sqrt{\frac{C_0 \tau}{c}} \int_0^h \|\varphi(t)\|_{(W^{k_0, 2})^*}\,dt = 0
\]
holds for all $\varphi \in L^1((0, h); W^{k_0, 2}(Q; \RR^n)^*)$.

We conclude this step by proving that $\eta \in L^{\infty}((0, h); \E \cap W^{k_0, 2}(Q; \RR^n))$. The result, in turn, is obtained by showing that  the sequence $\{\overline{\eta}_{\tau}\}_{\tau}$ converges to $\eta$ uniformly; to be precise, we will show that (up to the extraction of a subsequence)
\begin{equation}
\label{unif.conv}
\overline{\eta}_{\tau} \to \eta \ \ \text{ in } L^{\infty}((0, h); W^{2, p}(Q; \RR^n)).
\end{equation}
Notice that this readily implies the claim since $\eta \in L^{\infty}((0, h); W^{k_0, 2}(Q; \RR^n))$ and furthermore since $\E$ is closed with respect to weak convergence in $W^{2, p}(Q; \RR^n)$. To prove \eqref{unif.conv}, we argue as follows: By \eqref{weakH1Hk}, together with an application of the Aubin--Lions lemma we see that $\eta_{\tau} \to \eta$ in $C([0,h]; W^{2, p}(Q; \RR^n))$. Therefore, \eqref{C1/2} implies that 
\begin{align*}
\|\eta(t) - \overline{\eta}_{\tau}(t)\|_{W^{2, p}} & \le \|\eta(t) - \eta_{\tau}(t)\|_{W^{2, p}} + \|\eta_{\tau}(t) - \overline{\eta}_{\tau}(t)\|_{W^{2, p}} \\
& \le \|\eta - \eta_{\tau}\|_{L^{\infty}(W^{2, p})} + C\|\eta_{\tau}(t) - \overline{\eta}_{\tau}(t)\|_{W^{k_0, 2}} \\
& \le \|\eta - \eta_{\tau}\|_{L^{\infty}(W^{2, p})} + C\sqrt{\frac{C_0 \tau}{c}}
\end{align*}
holds for $\mathcal{L}^1$-a.e.\@ $t \in [0,h]$. Letting $\tau \to 0$ concludes the proof of \eqref{unif.conv}.
\newline
\emph{Step 2:} Next, we show that the function $\eta$ obtained in the previous step is a solution to \eqref{aux-Pb} in the sense of \Cref{sol-def-aux}. To this end, consider $\varphi \in W^{k_0,2}(Q;\RR^n)$ such that there is a curve $\Phi \in C^0([0,\e_0); \E \cap W^{k_0, 2}(Q; \RR^n)) \cap C^1([0, \e_0); W^{k_0, 2}(Q;\RR^n))$ with $\Phi'(0^+) = \varphi$, as in \Cref{T-char} $(iv)$. By the minimality of $\eta_k$ we obtain that
\begin{align}
 0 & \le \left.\frac{d}{d\e} \J_k(\Phi(\e))\right|_{\e = 0} 
  = \left[ D E (\eta_k) + D_2 R \left(\eta_{k-1}, \tfrac{\eta_{k}-\eta_{k-1}}{\tau}\right)\right]\langle \varphi \rangle - \langle f_k, \varphi \rangle_{L^2}. \label{VI_k}
\end{align}
By \Cref{T-char} $(ii)$, the same is then also true for all $\varphi \in \Adm{\eta}{W^{k_0,2}}$.

Let $\{\tau_j\}_{j \in \NN} \subset (0, 1)$ with $\tau_j \to 0^+$ be a decreasing subsequence for which \eqref{weakH1Hk} and \eqref{unif.conv} hold, and consider $\varphi \in C([0, h]; \Adm{\eta}{ W^{k_0,2})}$. By \Cref{prop:TErecoveryGlobal}, there exists a sequence $\{\varphi_j\}_{j \in \NN}$ such that $\varphi_j \to \varphi$ in $C([0, h]; W^{k_0, 2})$ and $\varphi_j \in C([0, h]; \Adm{\overline{\eta}_{j}}{W^{k_0, 2}})$. In particular, since this implies that $\varphi_j(t) \in \Adm{\eta_k}{W^{k_0, 2}}$ for $t \in [(k - 1)\tau_j,k \tau_j)$, for $\mathcal{L}^1$-a.e.\@ such $t$ we can rewrite \eqref{VI_k} as
\begin{equation}
\label{VI_tau}
 \left[ D E (\overline{\eta}_{\tau_j}(t)) + D_2 R(\underline{\eta}_{\tau_j}(t), \partial_t \eta_{\tau_j}(t))\right]\langle \varphi_j(t) \rangle  -  \left \langle   f_{\tau_j}(t), \varphi_j(t) \right \rangle_{L^2} \ge 0,
\end{equation}
where $f_{\tau_j}(t) \coloneqq f_k$ for all $t \in [(k - 1)\tau_j, k \tau_j)$. Integrating \eqref{VI_tau} over $[(k - 1) \tau_j, k \tau_j)$ and summing up the resulting inequalities we get
\begin{equation}
\label{EL-ineq}
\int_0^h \left[ DE(\overline{\eta}_{\tau_j}(t)) + D_2R(\underline{\eta}_{\tau_j}(t), \partial_t \eta_{\tau_j}(t))\right] \langle \varphi_j(t)\rangle\,dt \ge \int_0^h \langle f_{\tau_j}(t), \varphi_j(t) \rangle_{L^2}\,dt.
\end{equation}
Notice that $f_{\tau_j}(t) \to f(t)$ for $\mathcal{L}^1$-a.e.\@ $t \in [0,h]$ (to be precise, for every $t$ which is a Lebesgue point). Then by Lebesgue's dominated convergence theorem we also have that $f_{\tau_j} \to f$ in $L^2((0, h); L^2(Q; \RR^n))$. In turn, we can pass to the limit pass to the limit as $\tau_j \to 0^+$ on the right-hand side of \eqref{EL-ineq}. Convergence for the terms on the left-hand side is a consequence of \eqref{A5} and \eqref{B4}. Thus, we have shown that $\eta$ satisfies \eqref{var-ineq} for all $\varphi \in C([0,h];\Adm{\eta}{ W^{k_0,2}})$.
\newline
\emph{Step 3}: As it remains to verify that $\eta$ satisfies the energy inequality \eqref{EI}, the goal of this step is to show that $-\partial_t\eta$ is an admissible test function. 
For $t \in (\e, h - \e)$, let
\[
\partial^\e_t \eta(t) \coloneqq \frac{\eta(t - \e) - \eta(t)}{\e}.
\]
Let $\{\chi_{\e}\}_{\e}$ be a family of smooth cut-off functions with the property that $\chi_{\e} \in C_c((\e, h - \e); [0, 1])$ and $\chi_{\e} \to 1$ in $L^2((0, h))$. Then, $\chi_{\e}\partial^\e_t \eta \to - \partial_t \eta$ in $L^2((0, h); W^{k_0, 2}(Q; \RR^n))$. Fix $(t, x) \in C_{\eta}$. We claim that 
\[
\lim_{\e \to 0^+} \sum_{z\in\eta^{-1}(t,\eta(t,x))} \partial^\e_t \eta(t, z) \cdot n_\eta(t, x) \ge 0.
\]
Indeed, if this is not the case then we can find a sequence $\e_j \to 0^+$ and a number $\delta > 0$ such that $\sum_{z\in\eta^{-1}(t,\eta(t,x))}\partial^{\e_j}_t\eta(t, z) \cdot n_\eta(t, z) \leq - \delta$ for all $j$. In turn, by \Cref{bad-displacement}, we must have that 
\[
\eta(t - \e_j, \cdot) = \eta(t, \cdot) + \e_j \partial^{\e_j}_t\eta(t, \cdot) \notin \E
\]
for all $j$ sufficiently large. We thus reached a contradiction and the claim is proved. Now let 
\[
\varphi_{\e, \delta} \coloneqq \chi_{\e}\partial^\e_t \eta + \delta \xi_{\Gamma} \tilde{n}_\eta,
\]
where $\xi_{\Gamma} \colon \overline{Q} \to \RR$ is a smooth function that vanishes on $\Gamma$ and is otherwise positive. Then, for all small values of $\delta > 0$, in view of \Cref{time-TF} we have that $\varphi_{\e, \delta} \in C([0,h];\Adm{\eta}{W^{k_0,2}})$ for all $\e$ sufficiently small. The desired result follows by testing with $\varphi_{\e, \delta}$ in \eqref{var-ineq}, letting $\e \to 0^+$ first, and then $\delta \to 0^+$. A similar argument shows that $- \chi_{[0, t]}\partial_t \eta$ is an admissible test function.
\newline
\emph{Step 4:} In view of the previous step, we can substitute $\varphi = -\chi_{[0, t]}\partial_t \eta$ in the variational inequality \eqref{var-ineq} and obtain that 
\[
\int_0^t \left[ DE(\eta(s)) + D_2 R (\eta(s), \partial_t \eta(s))\right]\langle  \partial_t \eta(s)\rangle\,ds \le \int_0^t \langle f, \partial_t \eta(s)\rangle_{L^2}\,ds.
\]
Using the fundamental theorem of calculus for the term involving the elastic energy and \eqref{B2}, the previous inequality can be rewritten as
\[
E(\eta(t)) - E(\eta(0)) + 2 \int_0^t R(\eta(s),\partial_t \eta(s))\,ds \le \int_0^t \langle f(s), \partial_t \eta(s)\rangle_{L^2} \,ds.
\]
This concludes the proof.
\end{proof}

As a corollary, we show the existence of solutions to \eqref{aux-par}.

\begin{cor}
\label{par-approx-cor}
Let $E$ and $\eta_0$ be given as above and assume that for some $h>0$
\begin{equation}
\label{R-par}
\tilde{R}(\eta, b) \coloneqq R(\eta, b) + \frac{\rho}{2h}\|b\|_{L^2}^2,
\end{equation}
where $R$ now satisfies \eqref{B1}--\eqref{B4}.
Define additionally $\tilde{f} \coloneqq f + \rho\frac{\zeta}{h}$ for some $\zeta \in L^2((0, h); L^2(Q;\RR^n))$. Then the conclusions of \Cref{AuxExistence} continue to hold. In particular, this yields the existence of a weak solution to \eqref{aux-par} (in the sense of \Cref{sol-def-aux}), where \eqref{var-ineq} can be rewritten as
\begin{equation}
\label{var-ineq-par}
\int_0^h  [DE(\eta(t)) + D_2\tilde{R}(\eta(t), \partial_t \eta(t))] \langle \varphi(t)\rangle\,dt + \int_0^h \left\langle \rho \frac{\partial_t \eta(t)-\zeta}{h}, \varphi(t)\right\rangle_{L^2} \,dt \le \int_0^h \langle f(t), \varphi(t) \rangle_{L^2}\,dt.
\end{equation}
Furthermore, for every $t \in [0,h]$, $\eta$ satisfies the energy inequality
\begin{multline}
\label{EI-par}
E(\eta(t)) + 2 \int_0^t \tilde{R}(\eta(s), \partial_t \eta(s))\,ds + \frac{\rho}{2h}\int_0^t\|\partial_t \eta(s)\|_{L^2}^2\,ds \\ \le E(\eta_0) + \int_0^t \langle f(s), \partial_t \eta(s)\rangle_{L^2}^2\,ds+ \frac{\rho}{2h}\int_0^t  \norm[L^2]{\zeta}^2 ds.
\end{multline}
\end{cor}
\begin{proof}
We note that since $R$ satisfies \eqref{B3} then \eqref{B3q} holds for $\tilde{R}$. Similarly, since \eqref{B1}, \eqref{B2} and \eqref{B4} hold for $R$, they also hold for $\tilde{R}$. As a result, we can apply \Cref{AuxExistence}.

What is left is the energy inequality. For this we note that rewriting the energy inequality for $\tilde{R}$ and $\tilde{f}$ in terms of $R$ and $f$ and using Young's inequality yields
\begin{align*}
 &\phantom{{}={}}E(\eta(t)) - E(\eta(0)) + 2 \int_0^t R(\eta(s),\partial_t \eta(s))+ \frac{\rho}{h} \norm[L^2]{\partial_t \eta(s)} \,ds   \le \int_0^t \left\langle f(s)+ \rho\tfrac{\zeta}{h}, \partial_t \eta(s)\right\rangle_{L^2} \,ds\\
 &\leq \int_0^t \left[\left\langle f(s), \partial_t \eta(s)\right\rangle_{L^2} + \frac{\rho}{2h}\norm[L^2]{\zeta(s)}^2 + \frac{\rho}{2h}\norm[L^2]{\partial_t\eta(s)}^2 \right] ds.
\end{align*}
Rearranging the terms on both sides then results in the claim.
\end{proof}

\subsection{Existence of the contact force}
We now refine the results of the previous section to show that the inequality \eqref{var-ineq} can be characterized by the existence of a contact force.
\begin{definition}[Solution with a contact force]
\label{sol-def-aux-measure}
Let $\eta_0 \in \E \cap W^{k_0, 2}(Q; \RR^n)$ and $f \in L^2((0, h); L^2(Q; \RR^n))$ be given. We say that 
\[
\eta \in W^{1,2}((0, h); W^{k_0,2}(Q; \RR^n)) \cap L^\infty((0, h); \E \cap W^{k_0, 2}(Q; \RR^n)), \quad \sigma \in L_{w^*}^2((0,h);M(\partial Q; \RR^n)),
\]
where $\sigma$ is a contact force for $\eta$ (see Definition \ref{def-contact-force}), is a \emph{solution with a contact force} to \eqref{aux-Pb} if $\eta(0) = \eta_0$ and 
\begin{equation}
\label{eqn:measure-eqn}
\int_0^h \left[DE(\eta(t)) + D_2R(\eta(t), \partial_t \eta(t))\right] \langle \varphi(t)\rangle\,dt = \int_0^h \langle \sigma(t),\varphi(t)\rangle \,dt + \int_0^h \langle f(t), \varphi(t) \rangle_{L^2}\,dt
\end{equation}
for all $\varphi \in C([0,h]; W^{k_0,2}_\Gamma(Q;\RR^n))$.
\end{definition}

\begin{thm}
\label{existence-CF-quasi}
Let $E$ satisfy \eqref{A1}--\eqref{A5}, $R$ satisfy \eqref{B1}, \eqref{B2}, \eqref{B3q} as well as \eqref{B4}. 
Given $\eta_0\in\E\cap W^{k_0,2}(Q;\RR^n)$ and $f \in L^2((0, h); L^2(Q; \RR^n))$, there exists a solution with a contact force to \eqref{aux-Pb}. Moreover, $\eta$ satisfies the energy inequality \eqref{EI} for all $t \in [0, h]$, and the contact force satisfies $\sigma\in L_{w^*}^2((0, h);M(\partial Q;\RR^n))$ with estimate
\begin{equation}\label{cforce-h-estimate}
\int_0^h \|\sigma(t)\|_{M(\partial Q;\RR^n)}^2 dt \leq C_{E_0} \left(h + \|\partial_t\eta\|_{L^2(W^{1,2})}^2 +\|f\|_{L^2(L^2)}^2\right).
\end{equation}
\end{thm}
\begin{rmk}
If we applied the Lagrange multiplier theorem to our constrained minimization problem, we would get the contact ``\,force'' only in the form of a distribution, namely as an element of $(W^{k_0,2}(Q;\RR^n))^*$. An estimate in the same space for the time-continuous, parabolic solution is directly available by using the equation, cf. \cite{kromerQuasistaticViscoelasticitySelfcontact2019}. However \cite{PalmerHealey} proved that in the steady case, one in fact obtains a measure--valued contact force in the sense our definition. We will use their argument here, but in addition to that we employ a quantitative estimate on the norm of the measure to retain the desired regularity when passing to the limit as $\tau\to 0$.\end{rmk}
\begin{proof}
We use the approximation $\eta_\tau$ resp.\@ $\overline\eta_\tau$ and all other relevant notation from the proof of \Cref{AuxExistence}. Fix $k\in \{1,\dots,M \}$ and recall that we have that
\begin{equation*}
D\mathcal{J}_k(\eta_k) = DE(\eta_k)+D_2 R\left(\eta_{k-1},\tfrac{\eta_k-\eta_{k-1}}{\tau}\right)-f_k \in (W^{k_0,2}(Q;\RR^n))^*.
\end{equation*}
\newline\emph{Step 1: } By the assumptions on $\eta_0|_\Gamma$ (see \Cref{rmk:DirichletBoundary}) and the local injectivity, we can pick a compact subset $K\subset \partial Q$ such that $\eta(t)|_{\partial Q \setminus K}$ is always injective. Then by the action-reaction principle, estimating $\|\sigma_k \|_{M(K;\RR^n)}$ is enough to estimate $\|\sigma_k\|_{M(\partial Q;\RR^n)}$. Without loss of generality, we can choose $\xi_\Gamma \in C^\infty_\Gamma(Q;[0,1])$ such that $\xi_\Gamma(x) = 1$ for all $x \in K$.
We will now follow\footnote{To avoid confusion, we note that since we use interior normals, some signs and inequalities have to be turned around.} \cite{PalmerHealey} to get the existence of a contact force for the auxiliary problem.
Let us denote
\begin{align*}
M^+_{\eta_k}& \coloneqq \Biggr\{ (\ell,w)\in\RR \times C(\partial Q): \exists \varphi \in W^{k_0,2}_{\Gamma}(Q,\RR^n) \text{ s.t. } D\mathcal{J}_k(\eta_k) \langle\varphi \rangle \leq \ell \\ & \qquad \text{ and } \forall x \in C_{\eta_k} \text{ we have } \sum_{z \in \eta^{-1}(\eta(x))} \varphi(z)\cdot n_{\eta_k}(z) \geq \sum_{z\in \eta^{-1}(\eta(x))} w(z) \Biggr\}, \\
M^- & \coloneqq \{(\ell,w)\in\RR\times C(\partial Q): \ell\leq 0, w\geq 0\}.
\end{align*}
Now, if $\ell<0$ and $w>0$, then any $\varphi \in W_{\Gamma}^{k_0,2}(Q;\RR^n)$ satisfying 
\begin{equation*}
\sum_{z\in \eta^{-1}(\eta(z)))} \varphi(z)\cdot n_{\eta_k}(z)  \geq \sum_{z\in \eta^{-1}(\eta(x))} w(z)>0
\end{equation*}
for all $x \in C_\eta$ is an admissible test function by \Cref{T-char}. Hence $\varphi$ satisfies (by the previous proof) the variational inequality
\begin{equation*}
 D\mathcal{J}_k(\eta_k)\langle\varphi \rangle \geq 0 > \ell,
\end{equation*}
and thus $(\ell,w)\notin M^+_{\eta_k}$. This means that $M^+_{\eta_k} \cap \operatorname{int} M^- =\emptyset$. Since both sets are convex, we can find a separating hyperplane. That is, there exists $(\lambda_0,\sigma_0)\in(\RR\times C(\partial Q))^* = \RR\times M(\partial Q)$ with $(\lambda_0, \sigma_0) \neq (0,0)$ such that
\begin{align}
\lambda_0 \ell + \langle \sigma_0, w\rangle &\leq 0, \quad (\ell,w)\in M^+_{\eta_k}, \label{M+ineq} \\
\lambda_0 \ell + \langle \sigma_0, w \rangle &\geq 0, \quad (\ell,w)\in M^-. \notag
\end{align}
The latter inequality shows that for every $w\geq 0$, since $(0, w)\in M^-$, we have that $\langle \sigma_0,w \rangle \geq 0$. Hence the measure $\sigma_0 \in M^+(\partial Q)$ is non-negative. Furthermore, for any $w\in C(\partial Q)$ with $w\geq 0$ and $\supp w\cap C_{\eta_k}=\emptyset$ we see immediately that $(0,w)\in M_{\eta_k}^+\cap M^-$ and thus $\langle \sigma_0, w\rangle =0$. This proves that $\supp \sigma_0\subset C_{\eta_k}$.

Now, given $\varphi\in W^{k_0,2}_\Gamma(Q;\RR^n)$ we have that $( D\mathcal{J}_k(\eta_k)\langle\varphi\rangle, n_{\eta_k}\cdot \varphi )\in M^+_{\eta_k}$, and so 
\begin{equation*}
\lambda_0 D\mathcal{J}_k(\eta_k)\langle\varphi\rangle + \langle \sigma_0, n_{\eta_k}\cdot \varphi\rangle \leq 0.
\end{equation*}
Repeating the argument with $-\varphi$ in place of $\varphi$ we obtain that
\begin{equation}\label{eqn:AuxMeasurek}
\lambda_0 D\mathcal{J}_k(\eta_k)\langle\varphi\rangle + \langle \sigma_0, n_{\eta_k}\cdot \varphi\rangle = 0.
\end{equation}

Fix $\delta>0$ and denote $\varphi_\delta =\delta \xi_\Gamma \tilde n_{\eta_k}$, where $\tilde n_{\eta_k}$ is the smooth almost normal from Corollary \ref{smooth-almost-normal-notime}. Then $\varphi_\delta \in C^2_\Gamma(\overline Q; \RR^n)$ and moreover $\|\varphi_\delta \|_{W^{k_0,2}(Q;\RR^n)}\leq \tilde C\delta$, where by the estimate from \Cref{smooth-almost-normal} and the energy estimate \eqref{ape-1}, $\tilde C$ depends only on $C_0=E(\eta_0)+\|f\|_{L^2(L^2)}$. Then, for every $x\in C_{\eta_k}\cap K$, we have that

\begin{equation*}
\sum_{z\in\eta^{-1}(\eta(x))}
\varphi_\delta (z)\cdot n_{\eta_k}(z) = \delta \sum_{z\in\eta^{-1}(\eta(x))} \tilde n_{\eta_k}(z)\cdot n_{\eta_k}(z) \geq \frac{\delta}{2}.
\end{equation*}

This shows that $( D\mathcal{J}_k(\eta_k)\langle\varphi_\delta \rangle, \delta/2 ) \in M^+_{\eta_k}$, where $\delta/2$ is treated as a constant function. Hence by \eqref{M+ineq} we obtain that
\begin{equation}\label{normal-separates}
\lambda_0  D\mathcal{J}_k(\eta_k)\langle \varphi_\delta \rangle +\langle \sigma_0, \delta/2 \rangle \leq 0.
\end{equation}
From this we can see that $\lambda_0<0$. Indeed, if $\lambda_0=0$, then $\sigma_0\neq 0$ and $\langle \sigma_0, \delta/2\rangle >0$, and we thus reach a contradiction; if $\lambda_0>0$, take $\ell \geq  D\mathcal{J}_k(\eta_k)\langle \varphi_\delta \rangle$ with $\ell>0$. Then $(\ell,\delta/2)\in M^+_{\eta_k}$ and $\lambda_0 \ell + \langle \sigma_0, \delta/2\rangle >0$, yielding a contradiction in this case as well.

Therefore, we can define $\sigma_k\in M(\partial Q;\RR^n)$ via $d\sigma_k=-\frac{1}{ \lambda_0}n_{\eta_k}\,d\sigma_0$, which is a contact force for $\eta_k$ since we have verified that $\lambda_0<0$, $\sigma_0\in M^+(\partial Q)$ and $\supp \sigma_0\subset C_{\eta_k}$. 
\newline
\emph{Step 2:}
Now we estimate
\begin{equation*}
\|\sigma_k\|_{M(\partial Q;\RR^n)} \leq 2\|\sigma_k\|_{M(K;\RR^n)} \leq-\frac{2}{\lambda_0}\langle \sigma_0,\xi_\Gamma\rangle \leq \frac{4}{\delta}D\mathcal{J}_k(\eta_k)\langle \varphi_\delta \rangle.
\end{equation*}

Setting $\sigma_\tau(t)\coloneqq\sigma_k$ on $t\in[(k-1)\tau,k \tau)$, we multiply this last inequality by $\psi\in C([0,h]; \RR^+)$ and integrate over $[0,h]$ to obtain
\begin{align*}
\int_0^h \norm[M(\partial Q;\RR^n)]{\sigma_\tau(t)}\psi(t) \,dt & \leq 4\int_0^h  D\mathcal{J}_{\lfloor t/\tau \rfloor}(\overline \eta_\tau(t)) \langle \xi_\Gamma \tilde n_{\overline\eta_\tau}(t) \rangle \psi(t) \,dt  \\
\leq 4\int_0^h &\left( [DE (\overline\eta_\tau(t)) + D_2R(\underline\eta_\tau(t), \partial_t\eta_\tau(t))]\langle \xi_\Gamma \tilde n_{\overline\eta_\tau}(t) \rangle - \langle  f(t), \xi_\Gamma \tilde n_{\overline\eta_\tau} (t)\rangle_{L^2} \right) \psi(t) \,dt. 
\end{align*}
Since by \eqref{A5} $DE$ is bounded on sublevel sets of $E$, in view of \Cref{smooth-almost-normal} we obtain that
\begin{equation*}
\int_0^h DE (\overline\eta_\tau(t))\langle \xi_\Gamma \tilde n_{\overline\eta_\tau}(t) \rangle \psi(t)\,dt \leq C_{k_0} \max_{k} \|DE(\eta_k)\|_{(W^{k_0,2})^*} \int_0^h \psi(t)\,dt.
\end{equation*}
Furthermore, by \eqref{eq:D2Rgrowthbounds} we arrive at
\begin{equation*}
\int_0^h  D_2 R(\underline \eta_\tau (t),\partial_t\eta_\tau(t))\langle \xi_\Gamma \tilde n_{\overline \eta_\tau} (t)\rangle \psi(t) \,dt \leq  C_{E_0} \int _0^h \|\partial_t\eta_\tau\|_{W^{1,2}}\psi(t) \,dt.
\end{equation*}
Finally, since the remaining term satisfies the estimate
\begin{equation*}
\int_0^h \langle f(t), \xi_\Gamma \tilde n_{\overline \eta_\tau}(t) \rangle \psi(t) \,dt \leq C_{E_0}\int_0^h \|f(t)\|_{L^2} \psi(t)\,dt,
\end{equation*}
we conclude that
\begin{equation}\label{cforce-tau-estimate}
\int_0^h \|\sigma_\tau(t)\|_{M(\partial Q;\RR^n)}^2 dt \leq C_{E_0} \left(h + \|\partial_t\eta_\tau\|_{L^2(W^{1,2})}^2 +\|f\|_{L^2(L^2)}^2\right). 
\end{equation}
In view of the energy estimate, the right-hand side in \eqref{cforce-tau-estimate} is bounded by a constant (independent of $\tau$).
This means that we have estimated $\sigma_\tau$ in $L_{w^*}^2((0,h);M(\partial Q;\RR^n))$, so in particular also in $M([0,h]\times \partial Q; \RR^n)$. 
\newline
\emph{Step 3: } Upon dividing by $-\lambda_0$ and summing over $k$, \eqref{eqn:AuxMeasurek} can be rewritten as
\begin{equation*}
\int_0^h D\mathcal{J}_{\lfloor t/\tau \rfloor}(\overline \eta_\tau(t)) \langle \varphi\rangle - \langle \sigma_\tau(t), \varphi\rangle \,dt = 0.
\end{equation*}
By compactness of contact forces (see \Cref{compactness-cforce}) we can find a converging subsequence (which we do not relabel) such that
\begin{equation*}
\sigma_\tau \weakstarto \sigma \quad \text{ in } M([0,h] \times \partial Q;\RR^n), 
\end{equation*}
where $\sigma$ is a contact force for $\eta$. By this convergence, upon sending $\tau \to 0$, the equation \eqref{eqn:measure-eqn} holds for all $\varphi\in W_{\Gamma}^{k_0,2}(Q;\RR^n)$, and by a density argument for all $\varphi\in C([0,h],W_{\Gamma}^{k_0,2}(Q;\RR^n))$.
Moreover, thanks to \eqref{cforce-tau-estimate}, we are in a position to apply \Cref{cforce-l2intime} to conclude that $\sigma\in L_{w^*}^2((0, h);M(\partial Q;\RR^n))$ and that it satisfies \eqref{cforce-h-estimate}. This concludes the proof.
\end{proof}

We now state a version of \Cref{par-approx-cor} that includes contact forces.

\begin{cor}
\label{par-cforce-cor}
Let $E$ and $\eta_0$ be given as above and assume that for some $h>0$
\begin{equation}
\tilde{R}(\eta, b) \coloneqq R(\eta, b) + \frac{\rho}{2h}\|b\|_{L^2}^2,
\end{equation}
where $R$ satisfies \eqref{B1}--\eqref{B4}.
Define additionally $\tilde{f} \coloneqq f + \rho\frac{\zeta}{h}$ for some $\zeta \in L^2((0,h); L^2( Q;\RR^n))$. Then the conclusions of \Cref{existence-CF-quasi} continue to hold. In particular, this yields the existence of a weak solution with a contact force to \eqref{aux-par} (in the sense of \Cref{sol-def-aux-measure}), where \eqref{eqn:measure-eqn} can be rewritten as
\begin{multline}
\label{weak-cor}
\int_0^h \left[DE(\eta(t)) + D_2R(\eta(t), \partial_t \eta(t))\right] \langle \varphi(t)\rangle\,dt \\ + \int_0^h\rho \left \langle \frac{\partial_t \eta(t)-\zeta(t)}{h}, \varphi(t) \right \rangle_{L^2} \,dt = \int_0^h \langle \sigma(t), \varphi(t)\rangle \, dt + \int_0^h \langle f(t), \varphi(t) \rangle_{L^2}\,dt.
\end{multline}
Furthermore, for every $t \in [0,h]$, $\eta$ satisfies the energy inequality in \eqref{EI-par}
and the contact force satisfies $\sigma \in L_{w^*}^2((0, h);M(\partial Q;\RR^n))$ with
\begin{equation}
\int_0^h \|\sigma(t)\|_{M(\partial Q;\RR^n)}^2 dt \leq C_{E_0}\left(h + \frac{\|\zeta\|_{L^2(W^{1,2})}^2}{h} + \|f\|_{L^2(L^2)}^2\right).
\end{equation}
\end{cor}
\begin{proof}
The result is an immediate consequence of \Cref{existence-CF-quasi} and \Cref{par-approx-cor}, thus we omit the details.
\end{proof}

\section{Proof of the main theorem}
\label{inertia-sec}
In this section we present the proof of \Cref{main-thm-VI}. Additionally we conclude the section by returning to the corresponding quasistatic problem. In particular, we present an improvement of the existence result of a solution with contact force that was previously obtained by Kr\"{o}mer and Roub\'{i}\v{c}ek in \cite{kromerQuasistaticViscoelasticitySelfcontact2019}. 

\begin{proof}[Proof of \Cref{main-thm-VI}] We divide the proof into several steps.
\newline
\emph{Step 1:} Given $L \in \NN$, we set $h \coloneqq T/L$ and decompose $[0, T]$ into subintervals $[(j - 1)h, jh]$, with $j \ge 1$, of length $h$. Let $E^{(h)} \colon W^{k_0, 2}(Q; \RR^n) \to (-\infty, \infty]$ be defined via
\begin{equation}
\label{Eh}
E^{(h)}(\eta) \coloneqq E(\eta) + h^{a_0} \|\nabla^{k_0} \eta\|_{L^2}^2,
\end{equation}
where $k_0 - \frac{n}{2}  > 2 - \frac{n}{p}$ and $a_0 > 0$ is to be determined. Let $\eta_0$ be given as in the statement. In this step we prove that there exists a sequence $\{\eta_0^{(h)}\}_h \subset W^{k_0, 2}(Q; \RR^n)$ such that $\eta_0^{(h)} \to \eta_0$ in $W^{2, p}(Q; \RR^n)$ as $h \to 0$, and for all but finitely many values of $h$ we have that $\eta_0^{(h)}(Q) \subset \Omega$ and $\eta_0^{(h)}$ is a.e.\@ globally injective on $\overline{Q}$. Furthermore, we show that this sequence can be chosen with the property that 
\begin{equation}
\label{initial-rec}
E^{(h)}(\eta_0^{(h)}) \to E(\eta_0).
\end{equation}
We present the proof for the case $\eta_0 \in \partial \E$. Indeed, if $\eta_0 \in \operatorname{int}\E$ the proof follows from a similar but simpler argument. We begin by considering perturbations of the initial condition of the form $\eta_0 + h^{b_0} \tilde{n}_{\eta_0}$, with $b_0 > 0$, where $\tilde{n}_{\eta_0}$ is given as in \Cref{smooth-almost-normal-notime}. Reasoning as in \Cref{T-char} (see also Proposition 3 in \cite{PalmerHealey}), if $h$ is sufficiently small we have that $(\eta_0 + h^{b_0} \tilde{n}_{\eta_0})(Q) \subset \Omega$ and furthermore $\eta_0 + h^{b_0} \tilde{n}_{\eta_0}$ is a.e.\@ globally injective on $\overline{Q}$. Recall that $\eta_0 + h^{b_0} \tilde{n}_{\eta_0}$ is injective in $B_r(x) \cap \overline{Q}$ for all $x$ provided that $r$ is given as in \Cref{IFT}. Moreover, if $h$ is sufficiently small, for $x, y$ with $|x - y| \ge r$ we have that 
\[
| \eta_0(x) + h^{b_0} \tilde{n}_{\eta_0}(x) - \eta_0(y) - h^{b_0} \tilde{n}_{\eta_0}(y)| \ge ch^{b_0} 
\]
for some $c > 0$. Similarly, $\dist(\eta_0(x) + h^{b_0} \tilde{n}_{\eta_0}(x), \partial \Omega) \ge ch^{b_0}$ for all $x \in \overline{Q}$. This implies that there exist a smooth open set $Q'$ such that $Q \subset Q'$ and $\dist(Q, \partial Q') \ge c h^{b_0}$, and an extension of $\eta_0 + h^{b_0} \tilde{n}_{\eta_0}$ to $Q'$ (not relabeled) with the property that $(\eta_0 + h^{b_0} \tilde{n}_{\eta_0})(Q') \subset \Omega$ and such that $\eta_0 + h^{b_0} \tilde{n}_{\eta_0}$ is a.e.\@ globally injective on $\overline{Q'}$. Fix $b_1 > b_0$ and consider 
\begin{equation}
\label{moll-def}
\xi_h(x) \coloneqq \frac{1}{h^{n b_1}}\xi \left(\frac{x}{h^{b_1}}\right),
\end{equation}
where $\xi$ is the standard mollifier. We then define 
\begin{equation}
\label{eta-0h-def}
\eta_0^{(h)}(x) \coloneqq [(\eta_0 + h^{b_0} \tilde{n}_{\eta_0}) \ast \xi_h](x)
\end{equation}
and claim that $\eta_0^{(h)}$ has the desired properties for all $h$ sufficiently small. Since $\eta_0^{(h)} \to \eta_0$ in $W^{2, p}(Q; \RR^n)$, \eqref{E4} implies that $E(\eta_0^{(h)}) \to E(\eta_0)$. Moreover, as one can readily check, we have that 
\begin{equation}
\label{moll-eta0}
\|\nabla^{k_0} \eta_0^{(h)}\|_{L^2} \le C h^{b_1 (2 - k_0)} (\|\eta_0\|_{W^{2, p}} + h^{b_0}\|\tilde{n}_{\eta_0}\|_{C^2}).
\end{equation}
Consequently, if $a_0 > b_1(k_0 - 2)$, by \eqref{moll-eta0} we obtain that 
\[
h^{a_0}\|\nabla^{k_0} \eta_0^{(h)}\|_{L^2} \to 0
\]
as $h \to 0$, thus proving \eqref{initial-rec}. 
\newline
\emph{Step 2:} Next, we consider the time-delayed equation
\begin{equation}
\label{RTD}
\rho\frac{\partial_t \eta(t) - \partial_t \eta(t-h)}{h} + DE^{(h)}(\eta(t))+D_2 R^{(h)}(\eta(t), \partial_t\eta(t)) = f(t),
\end{equation}
on the interval $[0, h]$, where $E^{(h)}$ is defined as in \eqref{Eh} and
\[
R^{(h)}(\eta, b) \coloneqq R(\eta, b) + h \|\nabla^{k_0} b \|_{L^2}^2.
\]
Problem \eqref{RTD} is complemented by the initial condition $\eta(0) = \eta_0^{(h)}$, with $\eta_0^{(h)}$ given as in the previous step. Moreover, we define $\partial_t \eta(t) \coloneqq \eta^*$ for all $t < 0$. Notice that \eqref{RTD} can be recast as 
\begin{equation}
\label{tdp}
D\tilde{E}(\eta)+ D_2 \tilde{R}(\eta, \partial_t\eta) = \tilde{f},
\end{equation}
where 
\begin{equation}
\label{ERf-def}
\tilde{E}(\eta) \coloneqq E^{(h)}(\eta), \quad \tilde{R}(\eta, b) \coloneqq R^{(h)}(\eta, b) + \frac{1}{2h}\|b\|^2_{L^2}, \quad \tilde{f}(t) \coloneqq f(t) + \frac{\rho}{h}\partial_t \eta(t - h).
\end{equation}
As one can readily check, the energy-dissipation pair $(\tilde{E}, \tilde{R})$ satisfies the assumptions of \Cref{Aux-sec}. Thus, an application of \Cref{par-approx-cor} yields the existence of a solution to \eqref{tdp} in $[0, h]$, denoted $\eta^{(h)}_1$. For $j \ge 2$, we can then iteratively construct solutions to \eqref{tdp} in each of the intervals $[(j-1)h, jh]$, with the difference, however, that the initial condition for $\eta$ and the time-shifted derivative $\partial_t \eta(t - h)$ are given by the solution in the previous time interval. To be precise, $\eta^{(h)}_j$ is a solution to \eqref{tdp} on the interval $[(j - 1)h, jh]$ with initial condition $\eta^{(h)}_j((j-1)h) = \eta^{(h)}_{j - 1}((j-1)h)$ and forcing term (see \eqref{ERf-def}) given by 
\[
\tilde f(t) \coloneqq f(t) + \frac{\rho}{h}\partial_t \eta^{(h)}_{j - 1}(t-h).
\] 
Note that the result of each such step yields well posed data for the next step by the energy estimate \eqref{EI-par}.

We then define
\begin{equation}
\label{etah-def}
\eta^{(h)}(t) \coloneqq \eta^{(h)}_j(t) 
\end{equation}
for all $t \in [(j-1)h, jh]$.
\newline
\emph{Step 3:} Recall that each of the $\eta^{(h)}_j$ satisfies the energy inequality 
\eqref{EI-par} with the respective initial time $(j - 1)h$ and corresponding initial data. 
Fix $t \in [(j - 1)h, jh]$. Adding together the energy estimate for $\eta^{(h)}_j$ as well as all those for the previous steps with respect to the final time of the corresponding interval, we end up with
\begin{align} 
E^{(h)}(\eta^{(h)}(t)) & + \int_0^t2 R^{(h)}(\eta^{(h)}(s), \partial_t \eta^{(h)}(s))\,ds + \frac{\rho}{2h} \int_{t - (j-1)h}^t \|\partial_t \eta^{(h)}(s)\|_{L^2}^2\,ds \notag \\ 
& \le E^{(h)}(\eta_0^{(h)}) + \frac{\rho}{2}\|\eta^*\|^2_{L^2} + \int_0^t \langle f(s), \partial_t \eta^{(h)}(s) \rangle_{L^2} \,ds \notag \\
& \le E^{(h)}(\eta_0^{(h)}) + \frac{\rho}{2}\|\eta^*\|^2_{L^2} + \frac{1}{2\delta}\|f\|_{L^2(L^2)}^2 + \frac{\delta}{2} \int_0^t \|\partial_t \eta^{(h)}(s)\|_{L^2}^2\,ds \notag \\
& \le C + \frac{\delta}{2} \int_0^T \|\partial_t \eta^{(h)}(s)\|_{L^2}^2\,ds, \label{tdEE}
\end{align}
where $C$ is a constant that depends only on $\eta_0$, $\eta^*$, $\rho$, $T$, and $f$. In view of \eqref{E1} and recalling that $R^{(h)}$ is nonnegative, the previous inequality implies that 
\begin{equation}
\label{bdE}
\frac{\rho}{2h} \int_{t - (j-1)h}^t \|\partial_t \eta^{(h)}(s)\|_{L^2}^2\,ds \le C + \frac{\delta}{2} \int_0^T \|\partial_t \eta^{(h)}(s)\|_{L^2}^2\,ds.
\end{equation}
If we now let $t = T$, by multiplying both side of \eqref{bdE} by $h$ and summing over $j$ we obtain
\[
\frac{\rho}{2} \int_0^T \|\partial_t \eta^{(h)}(s)\|_{L^2}^2\,ds \le Lh\left(C + \frac{\delta}{2} \int_0^T \|\partial_t \eta^{(h)}(s)\|_{L^2}^2\,ds\right).
\]
Recalling that $Lh = T$ and setting $\delta \coloneqq T/2$ gives the bound
\begin{equation}
\label{dteta-bound}
\int_0^T \|\partial_t \eta^{(h)}(s)\|_{L^2}^2\,ds \le C,
\end{equation}
where $C$ is a constant independent of $h$. Furthermore, as an immediate consequence of \eqref{tdEE} and \eqref{dteta-bound}, we obtain
\begin{equation}
\label{unifEbound}
\begin{aligned}
\sup_{t \in [0,T]} \left\{E^{(h)}(\eta^{(h)}(t)) + \frac{\rho}{2h} \int_{t - h}^t \|\partial_t \eta^{(h)}(s)\|_{L^2}^2\,ds  \right\} & \le C, \\
\int_0^T R^{(h)}(\eta^{(h)}(s), \partial_t \eta^{(h)}(s))\,ds & \le C.
\end{aligned}
\end{equation}
In view of \eqref{ERf-def}, \eqref{unifEbound} implies that
\begin{equation}
\label{gradu'}
\int_0^T\|\nabla \partial_t \eta^{(h)}(s)\|^2_{L^2} + h \|\nabla^{k_0} \partial_t \eta^{(h)}(s)\|^2_{L^2}\,ds \le C.
\end{equation}
From \eqref{unifEbound} we see that $\{\eta^{(h)}\}_{h}$ is bounded in $L^{\infty}((0, T); W^{2, p}(Q; \RR^n))$. Furthermore, \eqref{unifEbound} and \eqref{gradu'} imply that $\partial_t \eta^{(h)}$ is uniformly bounded in $L^2((0, T); W^{1, 2}(Q; \RR^n))$, from which it follows that $\eta^{(h)}$ is bounded in $W^{1,2}((0, T); W^{1,2}(Q; \RR^n))$. In particular, we obtain that (eventually extracting a subsequence, which we do not relabel) 
\begin{equation}
\label{etah-to-eta-1}
\arraycolsep=1.4pt\def\arraystretch{1.6}
\begin{array}{rll}
\eta^{(h)} & \overset{\ast}{\rightharpoonup} \eta & \text{ in } L^{\infty}((0, T); W^{2, p}(Q; \RR^n)), \\
\eta^{(h)} & \rightharpoonup \eta & \text{ in } W^{1,2}((0, T); W^{1,2}(Q; \RR^n)),
\end{array}
\end{equation}
to some function $\eta \in L^{\infty}((0, T); W^{2, p}(Q; \RR^n)) \cap W^{1,2}((0, T); W^{1,2}(Q; \RR^n))$. An application of the classical Aubin--Lions lemma can be used to show that, up to the extraction of a further subsequence, 
\begin{equation}
\label{etah-to-eta-2}
\eta^{(h)} \to \eta \text{ in } C^0([0, T]; C^{1, \alpha}(Q; \RR^n)),
\end{equation}
where $\alpha$ is any number in $(0, 1 - n/p)$. In particular, we have that $\eta \in L^{\infty}((0, T); \E)$. Note also that with these convergences and the lower-semicontinuity of its left hand side, the first inequality in \eqref{tdEE} converges to the desired energy inequality.
\newline
\emph{Step 4:} Recall that each of the $\eta^{(h)}_j$ solves a variational inequality on the interval $[(j-1)h, jh]$. Rewriting these inequalities in terms of $\eta^{(h)}$ (see \eqref{etah-def}) and summing over $j$ yields that
\begin{multline}
\label{var-ineq-T}
\int_0^T \left[ DE^{(h)}(\eta^{(h)}(t)) + D_2R^{(h)}(\eta^{(h)}(t), \partial_t \eta^{(h)}(t))\right] \langle \varphi(t) \rangle\,dt \\ + \int_0^T \frac{\rho}{h} \langle \partial_t \eta^{(h)}(t) - \partial_t \eta^{(h)}(t - h), \varphi(t) \rangle_{L^2} \,dt \ge \int_0^T \langle f(t), \varphi(t) \rangle_{L^2}\,dt
\end{multline}
holds for all $\varphi \in C([0, T]; \Adm{\eta^{(h)}}{})$, as well as
\begin{multline}
\label{var-eq-T}
\int_0^T \left[ DE^{(h)}(\eta^{(h)}(t)) + D_2R^{(h)}(\eta^{(h)}(t), \partial_t \eta^{(h)}(t))\right] \langle \varphi(t) \rangle\,dt \\ + \int_0^T \frac{\rho}{h} \langle \partial_t \eta^{(h)}(t) - \partial_t \eta^{(h)}(t - h), \varphi(t) \rangle_{L^2} \,dt = \int_0^T \langle f(t), \varphi(t) \rangle_{L^2} + \langle \sigma^{(h)}(t),\varphi(t) \rangle\,dt
\end{multline}
for all $\varphi \in C([0, T],W^{k_0,2}_\Gamma(Q;\RR^n))$. The goal of this step is to obtain the technical intermediate results needed to pass to the limit as $h \to 0^+$ (or, alternatively, as $L \to \infty$) in \eqref{var-ineq-T} and \eqref{var-eq-T}. As observed in \cite{benesovaVariationalApproachHyperbolic2020}, the main difficulty is in passing to the limit with $DE(\eta^{(h)})$. Indeed, by assumption $DE$ is only continuous with respect to the strong topology of $W^{2, p}$. To improve the convergence of $\eta^{(h)}$ from what is provided by \eqref{etah-to-eta-1} and \eqref{etah-to-eta-2}, we rely on the Minty property of $DE$ (that is, \eqref{E6}). A key step in this direction is to obtain a uniform estimate in $L^2((0, T) ; L^2(Q; \RR^n))$ for the difference quotients that approximate the inertial term. To this end, we let 
\begin{equation}
\label{bh-def}
b^{(h)} \coloneqq \fint_t^{t + h}\partial_t \eta^{(h)}(s)\,ds = \frac{\eta^{(h)}(t + h) - \eta^{(h)}(t)}{h}.
\end{equation}
Then, \eqref{unifEbound} and \eqref{gradu'} imply that $b^{(h)}$ is bounded in $L^2((0, T); W^{1, 2}(Q; \RR^n))$. Recall that if $\varphi \in C^{\infty}_c((0, T) \times Q; \RR^n)$ then \eqref{var-ineq-T} holds as an equality. Then, reasoning as in Lemma 3.7 in \cite{benesovaVariationalApproachHyperbolic2020}, we obtain that $\partial_t b^{(h)}$ is bounded in $L^2((0, T); W^{-k_0, 2}(Q; \RR^n))$. Indeed, we have that
\begin{align}
\int_0^T \frac{\rho}{h} \langle \partial_t \eta^{(h)}(t) - \partial_t \eta^{(h)}(t - h), \varphi(t) \rangle_{L^2} \,dt & \le \int_0^T \biggr(\|DE(\eta^{(h)}(t))\|_{(W^{2, p})^*} + h^{a_0}\|\nabla^{k_0}\eta^{(h)}(t)\|_{L^2} \notag \\
& \qquad + \|D_2R(\eta^{(h)}(t), \partial_t \eta^{(h)}(t))\|_{(W^{1, 2})^*} \notag \\
& \qquad + h \|\nabla^{k_0} \partial_t \eta^{\eta}(t)\|_{L^2} + \|f(t)\|_{L^\infty} \biggr) \|\varphi(t)\|_{W^{k_0, 2}}\,dt. \label{dual-bound}
\end{align}
In view of \eqref{unifEbound}, \eqref{gradu'}, \eqref{E5}, and \eqref{R4}, it follows from \eqref{dual-bound} that
\begin{equation}\label{dteta-negative-estimate}
\int_0^T \rho \left\|\frac{\partial_t \eta^{(h)}(t) - \partial_t \eta^{(h)}(t - h)}{h}\right \|_{W^{-k_0, 2}}^2\,dt \le C,
\end{equation}
for some constant $C$ independent of $h$. Consequently, an application of the Aubin--Lions lemma shows that $b^{(h)}$ is precompact in $L^2((0, T); L^2(Q; \RR^n))$. As one can readily check (see in particular Lemma 3.8 in \cite{benesovaVariationalApproachHyperbolic2020}), up to the extraction of a subsequence (which we do not relabel), we have that $b^{(h)} \to \partial_t \eta$ in $L^2((0, T); L^2(Q; \RR^n))$.
 
Regarding the contact force, we can apply \Cref{par-cforce-cor} with $E^{(h)}$, $R^{(h)}$, $f$ and $\zeta = \partial_t \eta^{(h)}(\cdot-h)$, and use  $\xi_\Gamma\tilde n_{\eta^{(h)}}$ as a test function in \eqref{weak-cor}, where $\xi_\Gamma \in C_\Gamma^\infty(Q;[0,1])$ is given as in the proof of \Cref{existence-CF-quasi}. Note that this is allowed, as it follows from \Cref{smooth-almost-normal} that $\tilde n_{\eta^{(h)}}\in L^\infty((0,T); C^{k_0}(\overline Q;\RR^n))$. Using this together with the fact that $\|\sigma^{(h)}(t)\|_{M(\partial Q;\RR^n)}\leq 4 \langle \sigma^{(h)}(t), \xi_\Gamma \tilde n_{\eta^{(h)}}(t) \rangle$, we obtain that for $\mathcal{L}^1$-a.e.\@ $t\in(0,T)$
\begin{align*}
\frac{1}{4}\|\sigma^{(h)}(t)\|_{M(\partial Q;\RR^n)} & \leq \left[DE^{(h)}(\eta^{(h)}(t)) + D_2R^{(h)}(\eta^{(h)}(t), \partial_t \eta^{(h)}(t))\right] \langle \xi_\Gamma \tilde n_{\eta^{(h)}}(t)\rangle \\ & + \rho \left \langle \frac{\partial_t \eta^{(h)}(t)-\partial_t\eta^{(h)}(t-h)}{h}, \xi_\Gamma \tilde n_{\eta^{(h)}}(t) \right \rangle_{L^2} - \langle f(t), \xi_\Gamma \tilde n_{\eta^{(h)}}(t) \rangle_{L^2}.
\end{align*}
Reasoning as in the proof of \Cref{existence-CF-quasi}, we multiply the previous inequality by $\psi \in L^\infty((0,T); \RR^+)$ and integrate the resulting expression over $(0,T)$ to obtain 
\begin{align}
\frac{1}{4}\int_0^T & \|\sigma^{(h)}(t)\|_{M(\partial Q;\RR^n)}\psi(t) dt \leq \int_0^T \biggr\{\left[DE(\eta^{(h)}(t)) + D_2R(\eta^{(h)}(t), \partial_t \eta^{(h)}(t))\right] \langle \xi_\Gamma \tilde n_{\eta^{(h)}}(t)\rangle \notag  \\ 
& + 2h^{a_0}\left\langle \nabla^{k_0}\eta^{(h)}(t), \nabla^{k_0} (\xi_\Gamma \tilde n_{\eta^{(h)}}(t)) \right\rangle_{L^2}  +2 h \left\langle \nabla^{k_0} \partial_t\eta^{(h)}(t), \nabla^{k_0} (\xi_\Gamma \tilde n_{\eta^{(h)}}(t)) \right\rangle_{L^2} \notag \\ 
& + \rho \left \langle \frac{\partial_t \eta^{(h)}(t)-\partial_t\eta^{(h)}(t-h)}{h}, \xi_\Gamma \tilde n_{\eta^{(h)}}(t)\right \rangle_{L^2} - \langle f(t), \xi_\Gamma \tilde n_{\eta^{(h)}}(t) \rangle_{L^2}\biggr\}\psi(t)\,dt. \label{cf-L1bound}
\end{align}
We now estimate all the terms on the right hand side. Recall that $\norm[L^\infty(C^{k_0})]{\xi_\Gamma \tilde n_{\eta^{(h)}}}\leq C$ by \Cref{smooth-almost-normal}. By \eqref{unifEbound} and \eqref{E5} we see that
\begin{multline}
\label{cf-L1bound-DE}
\int_0^T DE(\eta^{(h)}(t)) \langle \xi_\Gamma \tilde n_{\eta^{(h)}}(t)\rangle \psi(t)\,dt \\ \leq \norm[L^\infty((W^{2,p})^*)]{DE(\eta^{(h)})}\norm[L^2(W^{2,p})]{\xi_\Gamma \tilde n_{\eta^{(h)}}} \norm[L^2]{\psi} \leq C \norm[L^2]{\psi};
\end{multline}
similarly, by \eqref{gradu'} and \eqref{eq:D2Rgrowthbounds} it follows that
\begin{multline}
\label{cf-L1bound-D2R}
\int_0^T D_2R(\eta^{(h)}(t), \partial_t \eta^{(h)}(t)) \langle \xi_\Gamma \tilde n_{\eta^{(h)}}(t)\rangle \psi(t)\,dt \\  \leq \norm[L^2((W^{1,2})^*)]{D_2R(\eta^{(h)},\partial_t\eta^{(h)})} \norm[L^\infty(W^{1,2})]{\xi_\Gamma \tilde n_{\eta^{(h)}}} \norm[L^2]{\psi} \leq C\norm[L^2]{\psi}.
\end{multline}
Notice also that by \eqref{unifEbound} we have that
\begin{multline}
\label{cf-L1bound-ha_0}
\int_0^T 2h^{a_0}\left\langle \nabla^{k_0}\eta^{(h)}(t), \nabla^{k_0} (\xi_\Gamma \tilde n_{\eta^{(h)}}(t)) \right\rangle_{L^2} \psi(t) \, dt \\ \leq 2h^{a_0} \norm[L^\infty(L^2)]{\nabla^{k_0}\eta^{(h)}(t)}\norm[L^2(W^{k_0,2})]{\xi_\Gamma \tilde n_{\eta^{(h)}}} \norm[L^2]{\psi} \leq h^{\frac{a_0}{2}}C \norm[L^2]{\psi},
\end{multline}
that an application of \eqref{gradu'} yields
\begin{multline}
\label{cf-L1bound-h}
\int_0^T  2h \left\langle \nabla^{k_0} \partial_t\eta^{(h)}(t), \nabla^{k_0} (\xi_\Gamma \tilde n_{\eta^{(h)}}(t)) \right\rangle_{L^2} \psi(t)\, dt \\ \leq 2h\norm[L^2(L^2)]{\nabla^{k_0}\partial_t\eta}\norm[L^2(W^{k_0,2})]{\xi_\Gamma \tilde n_{\eta^{(h)}} }\norm[L^2]{\psi} \leq \sqrt h C \norm[L^2]{\psi},
\end{multline}
and that the estimate
\begin{equation}
\label{cf-L1bound-f}
\int_0^T - \langle f(t), \xi_\Gamma \tilde n_{\eta^{(h)}}(t) \rangle_{L^2}\psi(t)\,dt \leq\norm[L^2(L^2)]{f}\norm[L^\infty(L^2)]{\xi_\Gamma \tilde n_{\eta^{(h)}}}\norm[L^2]{\psi}\leq C\norm[L^2]{\psi}
\end{equation}
follows as simple consequence of H\"older's inequality. 

Finally, for the inertial term, we first consider the case where $\psi \equiv 1$. Then we have by a discrete partial integration (i.e.\@, by a change of variables)
\begin{align*}
&\phantom{{}={}}\int_0^T \rho\left\langle \frac{\partial_t\eta^{(h)}(t)-\partial_t\eta^{(h)}(t-h)}{h}, \xi_\Gamma \tilde n_{\eta^{(h)}} (t)\right\rangle_{L^2} \psi(t) \,dt \\ 
&=\int_0^{T-h} \rho \left\langle \partial_t\eta^{(h)}(t), \xi_\Gamma \frac{\tilde n_{\eta^{(h)}} (t)  - \tilde{n}_{\eta^{(h)}}(t+h) }{h}\right\rangle_{L^2}\, dt \\ &+ \int_{T-h}^T \frac{\rho}{h} \left\langle \partial_t \eta^{(h)}(t), \xi_\Gamma \tilde n_{\eta^{(h)}} (t) \right \rangle_{L^2} dt -\int_{0}^h \frac{\rho}{h} \left\langle \partial_t \eta^{(h)}(t-h), \xi_\Gamma \tilde n_{\eta^{(h)}} (t) \right \rangle_{L^2} \,dt.
\end{align*}
For the second to last term on the right-hand side of the identity above, we estimate
\begin{align*}
 \int_{T-h}^T \frac{\rho}{h} \left\langle \partial_t \eta^{(h)}(t), \xi_\Gamma \tilde n_{\eta^{(h)}} (t)\right \rangle_{L^2}\, dt & \leq \frac{\rho}{h} \left( \int_{T-h}^T \norm[L^2]{\partial_t \eta^{(h)}(t)}^2\,dt \right)^{1/2} \left(\int_{T-h}^T\norm[L^2]{\xi_\Gamma \tilde n_{\eta^{(h)}} (t)}^2\,dt\right)^{1/2} \notag \\ 
& \leq \rho  \left(\fint_{T-h}^T \norm[L^2]{\partial_t \eta^{(h)}(t)}^2\,dt \right)^{1/2} \norm[L^{\infty}(L^2)]{\xi_\Gamma \tilde n_{\eta^{(h)}}} \le C
\end{align*}
and a similar, easier estimate holds for the first, as $\partial_t \eta^{(h)}(t-h)$ for $t<h$ is in fact already given by the initial data. This leaves us with
\begin{align*}
\int_0^{T-h} \rho \left\langle \partial_t\eta^{(h)}(t), \xi_\Gamma \frac{\tilde n_{\eta^{(h)}} (t)  - \tilde{n}_{\eta^{(h)}}(t+h) }{h}\right\rangle_{L^2} dt
\le C \rho \norm[L^2(L^2)]{\partial_t \eta^{(h)}} \norm[L^2(L^2)]{\partial_t \tilde n_{\eta^{(h)}}} \le C.
\end{align*}

Together with \eqref{cf-L1bound}--\eqref{cf-L1bound-f} this immediately shows that
\[
\int_0^T \|\sigma^{(h)}(t)\|_{M(\partial Q;\RR^n)} dt \leq C,
\]
thus proving that $\sigma^{(h)}$ is uniformly bounded in $L^1([0,T];M(\partial Q;\RR^n))$.

Finally we want to prove the absence of concentrations. For this, fix $t_0 \in [0,T)$, $\delta > 0$ and repeat the previous estimate with $\psi$ as the characteristic function of the small time interval $[t_0,t_0+\delta]$. As the normal direction cannot change too abruptly, if $\delta$ is small enough, we can choose $\tilde{n}_{\eta^{(h)}}$ to be piecewise constant in time. With this, as parts of the interval now cancels, we can estimate the inertial term using
 \begin{align}
  &\phantom{{}={}}\int_{t_0}^{t_0+\delta} \rho \left\langle \frac{\partial_t \eta^{(h)}(t)-\partial_t \eta^{(h)}(t-h)}{h}, \xi_{\Gamma} \tilde{n}_\eta^{(h)}(t) \right\rangle_{L^2}\,dt \notag \\
  &=\fint_{t_0+\delta-h}^{t_0+\delta} \rho \left\langle \partial_t \eta^{(h)}(t), \xi_{\Gamma} \tilde{n}_\eta^{(h)}(t) \right\rangle_{L^2} \, dt - \fint_{t_0-h}^{t_0} \rho \left\langle \partial_t \eta^{(h)}(t), \xi_{\Gamma} \tilde{n}_\eta^{(h)}(t + h) \right\rangle_{L^2}\, dt \notag \\
  &\leq \fint_{t_0+\delta-h}^{t_0+\delta} \rho \norm[L^2]{\partial_t \eta^{(h)}(t)}  \norm[L^2]{\xi_{\Gamma} \tilde{n}_\eta^{(h)}(t)}\,dt +  \fint_{t_0-h}^{t_0} \rho \norm[L^2]{\partial_t \eta^{(h)}(t)}  \norm[L^2]{\xi_{\Gamma} \tilde{n}_\eta^{(h)}(t + h)}\, dt \notag \\
  & \leq C\left(\sqrt{ \fint_{t_0+\delta-h}^{t_0+\delta} \norm[L^2]{\partial_t \eta^{(h)}(t)}^2\,dt} + \sqrt{ \fint_{t_0 - h}^{t_0} \norm[L^2]{\partial_t \eta^{(h)}(t)}^2\,dt} \right) \sup_{t \in [0, T]} \sqrt{|\supp \tilde{n}_{\eta^{(h)}}(t)|} \notag \\
  & \leq C \sup_{t \in [0, T]} \sqrt{|\supp \tilde{n}_{\eta^{(h)}}(t)|}, \label{cf-no-conc}
 \end{align}
 where in the last step we have used \eqref{unifEbound}.

 From the construction of $\tilde{n}_{\eta^{(h)}}$ (see \Cref{smooth-almost-normal}), it is clear that the support can be chosen arbitrarily small, albeit at the cost of increasing the norms of its derivatives. For any $\varepsilon > 0$, we can thus choose $\tilde{n}_{\eta^{(h)}}$ in such a way that the last term in \eqref{cf-no-conc} is bounded by $\varepsilon/2$. As all the other terms are estimated in terms of $\norm[L^2]{\psi}$ (see \eqref{cf-L1bound-DE}--\eqref{cf-L1bound-f}) with constant depending only on energy bounds and the choice of support of $\tilde{n}_{\eta^{(h)}}$, we can now choose $\delta$ small enough so that the contributions from all these remaining terms also sum up to at most $\varepsilon/2$. Then
 \begin{align}\label{eq:noConcentrations}
\frac{1}{4}\int_{t_0}^{t_0+\delta} \|\sigma^{(h)}(t)\|_{M(\partial Q;\RR^n)}\psi(t) dt < \varepsilon
 \end{align}
 independently of $t_0$ and $h$, which proves that $\sigma^{(h)} \in L^1((0, T); M(\partial Q;\RR^n))$ in fact cannot develop concentrations in time.
\newline
\emph{Step 5:} With this at hand, we proceed to prove that $\eta^{(h)}(t) \to \eta(t)$ in $W^{2, p}(K; \RR^n)$ for $\mathcal{L}^1$-a.e.\@ $t \in (0, T)$, where $K \subset \overline{Q}$ is such that $\dist(K, \Gamma) > 0$. Reasoning as in the first step, let $\tilde{\eta}^{(h)} \coloneqq (\eta \chi_{(0, T) \times Q}) \ast \xi_{h}$, where $\xi_h$ is defined as in \eqref{moll-def}. 
Let $\psi \in C^{\infty}_{\Gamma}(Q; [0, 1])$ be given. Straightforward computations show that
\begin{align*}
\|\tilde{\eta}^{(h)} \psi \|_{W^{k_0, 2}} & \le C h^{b_1(2 - k_0)}\|\eta\|_{W^{2, p}}, \\
\|\partial_t \tilde{\eta}^{(h)} \psi \|_{W^{k_0, 2}} & \le C h^{b_1(1 - k_0)}\|\partial_t \eta\|_{W^{1, 2}},
\end{align*}
where $C$ is a constant that does not depend on $h$. Let 
\[
\varphi^{(h)} \coloneqq (\eta^{(h)} - \tilde{\eta}^{(h)})\psi.
\]
Then, by \eqref{E6} we have that
\begin{align}
0 & \le \limsup_{h \to 0} \int_0^T \left[ DE(\eta^{(h)}(t)) - DE(\eta(t))\right]\!\langle (\eta^{(h)}(t) - \eta(t))\psi(t) \rangle\,dt \notag \\ 
& = \limsup_{h \to 0} \int_0^T DE(\eta^{(h)}(t))\langle  \varphi^{(h)}(t) \rangle\,dt \notag \\ 
& =  \limsup_{h \to 0} \int_0^T  DE^{(h)}(\eta^{(h)}(t))\langle  \varphi^{(h)}(t) \rangle\,dt - 2h^{a_0} \int_0^T \langle \nabla^{k_0} \eta^{(h)}(t), \nabla^{k_0} \varphi^{(h)}(t) \rangle_{L^2}\,dt, \label{minty-1}
\end{align}
where the first equality is obtained by using that $\tilde{\eta}^{(h)} \to \eta$ in $L^2((0, T); W^{2, p}(Q; \RR^n))$ together with the fact that $DE(\eta^{(h)})$ is uniformly bounded in $L^2((0, T); (W^{2, p}(Q; \RR^n))^*)$ as a consequence of \eqref{E5}. Observe that
\begin{align*}
\langle \nabla^{k_0} \eta^{(h)}(t), \nabla^{k_0} (\eta^{(h)}(t)\psi(t))\rangle_{L^2} & = \langle \nabla^{k_0} \eta^{(h)}(t), (\nabla^{k_0} \eta^{(h)}(t))\psi(t) + L^{(h)}(t)\rangle_{L^2} \\ 
& \ge \langle \nabla^{k_0} \eta^{(h)}(t), L^{(h)}(t)\rangle_{L^2}, 
\end{align*}
where $L^{(h)}$ only involves derivatives of $\eta^{(h)}$ of order less than $k_0$. Using the fact that $\eta^{(h)}$ is bounded in $L^{\infty}((0, T); W^{2, p}(Q; \RR^n))$ (see \eqref{unifEbound}) and that $h^{a_0}\nabla^{k_0}\eta^{(h)}$ is bounded in $L^{\infty}((0, T); L^2(Q; \RR^n))$, a standard interpolation inequality shows that $h^{a_0}\|L^{(h)}\|_{L^{\infty}L^2} \to 0$. Combining this with \eqref{minty-1} we obtain that
\begin{align}
0 & \le \limsup_{h \to 0} \int_0^T DE^{(h)}(\eta^{(h)}(t))  \langle \varphi^{(h)}(t) \rangle\,dt + C h^{a_0} \int_0^T \|\nabla^{k_0} \eta^{(h)}(t)\|_{L^2}\|\tilde{\eta}^{(h)}(t)\psi(t)\|_{W^{k_0, 2}} \,dt \notag \\ 
& \le \limsup_{h \to 0} \int_0^T DE^{(h)}(\eta^{(h)}(t))\langle  \varphi^{(h)}(t) \rangle\,dt + C h^{\frac{a_0}{2} - b_1(2 - k_0)} \notag \\
& = \limsup_{h \to 0} \int_0^T  DE^{(h)}(\eta^{(h)}(t))\langle \varphi^{(h)}(t) \rangle\,dt, \label{minty-2}
\end{align} 
where the last step follows by choosing $b_1$ sufficiently small. Observe that this condition is strictly stronger than what is required on $b_1$ from the first step. We can then use \eqref{var-eq-T} to rewrite the last term in \eqref{minty-2} as 
\begin{align*}
\int_0^T  DE^{(h)}(\eta^{(h)}(t))\langle \varphi^{(h)}(t) \rangle\,dt = & - \int_0^T  D_2R^{(h)}(\eta^{(h)}(t), \partial_t \eta^{(h)}(t))\langle \varphi^{(h)}(t) \rangle\,dt \\ & - \int_0^T \frac{\rho}{h} \langle \partial_t \eta^{(h)}(t) - \partial_t \eta^{(h)}(t - h), \varphi^{(h)}(t) \rangle_{L^2} \,dt \\ & +  \int_0^T \langle f(t), \varphi^{(h)}(t) \rangle_{L^2}\,dt + \int_0^T \langle \sigma^{(h)},\varphi^{(h)}(t) \rangle dt.
\end{align*}
Notice that 
\begin{multline}
 D_2R^{(h)}(\eta^{(h)}(t), \partial_t \eta^{(h)}(t))\langle \varphi^{(h)}(t) \rangle =  D_2R(\eta^{(h)}(t), \partial_t \eta^{(h)}(t))\langle \varphi^{(h)}(t) \rangle \\ + 2 h \langle \nabla^{k_0} \partial_t \eta^{(h)}(t), \nabla^{k_0}\varphi^{(h)}(t) \rangle_{L^2}.
\end{multline}
In view of \eqref{R4} and \eqref{etah-to-eta-1}, we have that $D_2R(\eta^{(h)}(t), \partial_t \eta^{(h)}) \to D_2R(\eta(t), \partial_t \eta(t))$ in $W^{1, 2}(Q; \RR^n)^*$ for a.e.\@ $t \in (0, T)$. Since $\varphi^{(h)}(t) \to 0$ in $W^{1, 2}(Q; \RR^n)$ for a.e.\@ $t \in (0, T)$ we have get that 
\[
 D_2R(\eta^{(h)}(t), \partial_t \eta^{(h)}(t))\langle \varphi^{(h)}(t) \rangle \to 0
\]
as $h \to 0$. Moreover, observe that 
\begin{equation}
\label{regto0}
h |\langle \nabla^{k_0} \partial_t \eta^{(h)}(t), \nabla^{k_0}\varphi^{(h)}(t) \rangle_{L^2}| \le h^{\frac{1}{2} - \frac{a_0}{2}}\|h^{\frac{1}{2}} \nabla^{k_0} \partial_t \eta^{(h)}(t)\|_{L^2} \|h^{\frac{a_0}{2}} \nabla^{k_0} \varphi^{(h)}(t)\|_{L^2}.
\end{equation}
Using \eqref{unifEbound} and \eqref{gradu'}, we obtain that the right-hand side of \eqref{regto0} vanishes, provided that $a_0 < 1$ (and that $b_1$ is chosen accordingly). In turn, Lebesgue's dominated convergence theorem implies that
\[
\int_0^T  D_2R^{(h)}(\eta^{(h)}(t), \partial_t \eta^{(h)}(t))\langle \varphi^{(h)}(t) \rangle\,dt \to 0
\]
as $h \to 0$. Next, we claim that
\[
\lim_{h \to 0 }\int_0^T \frac{1}{h} \langle \partial_t \eta^{(h)}(t) - \partial_t \eta^{(h)}(t - h), \varphi^{(h)}(t) \rangle_{L^2} \,dt = 0.
\]
To see this, we begin by rewriting the previous integral as
\begin{multline}
\label{dibp}
\int_0^T \frac{1}{h} \langle \partial_t \eta^{(h)}(t) - \partial_t \eta^{(h)}(t - h), \varphi^{(h)}(t) \rangle_{L^2} \,dt = \int_0^T \left\langle \partial_t \eta^{(h)}(t), (\tilde{\eta}^{(h)}(t) - \eta^{(h)}(t)) \frac{\psi(t + h) - \psi(t)}{h} \right\rangle_{L^2}\,dt \\ + \int_0^T \left \langle \partial_t \eta^{(h)}(t), \left(\frac{\tilde{\eta}^{(h)}(t + h) - \tilde{\eta}^{(h)}(t)}{h} - \frac{\eta^{(h)}(t + h) - \eta^{(h)}(t)}{h} \right)\psi(t + h) \right \rangle_{L^2}\,dt.
\end{multline}
Here, we notice that the first term on the right hand side of \eqref{dibp} converges to zero by recalling that $\partial_t \eta^{(h)} \rightharpoonup \partial_t \eta$ in $L^2((0, T); L^2(Q; \RR^n))$ (see \eqref{etah-to-eta-1}) and by also noticing that $\eta^{(h)} - \tilde{\eta}^{(h)} \to 0$ in $L^2((0, T); L^2(Q; \RR^n))$ by the properties of the mollification. Similarly, the second term vanishes as $h \to 0$ since we have previously shown that $b^{(h)}$ (defined in \eqref{bh-def}) admits a converging subsequence in $L^2((0, T); L^2(Q; \RR^n))$ and the analogous convergence continues to hold when $\eta^{(h)}$ is replaced with its mollification $\tilde{\eta}^{(h)}$. For the term involving the contact force, we note that Step 4 and
\eqref{etah-to-eta-2} together imply
\begin{align*}
 \lim_{h \to 0} \abs{\int_0^T \langle \sigma^{(h)}(t), \varphi^{(h)}(t) \rangle_{L^2}\,dt} \leq \lim_{h \to 0} \int_0^T \norm[M(\partial Q;\RR^n)]{\sigma^{(h)}(t)} \norm[C^0(\partial Q; \RR^n)]{\varphi^{(h)}(t)} dt = 0.
\end{align*}
Since clearly we also have that
\[
\lim_{h \to 0} \int_0^T \langle f(t), \varphi^{(h)}(t) \rangle_{L^2}\,dt = 0,
\] 
we conclude that 
\[
\limsup_{h \to 0} \int_0^T  DE^{(h)}(\eta^{(h)}(t))\langle \varphi^{(h)}(t) \rangle\,dt = 0.
\]
In view of \eqref{minty-2} and \eqref{E6}, this implies the desired result.
\newline
\emph{Step 6:} Let $\varphi \in C([0, T]; \Adm{\eta}{}) \cap C^1_c([0, T); L^2(Q; \RR^n))$ be given and let $\{\varphi_h\}_h$ be given as in \Cref{prop:TErecoveryGlobal}. Then \eqref{var-ineq-T} holds for $\varphi_h$ and we can pass to the limit as $h \to 0$. To be precise, we have that 
\[
\int_0^T  DE^{(h)}(\eta^{(h)}(t))\langle \varphi_h(t) \rangle\,dt \to \int_0^T  DE(\eta(t))\langle \varphi(t) \rangle\,dt,
\]
\[
\int_0^T  D_2R^{(h)}(\eta^{(h)}(t), \partial_t \eta^{(h)}(t))\langle \varphi_h(t) \rangle\,dt \to \int_0^T  D_2R(\eta(t), \partial_t \eta(t))\langle \varphi(t) \rangle\,dt,
\]
and clearly also 
\[
\int_0^T \langle f(t), \varphi_h(t) \rangle_{L^2}\,dt \to \int_0^T \langle f(t), \varphi(t) \rangle_{L^2}\,dt.
\]
Finally, recalling that $\partial_t \eta^{(h)}$ is to be understood as $\eta^*$ for $t < 0$ and assuming without loss of generality that $\varphi_h$ is extended constantly for negative times and that it vanishes for $t \ge T-h$, we observe that
\begin{align*}
&\phantom{{}={}}\int_0^T \frac{\rho}{h}\langle \partial_t \eta^{(h)}(t) - \partial_t \eta^{(h)}(t - h), \varphi_h(t) \rangle_{L^2}\,dt \\
& = - \int_0^{T-h} \frac{\rho}{h} \langle \partial_t \eta^{(h)}(t), \varphi_h(t + h) - \varphi_h(t) \rangle_{L^2}\,dt - \int_{-h}^0 \frac{\rho}{h} \langle \eta^*, \varphi_h(t + h)\rangle_{L^2}\,dt\\
& = -\int_0^{T-h} \frac{\rho}{h} \langle  \partial_t \eta^{(h)}(t), \int_t^{t+h} \partial_t \varphi_h(s) \,ds\rangle_{L^2}\, dt- \int_{-h}^0 \frac{\rho}{h} \langle \eta^*, \varphi_h(t + h)\rangle_{L^2}\,dt \\ 
& = -\int_0^T \frac{\rho}{h} \int_{t-h}^{t} \langle  \partial_t \eta^{(h)}(s),  \partial_t \varphi_h(t) \,ds\rangle_{L^2}\, dt - \int_{-h}^0 \frac{\rho}{h} \langle \eta^*, \varphi_h(t + h)\rangle_{L^2}\,dt\\
& = - \int_0^T \rho \langle b^{(h)}(t), \partial_t \varphi_h(t) \rangle_{L^2} \, dt - \int_{-h}^0 \frac{\rho}{h} \langle \eta^*, \varphi_h(t + h)\rangle_{L^2}\,dt\\
& \to - \int_0^T \rho \langle \partial_t \eta(t), \partial_t \varphi(t) \rangle_{L^2}\,dt - \rho \langle \eta^*, \varphi(0)\rangle_{L^2}.
\end{align*}
This shows that $\eta$ is a solution to the variational inequality \eqref{var-ineq-def}.

The last missing piece of the main result of this paper is to show that the previous procedure in fact yields existence of a solution to the full problem with a contact force. For this it is enough to show the convergence of contact forces. By the estimate on $\sigma^{(h)}$ in Step 4 and the convergence of $\eta^{(h)}$, \Cref{compactness-cforce} guarantees the existence of a subsequence (which we do not relabel) such that $\sigma^{(h)}\weakstarto \sigma$ in $M([0,T]\times \partial Q;\RR^n)$, where $\sigma \in M([0,T] \times \partial Q;\RR^n)$ is a contact force for $\eta$ and satisfies the action-reaction principle (see \Cref{def-contact-force}). Finally, in view of \eqref{eq:noConcentrations} we also obtain that $\sigma$ has no concentrations in time. Since the weak$^*$ convergence of measures is enough to pass to the limit in the term with $\sigma^{(h)}$ and since all the other terms in the weak formulation have been shown to convergence, this concludes the proof.
\end{proof}

\begin{rmk}[On the relation to \cite{kromerQuasistaticViscoelasticitySelfcontact2019}] 
\label{rem:KRopenProblem1}
In \cite{kromerQuasistaticViscoelasticitySelfcontact2019}, the authors studied the existence of solutions to the quasistatic problem \eqref{aux-Pb} via an implicit time discretization and arrived at results comparable with those obtained in \Cref{existence-CF-quasi}. We comment here on the two main differences between the results. 
\begin{itemize}
\item[$(i)$] The more important difference is that the methodology used in \cite{kromerQuasistaticViscoelasticitySelfcontact2019} does not allow one to recover the contact force as a measure. In turn, devising techniques that allow to avoid this loss of regularity for the contact force was posed an open problem in Remark 3.2 in \cite{kromerQuasistaticViscoelasticitySelfcontact2019}. To be precise, in their proof the authors derive all bounds on the contact force directly from the weak equation. As a result, the best regularity that one can achieve is that of the worst other term in the equation, and only allows to conclude that compactness holds in a negative Sobolev space. In contrast, we use a more detailed analysis of the discrete setting to provide explicit bounds in the space of measures, which in turn allows us to conclude existence of a proper limit measure.
 
\item[$(ii)$] The other difference is that throughout \Cref{Aux-sec} we consider energies and dissipation potentials for which the term of highest order is quadratic. In fact, this is only done in order to more easily obtain the precise energy estimate we need for proof of \Cref{main-thm-VI}. However, as shown in the proof of \Cref{main-thm-VI}, as long as the remaining highest order term still guarantees the compact embedding into $C^1$ and quasimonotonicity, we can use \eqref{E6} together with a standard Minty-type argument to get rid of the regularizing term. In fact, this provides another contrast to \cite{kromerQuasistaticViscoelasticitySelfcontact2019}, where this leads to an additional error term, as it is not possible to use the equation to both deal with the contact force and the estimate needed for the Minty-type argument at the same time.
\end{itemize} 
\end{rmk}

The existence of a solution with contact force for the quasistatic case is shown in the corollary below. We can thus claim to have solved the open problem formulated in Remark 3.2 in \cite{kromerQuasistaticViscoelasticitySelfcontact2019}.

\begin{cor}
\label{KR-recovery}
Let $E$ and $R$ be as in \eqref{E1}--\eqref{E6} and \eqref{R1}, \eqref{R2}, \eqref{R3q}, and \eqref{R4}, respectively, $T> 0$, and let $\eta_0 \in \E$ and $f \in L^2((0,T); L^2(Q; \RR^n))$ be given. Then the quasistatic problem \eqref{aux-Pb} admits a weak solution with a contact force, that is, there exist
\[
\eta \in L^\infty((0, T); \E) \cap W^{1,2}((0, T); W^{1,2}(Q; \RR^n)) 
\]
with $E(\eta) \in L^\infty(0, T)$ and $\eta(0) = \eta_0$ and a contact force $\sigma \in L_{w^*}^2((0, T); M(\partial Q;\RR^n))$ (see \Cref{def-contact-force}) such that
\[
\int_0^T \left[DE(\eta(t)) + D_2R(\eta(t), \partial_t \eta(t))\right] \langle \varphi(t)\rangle\,dt = \int_0^T \langle \sigma(t) ,\varphi(t)\rangle\,dt + \int_0^T \langle f(t), \varphi(t) \rangle_{L^2}\,dt
\]
holds for all $\varphi \in C([0,T]; W^{2,p}_{\Gamma}(Q;\RR^n))$.
\end{cor}
\begin{proof}
The proof is analogous to that of \Cref{main-thm-VI}, but simpler. To be precise, let $E^{(h)}, R^{(h)}$ be the regularized energy-dissipation pair considered in the proof of \Cref{main-thm-VI}, and let $\{\eta_0^{(h)}\}_h$ be given as in \eqref{eta-0h-def}. Reasoning as above, an application of \Cref{AuxExistence} yields the existence of a family of deformations $\{\eta^{(h)}\}_h$ that solves the variational inequality for the regularized problem, and such that 
\begin{equation}
E^{(h)}(\eta^{(h)}(t)) + 2 \int_0^t R^{(h)}(\eta^{(h)}(s), \partial_t \eta^{(h)}(s))\,ds \le E^{(h)}(\eta_0^{(h)}) + \int_0^t \langle f(s), \partial_t \eta^{(h)}(s)\rangle_{L^2} \,ds.
\end{equation}
We remark here that the parameter $h$ is only introduced in order to apply the theory developed in \Cref{Aux-sec} and that in the absence of the inertial term we no longer require to consider a time-delayed problem in subintervals of length $h$, but rather we directly study the problem on the entire interval $[0, T]$. Since 
\[
K_R\|b\|_{W^{1, 2}}^2 \le R(\eta, b) \le R^{(h)}(\eta, b),
\]
by Young's inequality we obtain that 
\[
E^{(h)}(\eta^{(h)}(t)) + \int_0^t R^{(h)}(\eta^{(h)}(s), \partial_t \eta^{(h)}(s))\,ds \le C, 
\]
where $C$ only depends on $E(\eta_0), f$, and $K_R$. The rest of the proof follows the same strategy as that of \Cref{main-thm-VI}. However, the argument presented in Steps 4 through 6 can be simplified by formally substituting $\rho = 0$, which in particular allows to ignore any difficulties involving the inertia-term and use $L^2$ in time bounds there. Thus, we omit the details.
\end{proof}

\begin{rmk}
In the preprint version of this article we claimed that also in the inertial case the contact force is square-integrable in time, i.e.\@ $\sigma \in L^2_{w^*}([0,T];M(\partial Q))$. However, while revising for publication, we realized that there was an incorrect estimate in the respective proof. While this does not affect our conclusions for the quasistatic regime (see \Cref{KR-recovery}), for the inertial problem we can only guarantee that $\sigma$ is a measure without any concentration in time. Nevertheless we still conjecture that our original bounds will end up being correct.

Looking at the proof of \Cref{main-thm-VI}, in particular at Step 4, we have good $L^2$-in-time bounds on almost all of the terms involved. This, together with the weak equation itself, means that the only way for $\sigma$ to have a non square-integrable part is if it is cancelled by a similar (but opposite) contribution coming from the inertial term. Additionally, this contribution will have to happen at the boundary and in normal direction, as this is where the contact force is supported and pointed.

It now stands to reason that such a concentration in $\partial_{tt} \eta$ can only happen precisely at the instant of collision, since once two parts of the solid are in touch their behavior in normal direction will be similar to that of a single solid. But since locally those instances are naturally isolated in time, the proven absence of concentrations implies that these should not give any meaningful contribution. We were however unable to prove this rigorously.
\end{rmk}

\section{Physical considerations} \label{sec:physConsiderations}

The aim of this section is to display the physical content of the purely mathematical discussion in the rest of this paper. The way to this is threefold. First we will illustrate by an example that our assumptions on energy and dissipation are reasonable. Then we will consider how our result is intricately related to momentum conservation, both globally as well as in a localized version. Finally we show that in contrast to e.g.\@ rigid body or thin film dynamics, no additional contact law is required in our bulk setting, as the contact set (see \Cref{def-contact-set}), and thus the region of possibly instantaneous momentum change, is of lower dimension and thus has no influence on the total momentum.

\subsection{An example energy-dissipation pair}

We have already remarked on the need for a second order material in order to have sufficient regularity of the normal vectors. Additionally, it is an indirect requirement to allow us a control on the determinant in general, which in turn is related to local injectivity as well as the existence of a Korn-type inequality. Nevertheless, while we only ever use the relatively generic assumptions detailed in \Cref{Modelling-sec}, we have an example energy-dissipation pair in mind which fulfils all of those and at the same time can be considered physical. For a full discussion, we refer in particular to \cite[Sec. 2.3]{benesovaVariationalApproachHyperbolic2020}.

A good choice of energy has to be frame invariant, i.e.\@ it should not change under rotations in the image. At the same time, it needs to penalize compression, but in such a way that there still is a well defined Fr\'echet-derivative at any deformation of finite energy. Finally, it should still be related to the simpler models studied in engineering. Perhaps the most simple candidate that unifies these requirements is given by
\begin{align*}
 E(\eta) \coloneqq \int_Q \mathcal{C} (\nabla \eta^T \nabla \eta - I) : (\nabla \eta^T \nabla \eta - I) + \frac{c_1}{(\det \nabla \eta)^a} + c_2\abs{\nabla^2 \eta}^p\, dx,
\end{align*}
where $\mathcal{C}$ is a fourth order tensor, $c_1,c_2$ are positive constants, and $a > \frac{pn}{p-n}$ is needed to guarantee injectivity (see \cite{MR2567249}). Here the first term corresponds to classical elasticity theory, while the others can be seen as small, frame invariant perturbations needed in order to deal with large and irregular deformations.

Similar to the energy, the example dissipation we have in mind needs to be frame invariant as well; as a result, it is actually better to consider the dissipation potential as a function depending on $\partial_t (\nabla \eta^T \nabla \eta)$ instead (see also \cite{antmanPhysicallyUnacceptableViscous1998} for a detailed discussion). This leads directly to the most simple example of such a dissipation satisfying our assumptions:

\begin{align*}
 R(\eta,\partial_t \eta) \coloneqq  \int_Q \abs{\partial_t \nabla \eta^T \nabla \eta + \nabla \eta^T \partial_t \nabla \eta}^2 dx = \int_Q \abs{\partial_t (\nabla \eta^T \nabla \eta)}^2 dx.
\end{align*}
Korn-type inequalities (see \eqref{R3}) for this choice of the dissipation potential have been show by Neff \cite{Neff02} and Pompe \cite{Pompe03}.

\subsection{Momentum conservation and collision forces} 
The total momentum in the Lagrangian representation is given by the integral of the momentum density $\int_Q \rho \partial_t \eta \,dx$. This quantity is conserved in the dynamical case. Globally we can note that we can test each of the three equations we obtain in the course of the proof of \Cref{main-thm-VI} (Euler--Lagrange, time-delayed and hyperbolic) with $\varphi \coloneqq e_i$ to obtain the respective versions: 

\begin{align*}
 \frac{\rho}{h}\int_Q \frac{\eta_k^i-\eta_{k-1}^i}{\tau} \,dx &= \int_{\partial Q} d \sigma^i_k  + \frac{\rho}{h}\int_Q \zeta^i_k \,dx \\
 \frac{\rho}{h}\int_Q \partial_t \eta^i(t) \,dx &= \int_{\partial Q} d \sigma^i  + \frac{\rho}{h}\int_Q \partial_t \eta^i(t-h) \,dx \\
 \frac{d}{dt}\rho \int_Q \partial_t \eta^i \,dx &= \int_{\partial Q} d \sigma^i 
\end{align*}
where $\sigma$ is the respective contact force and we use that physically reasonable energy and dissipation functionals only act on the spatial derivatives, so since $\nabla e_i = 0$ we have $ DE(\eta) \langle e_i\rangle = 0$ and so on.

For the second equation, we also note that it implies
\begin{align*}
 \rho\partial_t \fint_{t-h}^t \int_Q \partial_t \eta^i\, dx ds = \frac{\rho}{h} \int_Q\partial_t \eta^i(t)\, dx - \frac{\rho}{h} \int_Q \partial_t \eta^i(t-h) \,dx = \int_{\partial Q} d \sigma^i,
\end{align*}
which again shows that time averages are the suitable quantities to study when considering convergence of the time delayed equation.

Additionally, since the contact force we obtain satisfies the action-reaction principle, it is not hard to show that all its contributions from self-contacts have to cancel in the end, which leaves us with
\[
 \frac{d}{dt} \rho \int_Q \partial_t \eta \,dx = \int_{\eta^{-1}(\partial\Omega)} d \sigma.
\]

We can thus immediately conclude that the total momentum can only change if there is contact with the exterior boundary $\partial \Omega$ and it can only do so in normal direction.

Similarly we have a momentum density as conserved quantity: There is a symmetric matrix valued stress tensor, in the form of a distribution $A \in (W^{1,p}(Q;\RR^{n\times n}))^*$ fulfilling
\[
\nabla \cdot A = DE + DR.
\]
Now, testing the final equation (the same can also be done in case of the approximations, with similar results) with $\varphi e_i$ for some $\varphi \in C^\infty([0,T] \times Q)$ yields
\[ 
\int_Q \varphi(T)  \rho\partial_t \eta^i(T) \,dx - \int_0^T \int_Q\rho \partial_t \eta^i \partial_t \varphi \,dx dt + \int_0^T \langle \nabla \cdot A_i, \varphi\rangle = \int_0^T \int_{\partial Q} \varphi d\sigma^i \,dt + \int_Q \varphi(0) \rho \partial_t \eta^i(0) \,dx,
\] 
which is of course a weak formulation of
\[
\partial_t (\rho \partial_t \eta^i) = \nabla \cdot A_i + \sigma^i,
\]
that is, the physical conservation of momentum in continuum mechanics. 

\subsection{An example of ``true bouncing''}

It should be noted that we only show existence of \emph{a}~solution, and that at no point we claim that solutions must be unique. While this is primarily due to the non-linear nature of the elastic energy, the potential lack of uniqueness can be seen to have important implications in a context where contact is allowed, especially for the resulting rebound dynamics.

When considering the reduced example of a point particle or a rigid body colliding with a fixed obstacle, one typically specifies an additional contact law, usually in the form of a reflection of (a fraction of) the velocity across the contact plane. This, however, is not the approach that we  followed in this paper. Instead, we only prohibit entering the obstacle, which results in an obvious source of non-uniqueness. In particular, if we were to apply our method to this case, the expected result would be the rigid body getting stuck at or sliding along the obstacle, as these correspond to the solutions with the least possible change in velocity.

We claim that this unphysical behavior is not possible with our approach, since the elastic solids that we consider have full dimension. The main reason for this is that, while there has to be an instantaneous change of velocity at the point of contact, which is not specified by the equations, this change of velocity happens instantaneously only at the point (or possibly at the surface) of contact and only at the time of contact. As this is a set of lower dimension, its unspecified influence on the total momentum and on the kinetic energy is negligible. Only for times after contact, these change continuously via the contact force and conversion into other energy types respectively, all of which are accounted for by the equations.

We illustrate this with an example.
\begin{example}[Bouncing ball]
 Let $\Omega = \RR^+ \times \RR^{n-1}$ be the half plane and consider a ball of radius $r$, with uniform density $\rho$ as the elastic solid, i.e.\ $Q = B_r(0)\subset \RR^n$ and an elastic energy $E$ for which this is the only rest configuration, i.e.\ the critical points of $E$ are precisely the Euclidean transformations, which all minimize $E$ with zero energy. We assume that initially, the solid is in such a configuration away from the wall, i.e., $\eta_0(x) = x+ le_1$ with $l > r$ and it has uniform initial velocity $v_0 = -e_1$ pointing directly towards the wall, so that there will be a collision.
 
 Now consider the solid's center of mass at time $t$, which we will denote by $y(t) \coloneqq \fint_Q \eta(t) \,dx$. Per definition we know that it evolves along with the total momentum, i.e., 
 \[
 m\dot{y}(t) = p(t) \coloneqq \fint_Q \rho \partial_t \eta(t)\,dx.
 \]
 Of this we have derived in the previous subsection that it can only change through contact forces arising from collisions with the boundary $\partial \Omega$, which for our half space $\Omega$ means that $\dot{p}_1 \geq 0$ and $\dot{p}_i = 0$ for $i > 1$, as those forces only have one normal direction to act in. As the geometry implies that $y_1(t) > 0$ for all times and since the initial conditions imply $\dot{y}_1(0) = -1$, we know that contact will happen. In turn, the main question is to understand how the center of mass will evolve afterwards. In particular, our goal is to show that there is a time $T$ such that $\dot{y}_1(T) > 0$, which means that the ball will perpetually move away from the wall, i.e.\@ we have ``true bouncing''.
 
 To see this, we analyze the situation further. First we observe that, before contact, the ball cannot deform. Indeed, this is guaranteed by the equation, but also for energy reasons, as uniform velocity is already the minimizer of kinetic energy for given momentum, i.e.\@ there is no other way to obtain nonzero elastic energy in the energy balance. So at the time $t_0$ of first contact we have $y_1(t_0) = r$. Now, as the change of total momentum cannot have concentrations in time, and $\dot{y}_1(t_0) = -1$, there is a time $t_1> t_0$ for which $y_1(t_1) < r$. If we assume that there is no rebound, i.e.\@ $\dot{y}_1 \leq 0$, then this needs to hold true for all future times as well.
 
 But this is at odds with the tendency of the solid to evolve towards a relaxed configuration in the absence of forces. Specifically assume that there is no bouncing, i.e.\@ $y(t)$ stays bounded. As this implies $\dot{y}_1 \leq 0$, then there exists a limit $y_1(t) \searrow y_\infty < r$, and obviously $p_1(t) \nearrow 0$ as well. Combining the latter with a consequence of the energy estimate, that
 \begin{align*}
  \frac{\rho}{2} \norm[L^2]{\partial_t \eta(t)}^2 + \int_0^t c \norm[L^2]{\nabla \partial_t \eta}^2 dt \leq   \frac{\rho}{2}\norm[L^2]{\partial_t \eta(t)}^2 + \int_0^t R(\eta, \partial_t \eta)^2 dt
 \end{align*}
 is bounded for some $c>0$, we get that $\partial_t \eta(t) \to 0$ in $L^2$ exponentially by a Gronwall argument. This in turn implies the existence of a pointwise limit $\eta(t) \to \eta_\infty$ almost everywhere, and we can choose a sequence of times $\{t_k\}_k$ with $t_k \to \infty$ such that $\nabla \partial_t \eta(t_k)\to 0$ in $L^2$ since the its integral in time is bounded. Similarly, we know that the integral of the total contact force is equal to the change in momentum and thus remains bounded. Without loss of generality, we can then choose $t_k$ in such a way that this quantity vanishes as well. Finally, compactness gives us that $\eta(t_k) \rightharpoonup \eta_\infty$ in $W^{2,p}(Q;\RR^n)$.
 
 But then, we see from the equation that all terms vanish, except possibly for $DE(\eta_\infty)$, which in turn implies that also $DE(\eta_\infty) = 0$. Hence, we conclude that $\eta_\infty$ is a translation of the identity that maps into $\Omega$. Thus, we have shown that $\fint_Q \eta_\infty\, dx \geq r> y_\infty$, which is a contradiction.
 
\end{example}

\bibliographystyle{siam}
\bibliography{biblio}
\end{document}